\documentclass[10pt]{amsart}
\usepackage{a4wide,color}
\usepackage{amsmath,amssymb}
\usepackage[active]{srcltx}
\usepackage{ulem}

\allowdisplaybreaks

\newtheorem{theorem}{Theorem}
\newtheorem{lemma}[theorem]{Lemma}

\newtheorem{remark}[theorem]{Remark}
\newtheorem{corollary}[theorem]{Corollary}

\let \om=\omega

\begin{document}

\title[Boltzmann equation with uncertainties]{
Time-asymptotic behavior of the Boltzmann equation with random inputs in whole space and its stochastic Galerkin approximation}

\author{Shi Jin}
\address{Shi Jin, School of Mathematical Sciences, Institute of Natural
	Sciences, MOE-LSC,  Shanghai Jiao Tong University, Shanghai,
	China}
\email{shijin-m@sjtu.edu.cn}

\author{Qi Shao}
\address{Qi Shao, School of Mathematical Sciences, Shanghai Jiao Tong University, Shanghai,
China}
\email{1shaoqi3@163.com}

\author{Haitao Wang}
\address{Haitao Wang, School of Mathematical Sciences, Institute of Natural
Sciences, MOE-LSC, CMA-Shanghai, Shanghai Jiao Tong University, Shanghai,
China}
\email{haitallica@sjtu.edu.cn}

\date{\today }
%\thanks{2020 Mathematics Subject Classification: 35Q20; 82C40.}

\begin{abstract}
We consider the Boltzmann equation with random uncertainties arising from the initial data and collision kernel in the {\it whole space}, along with their stochastic Galerkin (SG) approximations. By employing Green's function method, we show that, the higher-order derivatives of the solution with respect to the random variable exhibit polynomial decay over time. These results are then applied to analyze the SG method for the SG system and to demonstrate the polynomial decay of the numerical error over time.
\end{abstract}

\keywords{Boltzmann equation, uncertainty quantification, gPC stochastic Galerkin.}
\maketitle

%\subjclass[2000]{35Q20; 82C40.}

%\paragraph{Keywords:}

%\paragraph{AMS classification:}

%65M15 error bounds

%%%%%%%%%%%%%%%%%%%%%%%%%%%%%%%%%%%%%%%%%%%%%%%%%%%%%%%%%%%%%%%%%%%%%%%%%%%%%%
\section{Introduction}
\subsection{The Boltzmann equation}\label{The Boltzmann equation}
 Consider the Boltzmann equation in the whole space
\begin{equation*}
\begin{cases}
\partial_t F(x,t,\xi)+\xi \cdot \nabla_xF(x,t,\xi)=Q(F,F)(x,t,\xi), \ \ (x,t,\xi)\in \mathbb{R}^3\times \mathbb{R}_+\times\mathbb{R}^3,\\
F(x,0,\xi)=F_0(x,\xi).
\end{cases}
\end{equation*}
Here, $x = (x^1,x^2,x^3) \in \mathbb{R}^3$ represents the space variables, $t \geq 0$ is the time variable, and $\xi = (\xi^1,\xi^2,\xi^3) \in \mathbb{R}^3$ denotes the microscopic velocity. The left-hand side of the equation, $\partial_tF + \xi \cdot \nabla_x F$, corresponds to the transport part, which measures the rate of change over time along the particle path. According to the Boltzmann equation, this change is attributed to the collision operator $Q(F,F)$, which takes the form of binary collisions:
\begin{equation}\label{Q}
\begin{aligned}
Q(g,h)=\frac12 \int_{\mathbb{R}^3\times S^2,(\xi-\xi_*)\cdot \Omega \ge0}&(-g(\xi)h(\xi_*)-h(\xi)g(\xi_*)+g(\xi')h(\xi_*')+h(\xi')g(\xi_*'))\\
&\cdot B(\xi-\xi_*, \Omega)d\xi_*d\Omega,
\end{aligned}
\end{equation}
where
\begin{equation*}
 \begin{cases}
\xi'=\xi-[(\xi-\xi_*)\cdot \Omega]\Omega, \\
\xi_*'=\xi_*+[(\xi-\xi_*)\cdot \Omega]\Omega. \\
\end{cases}
\end{equation*}
The collision kernel $B(\xi-\xi_*, \Omega)$ depends on the inter-molecular forces between the particles. For the hard sphere model, it can be expressed as:
\begin{equation*}
B(\xi-\xi_*, \Omega)=(\xi-\xi_*)\cdot \Omega,\ \ \Omega\in S^2.
\end{equation*}

The global Maxwellian states $\mathsf{M}=\mathsf{M}_{[\rho,v,\theta]}$ with $(\rho,v,\theta)$ being constants,
\begin{equation*}
\mathsf{M}_{[\rho,v,\theta]}=\frac{\rho}{(2\pi R\theta)^{\frac32}}\exp(-\frac{|\xi-v|^2}{2R\theta})
\end{equation*}
satisfy $Q(\mathsf{M},\mathsf{M})=0$ and are equilibriums of the Boltzmann equation. Here $v=(v^1,v^2,v^3)$ is the macroscopic velocity. $R$, $\rho$ and $\theta$ represent the gas constant, macroscopic density and macroscopic temperature, respectively. 

 Next consider the linearization of the Boltzmann equation around a global Maxwellian $\mathsf{M}$:
\begin{equation*}
F=\mathsf{M}+\sqrt{\mathsf{M}}f,
\end{equation*}
where the linearization leads to
\begin{equation}\label{LBE}
 \begin{cases}
\partial_tf+\xi \cdot \nabla_xf=\mathsf{L}f,\ \mathrm{linearized} \ \mathrm{Boltzmann}\ \mathrm{equation},\\
\mathsf{L}f=\frac{2Q(\sqrt{\mathsf{M}}f,\mathsf{M})}{\sqrt{\mathsf{M}}}, \ \ \ \,  \ \mathrm{linearized}\ \mathrm{collision} \ \mathrm{operator}.
\end{cases}
\end{equation}
Further, we consider the nonlinear problem for the perturbation from equilibrium
\begin{equation*}\label{NLBE}
\partial_tf+\xi \cdot \nabla_xf=\mathsf{L}f+\Gamma(f,f),%\ \mathrm{nonlinearized} \ \mathrm{Boltzmann}\ \mathrm{equation},\\
\end{equation*}
where the nonlinear collision operator $\Gamma$ is given by
\begin{equation}\label{gamma10}
\Gamma(f,g)=\frac{Q(\sqrt{\mathsf{M}}f,\sqrt{\mathsf{M}}g)}{\sqrt{\mathsf{M}}}.%\quad \quad \ \ , \ \mathrm{nonlinearized}\ \mathrm{collision} \ \mathrm{operator}.
\end{equation}

Although the Boltzmann equation is derived from first-principle, uncertainties may still arise from the collision kernel, background equilibrium state, the initial data and the boundary data. Since the collision kernel depends on the inter-molecular forces between the particles, an incomplete characterization of the interaction mechanism may lead to uncertainties. The uncertainties in the background equilibrium state, the initial data and the boundary data originate from observational or modelling errors. A rigorous quantification of these uncertainties is essential for validating model predictions and enhancing the robustness of simulations, thereby improving the reliability of theoretical analysis and helping model calibration. 

In this paper, we study the Boltzmann equations with random uncertainties from the initial data and collision kernel in the whole space.
 We introduce the initial value problem of the Boltzmann equation with uncertainties:
\begin{equation}\label{UB}
 \begin{cases}
\partial_tF(x,t,\xi,z)+\xi\cdot \nabla_xF(x,t,\xi,z)=Q^z(F,F)(x,t,\xi,z), \ (x,t,\xi,z)\in \mathbb{R}^3\times \mathbb{R}_+\times\mathbb{R}^3\times I_z,\\
F(x,0,\xi,z)=F_0(x,\xi,z),\\
\end{cases}
\end{equation}
where the uncertain variable $z$ is one-dimensional, and $I_z$ has finite support $|z| \le C_z$
(which is the case, for example, for the uniform and Beta distributions). %Without loss of generality, let $C_z=1$. 
Both the initial data $F_0(x,\xi,z)$ and the collision operator $Q^z$ are influenced by uncertainties. The uncertainty of $Q^z$ comes from $B(\xi-\xi_*, \Omega,z)$, which satisfies 
\begin{equation*}
\begin{aligned}
B(|\xi-\xi_*|, \Omega,z)=|\xi-\xi_*|b(\frac{(\xi-\xi_*)\cdot \Omega}{|\xi-\xi_*|},z),
\end{aligned}
\end{equation*}
and $b(\eta,z)$ satisfies 
\begin{equation}\label{b1}
\begin{aligned}
C_{b_1}\le\int_0^{\frac{\pi}{2}}b(\cos{\theta},z)\sin{\theta}d\theta \le C_{b_2},%\ \text{for}\ \text{some}\ C_{b_1},C_{b_2}>0.
\end{aligned}
\end{equation}
here, $C_{b_1}, C_{b_2}>0$, and both of $C_{b_1}$ and $C_{b_2}$ are independent of $z$.
Further, in order to study the higher-order derivatives of the solution with respect to $z$, $b(\eta,z)$ is assumed to satisfy 
\begin{equation}\label{b2}
\begin{aligned}
\sum_{k=1}^\alpha|\partial^k_z b(\eta,z)|\le C_{b_*},\ \text{for}\ \text{some}\ C_{b_*}>0 \ \text{and}\ \text{given}\ \alpha \in \mathbb{N}_+,
\end{aligned}
\end{equation}
where $C_{b_*}$ is independent of $z$.
From \cite{liu2018hypocoercivity}, Property \eqref{b1} and Property (\ref{b2}) are reasonable assumptions for the uncertainty quantification problem.

Assuming that the background Maxwellian
$\mathsf{M}$ has no uncertainties, we obtain the nonlinear Boltzmann equation for perturbation with uncertainties as follows:
\begin{equation}\label{UC}
 \begin{cases}
\partial_tf(x,t,\xi,z)+\xi\cdot \nabla_xf(x,t,\xi,z)=\mathsf{L}^zf(x,t,\xi,z)+\Gamma^z(f,f),\ x\in \mathbb{R}^3,\\
f(x,0,\xi,z)=f_0(x,\xi,z).\\
\end{cases}
\end{equation}
Here, the uncertainty in the collision operator $Q^z$ leads to the uncertainty in $\mathsf{L}^z$ and $\Gamma^z$. By \eqref{LBE} and \eqref{gamma10}, $\mathsf{L}^z$ and $\Gamma^z$ are given by
\begin{equation}\label{Q1}
\begin{aligned}
\mathsf{L}^zf=\int_{\mathbb{R}^3\times S^2,(\xi-\xi_*)\cdot \Omega \ge0}\sqrt{\mathsf{M}}\mathsf{M}_*(-\frac{f}{\sqrt{\mathsf{M}}}-\frac{f_*}{\sqrt{\mathsf{M}_*}}+\frac{f'}{\sqrt{\mathsf{M}'}}+\frac{f^{\prime}_*}{\sqrt{\mathsf{M}^{\prime}_*}})b(\frac{(\xi-\xi_*)\cdot \Omega}{|\xi-\xi_*|},z)|\xi-\xi_*|d\xi_*d\Omega,\\
\end{aligned}
\end{equation}
and
\begin{equation}\label{Q2}
\begin{aligned}
\Gamma^z(f,g)=\frac12 \int_{\mathbb{R}^3\times S^2,(\xi-\xi_*)\cdot \Omega \ge0}\sqrt{\mathsf{M}_*}(-fg_*-f_*g+f'g^{\prime}_*+f^{\prime}_*g')b(\frac{(\xi-\xi_*)\cdot \Omega}{|\xi-\xi_*|},z)|\xi-\xi_*|d\xi_*d\Omega,\\
\end{aligned}
\end{equation}
with the abbreviations $\phi_*=\phi(\xi_*),\phi^\prime=\phi(\xi^\prime),\phi^{\prime}_*=\phi(\xi^\prime_*)$.

\subsection{Notations}
Let us define some notations used in this paper. We denote $\left \langle
\xi \right \rangle ^{s}=(1+|\xi |^{2})^{s/2}$, $s\in {\mathbb{R}}$. For the
microscopic variable $\xi $, we denote the Lebesgue spaces
\begin{equation*}
\|g\|_{L_{\xi }^{q}}=\Big(\int_{{\mathbb{R}}^{3}}|g|^{q}d\xi \Big)^{1/q}\text{
if }1\leq q<\infty \text{,}\quad \quad \|g\|_{L_{\xi }^{\infty }}=\text{ess} \sup_{\xi
\in {\mathbb{R}}^{3}}|g(\xi )|\text{,}
\end{equation*}%
and the weighted norms can be defined by
\begin{equation*}
\|g\|_{L_{\xi ,\beta }^{q}}=\Big(\int_{{\mathbb{R}}^{3}}\left \vert \left
\langle \xi \right \rangle ^{\beta }g\right \vert ^{q}d\xi \Big)^{1/q}\text{
if }1\leq q<\infty \text{,}\quad \quad \|g\|_{L_{\xi ,\beta }^{\infty
}}=\text{ess} \sup_{\xi \in {\mathbb{R}}^{3}}\left \vert \left \langle \xi \right
\rangle ^{\beta }g(\xi )\right \vert \text{.}
\end{equation*}%
%in ${\mathbb{R}}^{3}$
The $L_{\xi
}^{2}$ inner product will be denoted by $(\cdot
,\cdot)_\xi$, i.e.,
\begin{equation*}
(f,g)_\xi=\int_{\mathbb{R}^3} f(\xi )\overline{g(\xi )}d\xi
\text{.}
\end{equation*}%

For the space variable $x$, we have similar notations, namely,
\begin{equation*}
\|g\|_{L_{x}^{q}}=\left( \int_{{\mathbb{R}^{3}}}|g|^{q}dx\right) ^{1/q}\text{
if }1\leq q<\infty \text{,}\quad \quad \|g\|_{L_{x}^{\infty }}=\text{ess} \sup_{x\in {%
\mathbb{R}^{3}}}|g(x)|\text{.}
\end{equation*}%
%Furthermore, we define the high order Sobolev norm: let $s\in {\mathbb{N}}$
%and define
%\begin{equation*}
%\| g\| _{H_{\xi }^{s}}=\sum_{|\alpha |\leq s}\|
%\partial _{\xi }^{\alpha }g\| _{L_{\xi }^{2}}\text{,\  \  \  \  \  \ }%
%\| g\| _{H_{x}^{s}}=\sum_{|\alpha |\leq s}\|
%\partial _{x}^{\alpha }g\| _{L_{x}^{2}}\text{,}
%\end{equation*}%
%where $\alpha $ is any multi-index with $|\alpha |\leq s$.

Next, with $\mathcal{X}$ and $\mathcal{Y}$ being normed spaces, we define
\begin{equation*}
\left \Vert g\right \Vert _{\mathcal{X(Y)}}=\| \ \| g\| _{\mathcal{Y}}\| _{\mathcal{X}}\text{.}
\end{equation*}%=\left \Vert g\right \Vert _{\mathcal{XY}}

 Finally, for the uncertainty variable $z$, we denote 
\begin{equation*}
\Vert g\Vert _{X^{\xi,2}_{x,r}}=\text{ess} \sup_{z
\in {I_z}}\Vert g\Vert _{L_{x}^{r}(L_{\xi}^{2})},\ \Vert g\Vert _{X_{\xi,2}^{x,r}}=\text{ess} \sup_{z
\in {I_z}}\Vert g\Vert _{L_{\xi}^{2}(L_{x}^{r})},
\end{equation*}
and 
\begin{equation*}
\Vert g\Vert _{X^{\xi,\infty;\beta}_{x,r}}=\text{ess} \sup_{z
\in {I_z}}\Vert g\Vert _{L_{x}^{r}(L_{\xi,\beta}^{\infty})},\ \Vert g\Vert _{X_{\xi,\infty;\beta}^{x,r}}=\text{ess} \sup_{z
\in {I_z}}\Vert g\Vert _{L_{\xi,\beta}^{\infty}(L_{x}^{r})}.
\end{equation*}
Since the weighted velocity space in this paper is only $L^\infty_\xi$, the notation is simplified by
\begin{equation*}
\Vert g\Vert _{X^{\xi,\infty;\beta}_{x,r}}=:\Vert g\Vert _{X^{\xi;\beta}_{x,r}},\ \Vert g\Vert _{X_{\xi,\infty;\beta}^{x,r}}=:\Vert g\Vert _{X_{\xi;\beta}^{x,r}}.
\end{equation*}

 In order to study the SG system, the Boltzmann equations for multiple species are considered. A vector-valued function $\vec{g}_K(x,t,\xi)$ is defined by
$$
\vec{g}_K=(g_1,\cdots,g_k,\cdots,g_K)^T.
$$
The norm on $L_{\xi}^2$ (and similarly for the other spaces) is defined by 
$$
\left\|\vec{g}_K\right\|^2_{L_{\xi}^2}=\sum_{k=1}^{K}\int_{\mathbb{R}^3}g_k^2d\xi, \ \ \text{for} \ g_k\in L_{\xi}^2, \ 1\le k\le K.
$$

For simplicity of notations, hereafter, we abbreviate $\le C$ to $\lesssim$ , where $C$ is a constant depending
only on fixed numbers.
 
In order to emphasize the time dependence of the solution operator, we denote the solution operator to the linearized Boltzmann equation (\ref{LBE}) as $\mathbb{G}^t$. 

\subsection{Review of previous works and main objectives}\label{main objectives}
Let us briefly review the relevant works on the Boltzmann equation and uncertainty problems. In terms of theoretical analysis, the works in \cite{guo2004boltzmann,liu2004energy,liu2004boltzmann,Nishida1978fluid,ukai1974existence} provide classical results in the study of the Boltzmann equation. These studies contribute to the analysis of stability and large-time behavior of solutions, offering insights into the dynamical evolution of perturbations and decay over time. To investigate the quantitative behavior of solutions, \cite{liu2004green,[LiuYu1],liu2011solving} introduced the Green's function approach for the initial value problem of the Boltzmann equation. 

From the classical result in \cite{ukai1974existence}, the large-time behavior of solutions to the Boltzmann equation on the torus domain exhibits {\it exponential} decay, which can also be derived using the energy method (see \cite{briant2015from, MouNeu06}, for example). Based on \cite{briant2015from, MouNeu06}, for uncertainty problems, Liu-Jin established a general framework by energy method in \cite{liu2018hypocoercivity} to study linear and nonlinear kinetic equations with random uncertainties from the initial data or collision kernel, along with their stochastic Galerkin approximations on the torus. Then, Daus-Liu-Jin investigated the case of the nonlinear multi-species Boltzmann equation with random uncertainty on the torus in \cite{EJin}. Studying the uncertainty problems will contribute to verify the spectral accuracy of the Generalized Polynomial Chaos approach within the Stochastic Galerkin framework (referred to as gPC-SG) approximation for the numerical solution. The gPC-SG \cite{g1,g2,g3,g4,g5,g6} is a widely used and efficient method for the numerical solution of such equations with uncertainties. In comparison to the classical Monte Carlo method, if the solution is sufficiently smooth, the gPC-SG approach offers spectral accuracy in the random space while the Monte Carlo method achieves convergence with only half-order accuracy. 

 Previous studies on the Boltzmann equation with uncertainties have primarily focused on {\it torus} domains. It is natural to consider the Boltzmann equation with uncertainties in the {\it whole space}. Compared to the case of the torus,
 the large-time behavior of the solution in the whole space exhibits {\it polynomial} decay rather than exponential decay by the classical results (see \cite{shu}, for example). 

{\it This is the first study of the uncertainty quantification (referred to as UQ) problem for the Boltzmann equation in the whole space.}
 One of our primary objectives is to analyze the large-time behavior of solutions and their higher-order derivatives with respect to random inputs. Study of the higher-order derivatives in the random uncertainty variable is important for sensitivity analysis \cite{smith2024}, as well as the study of numerical accuracy in uncertainty quantification.

Let $\mathbb{G}^t(z)$ be the solution operator of linearized problem of (\ref{UC}):
\begin{equation}\label{LUC}
 \begin{cases}
\partial_tg(x,t,\xi,z)+\xi\cdot \nabla_xg(x,t,\xi,z)=\mathsf{L}^zg(x,t,\xi,z),\ x\in \mathbb{R}^3,\\
g(x,0,\xi,z)=g_0(x,\xi,z).\\
\end{cases}
\end{equation}
Since the uncertainty does not alter the fundamental structure of the linearized Boltzmann equation, the solution operator $\mathbb{G}^t(z)$ still solves the linearized problem (\ref{LUC}) for any fixed $z$.
By the classical results in \cite{shu,[LiuYu1],liu2011solving}, whose details are given in the Theorem \ref{CLIP} of this paper, $\mathbb{G}^t(z)$ satisfies that
\begin{equation}\label{GTZ}
\begin{aligned}
\left\|\mathbb{G}^t(z)g_0\right\|_{L^{\infty}_{\xi,\beta}(L_x^2)}&\lesssim\frac{1}{(1+t)^{\frac34}}(\left\|g_0(z)\right\|_{L^{\infty}_{\xi,\beta}(L_x^1)}+\left\|g_0(z)\right\|_{L^{\infty}_{\xi,\beta}(L_x^{2})}),\\
\left\|\mathbb{G}^t(z)g_0\right\|_{L^{\infty}_{\xi,\beta}(L_x^\infty)}&\lesssim\frac{1}{(1+t)^{\frac32}}(\left\|g_0(z)\right\|_{L^{\infty}_{\xi,\beta}(L_x^1)}+\left\|g_0(z)\right\|_{L^{\infty}_{\xi,\beta}(L_x^{\infty})}),\\
\left\|\mathbb{G}^t(z)g_0\right\|_{L^{\infty}_{\xi,\beta}(L_x^2)}&\lesssim\left\|g_0(z)\right\|_{L^{\infty}_{\xi,\beta}(L_x^2)},\ \text{for} \ \beta>3/2.
\end{aligned}
\end{equation}
Compared to the requirements of the energy method on the initial data, the semigroup method {\it does not impose any regularity with respect to $x$ or $\xi$ on the initial data } $g_0$.

When we consider the large-time behavior of the higher-order derivatives with respect to $z$, a challenge arises: due to the uncertainty in the collision kernel, the time decay of higher-order derivatives of the solution with respect to $z$ slows down. 
Let us take linearized equation (\ref{LUC}) as an illustration.
By differentiating equation (\ref{LUC}) with respect to $z$, we obtain
\begin{equation}\label{UCD}
 \begin{cases}
\partial_t(\partial_zg)+\xi\cdot \nabla_x(\partial_zg)=\mathsf{L}^z(\partial_zg)+(\partial_z\mathsf{L}^z)g,\ x\in \mathbb{R}^3,\\
\partial_zg(x,0,\xi,z)=\partial_zg_0(x,\xi,z).\\
\end{cases}
\end{equation}
Here, for $1\le k \le \alpha$, $(\partial^k_z\mathsf{L}^z)g$ satisfies 
\begin{equation}\label{Q5}
\begin{aligned}
(\partial^k_z\mathsf{L}^z)g=\int_{\mathbb{R}^3\times S^2,(\xi-\xi_*)\cdot \Omega \ge0}\sqrt{\mathsf{M}}\mathsf{M}_*(-\frac{g}{\sqrt{\mathsf{M}}}-\frac{g_*}{\sqrt{\mathsf{M}_*}}+\frac{g'}{\sqrt{\mathsf{M}'}}+\frac{g^{\prime}_*}{\sqrt{\mathsf{M}^{\prime}_*}})\partial^k_zb(\frac{(\xi-\xi_*)\cdot \Omega}{|\xi-\xi_*|},z)|\xi-\xi_*|d\xi_*d\Omega.\\
\end{aligned}
\end{equation}

Apply Duhamel's principle to express $\partial_zg$ by $\mathbb{G}^t(z)$ as follows 
\begin{equation*}
\partial_zg(x,t,\xi,z)=\mathbb{G}^{t}(z)\partial_zg_0+\int_0^t\mathbb{G}^{t-\tau}(z)(\partial_z\mathsf{L}^z)g(\tau)d\tau.
\end{equation*}
Assuming $\partial_zg_0=0$, then by the estimates of $\mathbb{G}^{t}(z)$ and $(\partial_z\mathsf{L}^z)g$ in \eqref{GTZ} and \eqref{Lz}, we have
\begin{equation*}
\begin{aligned}
\left\|\partial_zg(z)\right\|_{L^{\infty}_{\xi,\beta}(L_x^2)}&=\left\| \int_0^t\mathbb{G}^{t-\tau}(z)(\partial_z\mathsf{L}^z)g(\tau)d\tau\right\|_{L^{\infty}_{\xi,\beta}(L_x^2)}\\
&\lesssim\int_0^t\left\| (\partial_z\mathsf{L}^z)g(\tau)\right\|_{L^{\infty}_{\xi,\beta-1}(L_x^2)}d\tau\\
&\lesssim\int_0^t(1+\tau)^{-\frac34}d\tau(\left\|g_0(z)\right\|_{L^{\infty}_{\xi,\beta}(L_x^1)}+\left\|g_0(z)\right\|_{L^{\infty}_{\xi,\beta}(L_x^{2})})\\
&\lesssim(1+t)^{\frac14}(\left\|g_0(z)\right\|_{L^{\infty}_{\xi,\beta}(L_x^1)}+\left\|g_0(z)\right\|_{L^{\infty}_{\xi,\beta}(L_x^{2})}).
\end{aligned}
\end{equation*}
The time decay rate of $\partial_zf$ estimate in $L_x^2$ changes from $(1+t)^{-3/4}$ decay to $(1+t)^{\frac14}$ growth. Even worse, the $L_x^2$ estimate for $\partial^2_zg$ exhibits faster time growth by repeating this procedure. The decay estimate of $\mathbb{G}^{t}(z)$ obtained by the energy method \cite{duan2011,kawashima1990} is similar to that obtained by the semigroup method and leads to the same time growth as well. Notably, this issue does not occur in the case of torus domain, where the solution decays exponentially over time. Moreover, for the problem (\ref{LUC}) on the torus, the large-time behavior of $\partial_zg$ and higher-order terms $\partial_z^kg$ still exhibits exponential decay through Duhamel's principle.

Therefore, the main challenge we face in this paper is that directly applying the standard semigroup estimate leads to a growth in time for higher-order derivatives with respect to random variable. Thanks to the Green's function approach, we are able to establish refined semigroup estimate and restore the time decay of higher-order derivatives similar to the problem without uncertainties, up to some logarithmic growth factor.

\subsection{The main results}\label{Main result}
We focus on the initial value problem of the Boltzmann equation with uncertainties, i.e., problem (\ref{UC}), where both the initial data $f_0$ and the collision kernel are influenced by uncertainties. By the Green's function approach, we have the following theorem:
\begin{theorem}\label{T1}
 Assuming that the collision kernel satisfies (\ref{b1}) and (\ref{b2}),
then for given $\alpha\in \mathbb{N}$, $C_{b_1},C_{b_2}$ and $C_{b_*}$ in (\ref{b1}) and (\ref{b2}), there exist positive constants $\delta$ and $C_\alpha$ such that if 
$$
\sum_{k=0}^\alpha(\left\|\partial^k_zf_0\right\|_{X_{\xi;\beta}^{x,1}}+\left\|\partial^k_zf_0\right\|_{X_{\xi;\beta}^{x,\infty}})<\delta,
$$
then there exists a unique global solution $f=f(x,t,\xi,z)$ to equation (\ref{UC}), satisfying 
$$
\left\|f(t)\right\|_{X_{\xi;\beta}^{x,2}}\le \frac{ C_\alpha}{(1+t)^{\frac34}}(\left\|f_0\right\|_{X_{\xi;\beta}^{x,1}}+\left\|f_0\right\|_{X_{\xi;\beta}^{x,\infty}}),
$$
and
$$
\left\|f(t)\right\|_{X_{\xi;\beta}^{x,\infty}}\le \frac{ C_\alpha}{(1+t)^{\frac32}}(\left\|f_0\right\|_{X_{\xi;\beta}^{x,1}}+\left\|f_0\right\|_{X_{\xi;\beta}^{x,\infty}}).
$$
Moreover, for $1\le k\le \alpha$, it holds that
$$
\left\|\partial^k_zf(t)\right\|_{X_{\xi;\beta}^{x,2}}\le \frac{ C_\alpha(\ln(1+t))^{k-1}}{(1+t)^{\frac34}}\sum_{s=0}^k(\left\|\partial^s_zf_0\right\|_{X_{\xi;\beta}^{x,1}}+\left\|\partial^s_zf_0\right\|_{X_{\xi;\beta}^{x,\infty}}),
$$
and
$$
\left\|\partial^k_zf(t)\right\|_{X_{\xi;\beta}^{x,\infty}}\le \frac{ C_\alpha(\ln(1+t))^{k-1}}{(1+t)^{\frac32}}\sum_{s=0}^k(\left\|\partial^s_zf_0\right\|_{X_{\xi;\beta}^{x,1}}+\left\|\partial^s_zf_0\right\|_{X_{\xi;\beta}^{x,\infty}}).
$$

In this theorem, all constants are independent of $z$.
	\end{theorem}

\begin{remark}
Theorem \ref{CLIP} presents the classical results for the initial value problem of the linearized Boltzmann equation. 
The large-time behavior of the solutions to problems (\ref{UC}) with uncertainties is consistent with the classical results. For the estimates of $\partial^k_zf$, the time decay exhibits an additional logarithmic growth depending on $k$. 
\end{remark}

Another objective of this paper is to verify the spectral accuracy of the gPC-SG approximation for the numerical solution. We begin by briefly introducing the gPC-SG system. Let $f$ be the solution to equation (\ref{UC}), and we want to seek an approximate solution $f^K$ in the following form:
\begin{equation}\label{RSG1}
\begin{aligned}
f(x,t,\xi,z)\approx \sum_{{k}=1}^Kf_{{k}}(x,t,\xi)\psi_{{k}}(z)=:f^K(x,t,\xi,z).
\end{aligned}
\end{equation}
 Here $\{\psi_k(z)\}$ are orthonormal gPC basis functions satisfying 
\begin{equation*}
\begin{aligned}
\int_{I_z}\psi_{k}(z)\psi_{j}(z)\pi(z)dz=\delta_{kj},\  1\le k,j\le K,
\end{aligned}
\end{equation*}
and $\pi(z)$ is the probability distribution function of $z$. 

We assume the technical condition 
\begin{equation}\label{psik}
\|\psi_k\|_{L^\infty_z}\le C k^n,\ k\ge1,
\end{equation}
with a parameter $n\ge 0$ and a positive constant $C$. For the case $I_z=[-1,1]$ with uniform distribution, $\psi_k$'s are the normalized Legendre polynomials,
and \eqref{psik} holds with $n=1/2$. For the case $I_z=[-1,1]$ with the distribution $\pi(z)=\frac{2}{\sqrt{\pi\sqrt{1-z^2}}}$,
 $\psi_k$'s are the normalized Chebyshev polynomials and \eqref{psik} holds with $n= 0$. 

One can expand $f$ by 
\begin{equation*}
\begin{aligned}
f(x,t,\xi,z)= \sum_{k=1}^{\infty}\widehat{f}_{k}(x,t,\xi)\psi_{k}(z), \ \ \ \widehat{f}_{k}(x,t,\xi):=\int_{I_z}f(x,t,\xi,z)\psi_{k}(z)\pi(z)dz.
\end{aligned}
\end{equation*}
Define the Galerkin projection operator $P_{K}$ as 
\begin{equation}\label{PK}
\begin{aligned}
P_Kf(x,t,\xi,z):= \sum_{k=1}^{K}\widehat{f}_{k}(x,t,\xi)\psi_{k}(z).
\end{aligned}
\end{equation}

By inserting ansatz (\ref{RSG1}) into the nonlinear Boltzmann equation \eqref{UC} and applying the standard Galerkin projection, one obtains the gPC-SG system for $f_{k}$, $1\le k\le K$:
\begin{equation}\label{RPCSG}
 \begin{cases}
\partial_tf_{k}+\xi\cdot \nabla_xf_{k}=\mathsf{L}_{k}(f^K)+\Gamma_{k}(f^K,f^K),\ x\in \mathbb{R}^3,\\
f_{k}(x,0,\xi)=f_{k}^0(x,\xi), \\
\end{cases}
\end{equation}
with the initial data given by 
\begin{equation*}
f^0_{k}(x,\xi):=\int_{I_z}f_0(x,\xi,z)\psi_{k}(z)\pi(z)dz.
\end{equation*}
For more details about $\mathsf{L}_{k}$ and $\Gamma_{k}$, see Section \ref{Galerkin method}.

Corresponding to $f^K$ and $f_0$, the vector-valued functions $\vec{f}_K(x,t,\xi)$ and $\vec{f}_{\mathrm{in},K}$ are defined by
\begin{equation}\label{fini}
\vec{f}_K=(f_1,\cdots,f_k,\cdots,f_K)^T,\  \vec{f}_{\mathrm{in},K}=(f_1^0,\cdots,f_k^0,\cdots,f_K^0)^T.
\end{equation}

When the collision kernel contains uncertainties, the integral of its product with the basis function $\psi_k(z)$ leads to a growth in the nonlinear terms $\Gamma_{k}(f^K,f^K)$ with respect to $K$. To absorb this growth, following the approach in \cite{liu2018hypocoercivity}, we introduce weighted estimates, which enables one to obtain the desired estimates with initial data independent of $K$.

Define a weight matrix $\mathbf{W}$ by $\mathbf{W}=\{ w_{ij}\}_{K\times K}$, and $w_{ij}$ satisfy that
\begin{equation}\label{weight}
 \begin{cases}
w_{ij}=i^m, \ \text{if} \ i=j,\\
w_{ij}=0, \ \text{if} \ i\ne j, \\
\end{cases}
\end{equation}
here, $m> n+1$, and $n$ depends on the selected orthonormal gPC basis functions in \eqref{psik}.

Assuming the collision kernel is linear in $z$, 
we can utilize the semigroup method and Theorem \ref{T1} to solve the gPC-SG system and provide the estimates of the gPC error, and
they satisfy the following theorem:
\begin{theorem}\label{ESG} 
Assuming that $\{\psi_k(z)\}$ in \eqref{RSG1} satisfies 
\begin{equation*}
\|\psi_k\|_{L^\infty_z}\le C k^n,\ k\ge1,\ \text{for} \ \text{some} \ C>0, 
\end{equation*}
and $m>n+1$. If the collision kernel is linear in $z$, then we have the following:

$(\mathrm{I})$ 
Let $f^K$ be the solution of (\ref{RPCSG}) with the initial data $f_0$. $\vec{f}_K$ is the vector-valued function corresponding to $f^K$ and $\vec{f}_{\mathrm{in},K}$ is the the initial data corresponding to $f_0$, as defined in \eqref{fini}. There exist positive constants $\delta$ and $C_S$ independent of $K$ such that if 
$$
\left\|\mathbf{W}\vec{f}_{\mathrm{in},K}\right\|_{L^{\infty}_{\xi,\beta}(L_x^\infty)}+\left\|\mathbf{W}\vec{f}_{\mathrm{in},K}\right\|_{L^{\infty}_{\xi,\beta}(L_x^{1})}=:\varepsilon_S(\vec{f}_{\mathrm{in},K})<\delta, \ \beta > 3/2,
$$
 then 
\begin{equation*}
\begin{aligned}
\left\|\mathbf{W}\vec{f}_K(t)\right\|_{L^{\infty}_{\xi,\beta}(L_x^2)}\le C_S(1+t)^{-\frac34}\varepsilon_S(\vec{f}_{\mathrm{in},K}),
\end{aligned}
\end{equation*}
\begin{equation*}
\begin{aligned}
\left\|\mathbf{W}\vec{f}_K(t)\right\|_{L^{\infty}_{\xi,\beta}(L_x^{\infty})}\le C_S(1+t)^{-\frac32}\varepsilon_S(\vec{f}_{\mathrm{in},K}).
\end{aligned}
\end{equation*}
 Moreover, if $\mathsf{P}_0\vec{f}_{\mathrm{in},K}=\vec{0}$, then we will get extra $(1+t)^{-1/2}$ decay rate for each estimate above. More details about the macroscopic projection $\mathsf{P}_0$ are given in Section \ref{macro-micro projections}.  

$(\mathrm{II})$ Let $f$ be the solution of equation (\ref{UC}) . For given $\alpha\in\mathbb{N}_+$, there exist positive constants $\delta$ and $C_{\alpha,e}$ such that if 
$$
\sum_{s=0}^\alpha\sum_{k=1}^\infty\left\|k^m\int_{I_z}\psi_k(z)\pi(z)\partial^s_zf_0dz\right\|_{L^{\infty}_{\xi,\beta}(L_x^{1}\cap L_x^{\infty})}=:\varepsilon_\alpha(f_0)<\delta,
$$
 then the gPC error $f^e:=f-f^K$ satisfies that
\begin{equation*}
\begin{aligned}
\left\|f^e(t)\right\|_{X_{\xi;\beta}^{x,2}}&\le\frac{ C_{\alpha,e}(\ln(1+t))^{\alpha-1}}{K^\alpha(1+t)^{\frac12}}\varepsilon_\alpha(f_0),\\
\left\|f^e(t)\right\|_{X_{\xi;\beta}^{x,\infty}}&\le\frac{ C_{\alpha,e}(\ln(1+t))^{\alpha-1}}{K^\alpha(1+t)^{\frac54}}\varepsilon_\alpha(f_0).
\end{aligned}
\end{equation*}
In this theorem, all constants are independent of $z$ and $K$.
\end{theorem}

\begin{remark}
The large-time behavior of the solution for the Boltzmann equation system (\ref{RPCSG}) is identical to that of the Boltzmann equation without uncertainties. The weighted estimates of the solution allows the initial data to be independent of $K$.
As the number $K$ of expansion terms increases, the numerical error decreases, which implies that the accuracy of the gPC-SG system is guaranteed and reliable.
\end{remark}
We now outline the key ideas and strategy to prove our main results. First of all, using the Green's function approach in Theorem \ref{UIf}, we obtain more refined estimates for the solution operator $\mathbb{G}^{t}(z)$ of the linearized problem \eqref{LUC}. By applying {\it $L_x^p\text{-}L_x^r$ type decay estimates}, we establish the large-time behavior of solutions at the linear level.

Secondly, the large-time behavior of the derivatives and higher-order derivatives of the solution with respect to $z$ is considered. Taking linearized equation \eqref{LUC} as an illustration and returning to the challenge mentioned in Section \ref{main objectives}, we need two observations to help estimate $\partial_zg$.  One is that the projection operators $\mathsf{P}_0$ and $\mathsf{P}_1$ in the macro-micro decomposition depends only on the background Maxwellian $\mathsf{M}$.
As a result, the null space of $\partial_z\mathsf{L}^z$ remains the same as that of $\mathsf{L}^z$, and {\it $(\partial_z\mathsf{L}^z)g$ is a microscopic part}.
Another observation is that by Theorem \ref{UIf}, Green's operator acting on {\it the microscopic part provide an additional $(1+t)^{-1/2}$ decay rate}.
Therefore, by macro-micro decomposition, we write $\partial_zg$ as
\begin{equation*}
\partial_zg(x,t,\xi,z)=\mathbb{G}^{t}(z)\partial_zg_0+\int_0^t\mathbb{G}^{t-\tau}(z)(\mathsf{P}_1(\partial_z\mathsf{L}^z)\mathsf{P}_1)\mathbb{G}^{\tau}(z)g_0d\tau.
\end{equation*}
The introduction of microscopic projection operator $\mathsf{P}_1$ and further details are given in Section \ref{macro-micro projections}. 

Suppose $\partial_zg_0=0$,  
by $L_x^p\text{-}L_x^r$ type decay estimates for $\mathbb{G}^t(z)$ in Lemma \ref{UCf} and the additional $(1+t)^{-1/2}$ decay rate, for $1\le r\le 2$, we have 
\begin{equation*}
\begin{aligned}
\left\|\partial_zg\right\|_{X_{\xi;\beta}^{x,2}}&=\left\|\int_0^t\mathbb{G}^{t-\tau}(z)(\mathsf{P}_1(\partial_z\mathsf{L}^z)\mathsf{P}_1)\mathbb{G}^{\tau}(z)g_0d\tau\right\|_{X_{\xi;\beta}^{x,2}}\\
&\lesssim\int^t_0(1+t-\tau)^{-\frac32(\frac{1}{r}-\frac{1}{2})-\frac12}(1+\tau)^{-\frac32(1-\frac{1}{r})-\frac12}(\left\|g_0\right\|_{X_{\xi;\beta}^{x,1}}+\left\|g_0\right\|_{X_{\xi;\beta}^{x,2}})d\tau\\
&=\int^{\frac{t}{2}}_0(1+t-\tau)^{-\frac32(\frac{1}{1}-\frac{1}{2})-\frac12}(1+\tau)^{-\frac32(1-\frac{1}{1})-\frac12}(\left\|g_0\right\|_{X_{\xi;\beta}^{x,1}}+\left\|g_0\right\|_{X_{\xi;\beta}^{x,2}})d\tau\\
&\ \ +\int^t_{\frac{t}{2}}(1+t-\tau)^{-\frac32(\frac{1}{2}-\frac{1}{2})-\frac12}(1+\tau)^{-\frac32(1-\frac{1}{2})-\frac12}(\left\|g_0\right\|_{X_{\xi;\beta}^{x,1}}+\left\|g_0\right\|_{X_{\xi;\beta}^{x,2}})d\tau\\
&\lesssim\frac{1}{(1+t)^{\frac34}}(\left\|g_0\right\|_{X_{\xi;\beta}^{x,1}}+\left\|g_0\right\|_{X_{\xi;\beta}^{x,2}}).
\end{aligned}
\end{equation*}
 More details about the above estimate are given in Section \ref{the solutions}. Therefore, we derive the estimates for $\partial_zg$ without any growth in time. 
 In the final step, following the approach in \cite{liu2004green}, we propose an appropriate ansatz for $f$ and close the nonlinear problem (\ref{UC}). 

In the second part of this paper, the results of studying uncertainty problems are applied to analyze the SG methods for the gPC-SG system and demonstrate the time decay of the numerical error. First, we obtain the weighted estimate of ${f}^K$ by using the semigroup method, which requires the spectral structure of the multi-species linearized operator. Then we decompose the gPC error $f^e$ as follows
\begin{equation*}
\begin{aligned}
f^e:=f-f^K=f-P_Kf+P_Kf-f^K,
\end{aligned}
\end{equation*}
where $P_K$ is the Galerkin projection operator defined in \eqref{PK}.
By Theorem \ref{T1} and the standard estimate on the projection error, the estimate of $f-P_Kf$ is obtained. Finally, treating $f-P_Kf$ as a source term, we apply  Duhamel's principle to the nonlinear Boltzmann equation system for $P_Kf-f^K$, thereby obtaining the estimate for $f^e$.

\subsection{Organization of the paper} The rest of this paper is organized as follows: in Section \ref{Preliminaries}, we introduce some preliminaries for the Boltzmann equation and The Green's function method. Section \ref{Linearized problem} presents refined estimates for the solution operators of the linearized Boltzmann equation in the whole space, along with the estimates for higher-order derivatives of solution. Section \ref{Nonlinear problem} is dedicated to the proof of Theorem \ref{T1}. Section \ref{gPC-SG method} covers the gPC-SG approximation for the uncertain Boltzmann equation and the numerical error.

\section{Preliminaries}\label{Preliminaries}

In this preliminary section, we will review some basic properties of the linearized collision
operator $\mathsf{L}$ and the nonlinear operator $\Gamma$. Afterwards, we introduce the Green's function approach for the initial value problem of the Boltzmann equation in the whole space. %Finally, we introduce the long-short wave decomposition.

\subsection{Linearized collision operator and macro-micro projections}\label{macro-micro projections}
Consider perturbation around a normalized global Maxwellian, 
$$\mathsf{M}_{[1,\vec{0},1]}=\frac{1}{(2\pi)^{3/2}}e^{-\frac{|\xi|^2}{2}},$$
the linearized collision operator $\mathsf{L}$ defined in \eqref{LBE} can be written as % from \cite{liu2011solving}:
\begin{equation*}\label{L}
\mathsf{L}g(\xi)=(-\nu+\mathsf{K})g(\xi)\equiv -\nu(\xi)g(\xi)+\int_{\mathbb{R}^3}K(\xi,\xi_*)g(\xi_*)d\xi_*.
\end{equation*}
Here, $\nu(\xi)$ satisfies that
\begin{equation*}\label{V}
%C_{\nu_1}(1+|\xi|)\le \nu(\xi)\le C_{\nu_2}(1+|\xi|),\ \text {for some} \ C_{\nu_1},C_{\nu_2}>0.
C_{1}(1+|\xi|)\le \nu(\xi)\le C_{2}(1+|\xi|),\ \text {for some} \ C_{1},C_{2}>0.
%\nu(\xi)\sim 1+|\xi|.
\end{equation*}
 Regarding $\mathsf{K}$, we have the following properties:
\begin{lemma}\label{K}For any $\beta\ge 0$, there exist positive constants $C(\beta)$ and $C$ such that
\begin{equation*}
 \begin{cases}
\lVert \mathsf{K}g \rVert_{L^{\infty}_{\xi,\beta+1}}\le C(\beta)\lVert g \rVert_{L^{\infty}_{\xi,\beta}},\\
\lVert \mathsf{K}g \rVert_{L^{\infty}_{\xi}}\le C\lVert g \rVert_{L^{2}_{\xi}}.
\end{cases}
\end{equation*}
\end{lemma}
Next we introduce the macro-micro projection. From \cite{liu2011solving}, the kernel of the linearized collision operator $\mathsf{L}$ is the span of
\begin{equation*}
 \begin{cases}
\chi_0 \equiv \mathsf{M}^{1/2} ,\\
\chi_i \equiv (\xi^i-v^i)\mathsf{M}^{1/2} ,i=1,2,3,\\
\chi_4 \equiv \frac{1}{\sqrt{6}}(|\xi-v|^2-3)\mathsf{M}^{1/2} .\\
\end{cases}
\end{equation*}
Then, define 
\begin{equation*}
\mathsf{P}_0g\equiv \sum^4_{j=0}(g,\chi_j)_\xi\chi_j,
\end{equation*}
and
\begin{equation*}
\mathsf{P}_1\equiv \mathsf{I}-\mathsf{P}_0.
\end{equation*}
Further, $\mathsf{L}$ has the following explicit spectral-gap estimate:
\begin{equation}\label{SGE}
(\mathsf{L}g,g)_\xi\le-\nu_1(\mathsf{P}_1g,\mathsf{P}_1g)_\xi,
\end{equation}
here, $\nu_0>\nu_1>0$, and $\nu_0\equiv\nu(0)=\min\nu(\xi)$.

The linear operator $\mathsf{L}$ and the nonlinear operator $\Gamma$ in (\ref{gamma10}) satisfy that
\begin{equation}\label{P0P1}
\mathsf{P}_0\mathsf{L}=\mathsf{L}\mathsf{P}_0=0,\ \mathsf{P}_0\Gamma(h,u)=0, %\ \text{and}\ \mathsf{P}_1\Gamma(f,g)=\Gamma(f,g),
\end{equation}
and we also have the following lemma from \cite{liu2004green}:
\begin{lemma}\label{Gamma}For any $\beta\ge 0$, there exists positive constant $C(\beta)$ such that
\begin{equation*}
\begin{aligned}
\left\|\Gamma(h,u) \right\|_{L^{\infty}_{\xi,\beta}(L_x^1)}&\le C(\beta)  \left\|h \right\|_{L^{\infty}_{\xi,\beta+1}(L_x^2)}\left\|u \right\|_{L^{\infty}_{\xi,\beta+1}(L_x^2)},\\
\left\|\Gamma(h,u) \right\|_{L^{\infty}_{\xi,\beta}(L_x^{2})}&\le C(\beta)  \left\|h \right\|_{L^{\infty}_{\xi,\beta+1}(L_x^{\infty})}\left\|u \right\|_{L^{\infty}_{\xi,\beta+1}(L_x^{2})},\\
\left\|\Gamma(h,u) \right\|_{L^{\infty}_{\xi,\beta}(L_x^{\infty})}&\le C(\beta)  \left\|h\right\|_{L^{\infty}_{\xi,\beta+1}(L_x^{\infty})}\left\|u \right\|_{L^{\infty}_{\xi,\beta+1}(L_x^{\infty})}.
\end{aligned}
\end{equation*}
%\begin{equation*}
%\lVert \nu^{-1}(\xi)\Gamma(f,g) \rVert_{L^{\infty}_{\xi,\beta}}\le C(\beta)\lVert f \rVert_{L^{\infty}_{\xi,\beta}}\lVert g \rVert_{L^{\infty}_{\xi,\beta}}.
%\end{equation*}
\end{lemma}

Define
\begin{equation}\label{Q6}
\begin{aligned}
(\partial^k_z\Gamma^z)(h,u)=\frac12 \int_{\mathbb{R}^3\times S^2,(\xi-\xi_*)\cdot \Omega \ge0}\sqrt{\mathsf{M}_*}(-hu_*-h_*u+h'u^{\prime}_*+h^{\prime}_*u')\partial^k_zb(\frac{(\xi-\xi_*)\cdot \Omega}{|\xi-\xi_*|},z)|\xi-\xi_*|d\xi_*d\Omega.\\
\end{aligned}
\end{equation}
The following lemma provides the detailed properties of  $\mathsf{L}^z$ and $\Gamma^z$. The proof of the first part of this lemma is based on Lemma \ref{Gamma}, \eqref{Q1}, \eqref{Q2}, \eqref{Q5} and \eqref{Q6}, as well as conditions \eqref{b1} and \eqref{b2}. The second part of this lemma is similar to \eqref{P0P1}. Its proof is based on that the kernel of the linear operator $\mathsf{L}$ depends only on the background equilibrium state $\mathsf{M}$, and it is exactly the same as that of $\mathsf{L}^z$. Here the detailed proof is omitted.

\begin{lemma}\label{bz}
 Assuming that the collision kernel satisfies (\ref{b1}) and (\ref{b2}), then
the following statements $(\mathrm{I})-(\mathrm{II})$ are true.

$(\mathrm{I})$ For given $\alpha\in \mathbb{N}, \beta\ge 0$, there exists positive constant $C^*_\beta$ independent of $z$ such that $\mathsf{L}^z$ and $\Gamma^z$ satisfy, for $0\le k \le \alpha$,
\begin{equation}\label{Lz}
\begin{aligned}
\|(\partial^k_z\mathsf{L}^z)h\|_{L^{\infty}_{\xi,\beta}}\le C^*_\beta \|h\|_{L^{\infty}_{\xi,\beta+1}},
\end{aligned}
\end{equation}
and
\begin{equation}\label{gamma}
\begin{aligned}
\left\|(\partial^k_z\Gamma^z)(h,u) \right\|_{L^{\infty}_{\xi,\beta}(L_x^1)}&\le C^*_\beta \left\|h(z) \right\|_{L^{\infty}_{\xi,\beta+1}(L_x^2)}\left\|u(z) \right\|_{L^{\infty}_{\xi,\beta+1}(L_x^2)},\\
\left\|(\partial^k_z\Gamma^z)(h,u) \right\|_{L^{\infty}_{\xi,\beta}(L_x^{2})}&\le C^*_\beta  \left\|h(z) \right\|_{L^{\infty}_{\xi,\beta+1}(L_x^{\infty})}\left\|u(z)\right\|_{L^{\infty}_{\xi,\beta+1}(L_x^{2})},\\
\left\|(\partial^k_z\Gamma^z)(h,u) \right\|_{L^{\infty}_{\xi,\beta}(L_x^{\infty})}&\le C^*_\beta \left\|h(z)\right\|_{L^{\infty}_{\xi,\beta+1}(L_x^{\infty})}\left\|u(z)\right\|_{L^{\infty}_{\xi,\beta+1}(L_x^{\infty})}.
\end{aligned}
\end{equation}
%Here, $C^*_\beta$ is determined by $C_{b_1}$, $C_{b_2}$ and $C_{b_*}$ in condition (\ref{b1}) and (\ref{b2}), and independent of $z$.

$(\mathrm{II})$ For $k\ge 0$, $\mathsf{L}^z$ and $\Gamma^z$ satisfy that
$$
\mathsf{P}_0(\partial^k_z\mathsf{L}^z)=(\partial^k_z\mathsf{L}^z)\mathsf{P}_0=0, \mathsf{P}_0(\partial^k_z\Gamma^z)(h,u)=0.
$$
\end{lemma}

%\subsection{The Green's function}
\subsection{The Green's function}
Next, we introduce the Green's function approach for the initial value problem of the Boltzmann equation in the whole space. 
Consider the initial value problem of the linearized Boltzmann equation:
\begin{equation}\label{LIP}
 \begin{cases}
\partial_tg(x,t,\xi)+\xi\cdot \nabla_xg(x,t,\xi)=\mathsf{L}g(x,t,\xi),\ x\in \mathbb{R}^3,\\
g(x,0,\xi)=g_0(x,\xi).\\
\end{cases}
\end{equation}
The Green's function ${G}(x,t,\xi;\xi_0)$ of problem (\ref{LIP}) satisfies 
\begin{equation*}\label{3DG}
 \begin{cases}
\partial_t{G}+\xi \cdot \nabla_x{G}=\mathsf{L}{G},\ x\in \mathbb{R}^3, \\
{G}(x,0,\xi;\xi_0)=\delta^{(3)}(x)\delta^{(3)}(\xi-\xi_0).
\end{cases}
\end{equation*}

The Green's function $G$ is the kernel function of the solution operator $\mathbb{G}^t$ for equation (\ref{LIP}).
The detailed structure of the Green's function can be found in Theorem 7.13 of \cite{liu2011solving}.  
The following theorem provides arbitrary $L^r_xL^\infty_{\xi,\beta}$ estimates for a given $L^p_xL^\infty_{\xi,\beta}$ initial data. These results can be obtained by Lemma \ref{K}, Young's inequality, Sobolev embedding and Ukai's bootstrap argument, based on the conclusions of Theorem 7.13 in \cite{liu2011solving}. We omit the proof here, and the interested readers are referred to \cite{W1,liu2011solving,krm} for some details.
\begin{theorem}\label{UIf}
The solution operator $\mathbb{G}^t$ of equation (\ref{LIP}) can be decomposed into the following form:
\begin{equation*}\label{G}
\begin{aligned}
\mathbb{G}^t=\mathbb{G}_F^t+\mathbb{G}_K^t+\mathbb{G}_R^t.
\end{aligned}
\end{equation*}
Here, $\mathbb{G}^t_{F}$ is the fluid-like part, $\mathbb{G}^t_{K}$ is the particle-like part, and $\mathbb{G}^t_{R}$ is the remainder part. For the initial value $g_0$ in (\ref{LIP}), if $\beta > 3/2,$ and $1\le p\le r \le \infty$, then it holds that
\begin{equation*}
\begin{aligned}
\left\|\mathbb{G}_F^tg_0+\mathbb{G}_R^tg_0\right\|_{L^{\infty}_{\xi,\beta}(L_x^{r})}&\lesssim\frac{1}{(1+t)^{\frac32(\frac{1}{p}-\frac{1}{r})}}\left\|g_0\right\|_{L^{\infty}_{\xi,\beta}(L_x^{p})},\\
\left\|\mathbb{G}_K^tg_0\right\|_{L^{\infty}_{\xi,\beta}(L_x^{r})}&\lesssim e^{-t/C}\left\|g_0\right\|_{L^{\infty}_{\xi,\beta}(L_x^{r})}, \ \text{for}\ \text{some}\ C>0.\\
%&\left\|\mathbb{G}_R^tf_0\right\|_{L^{\infty}_{\xi,\beta}(L_x^{r})}\lesssim\frac{1}{(1+t)^{\frac32(\frac{1}{p}-\frac{1}{r})+\frac{1}{2}}}\left\|f_0\right\|_{L^{\infty}_{\xi,\beta}(L_x^{p})}.
\end{aligned}
\end{equation*}
If $\beta> 3/2,$ and $1\le p\le 2\le r \le \infty$, then $\mathbb{G}^t$ satisfies that
\begin{equation*}
\begin{aligned}
\left\|\mathbb{G}^tg_0\right\|_{L_x^r(L^{2}_{\xi})}\lesssim\frac{1}{(1+t)^{\frac32(\frac{1}{p}-\frac{1}{r})}}\left\|g_0\right\|_{L_x^p(L^2_{\xi})}
+e^{-t/C}(\left\|g_0\right\|_{L^{2}_{\xi}(L_x^2)}+\left\|g_0\right\|_{L^{2}_{\xi}(L_x^2)}^{2/r}\left\|g_0\right\|_{L^{2}_{\xi}(L_x^{\infty})}^{1-2/r}),
\end{aligned}
\end{equation*}
\begin{equation*}
\begin{aligned}
\left\|\mathbb{G}^tg_0\right\|_{L^{\infty}_{\xi,\beta}(L_x^{r})}&\lesssim\frac{1}{(1+t)^{\frac32(\frac{1}{p}-\frac{1}{r})}}\left\|g_0\right\|_{L_x^p(L^2_{\xi})}+e^{-t/C}(\left\|g_0\right\|_{L^{\infty}_{\xi,\beta}(L_x^{r})}+\left\|g_0\right\|_{L^{2}_{\xi}(L_x^2)}+\left\|g_0\right\|_{L^{2}_{\xi}(L_x^2)}^{2/r}\left\|g_0\right\|_{L^{2}_{\xi}(L_x^{\infty})}^{1-2/r}).
\end{aligned}
\end{equation*}
 Moreover, if $\mathsf{P}_0g_0=0$, then we will get extra $(1+t)^{-1/2}$ decay rate for each estimate above.
\end{theorem}

\begin{remark}\label{UIf1}
The fluid-like part $\mathbb{G}^t_{F}$ governs the large-time behavior of the solution. $\mathbb{G}^t_{R}$ shares a similar structure
with $\mathbb{G}^t_{F}$ but decays more rapidly over time, and for ease of use, Theorem \ref{UIf} does not single out the properties of $\mathbb{G}^t_{R}$
separately. $\mathbb{G}^t_{K}$ mainly includes the singularity of the solution,
and it decays exponentially with respect to $t$. 
\end{remark}
By choosing appropriate $r$ and $p$ in Theorem \ref{UIf}, classical conclusions can be inferred:
\begin{theorem}[\cite{shu,[LiuYu1],liu2011solving}]\label{CLIP} Let $g$ be the solution of (\ref{LIP}) with the initial data $g_0$ and $\beta > 3/2$. Then  
\begin{equation*}
\begin{aligned}
\left\|g(t)\right\|_{L^{\infty}_{\xi,\beta}(L_x^2)}\lesssim\frac{1}{(1+t)^{\frac34}}(\left\|g_0\right\|_{L^{\infty}_{\xi,\beta}(L_x^1)}+\left\|g_0\right\|_{L^{\infty}_{\xi,\beta}(L_x^{2})}),
\end{aligned}
\end{equation*}
\begin{equation*}
\begin{aligned}
\left\|g(t)\right\|_{L^{\infty}_{\xi,\beta}(L_x^{\infty})}\lesssim\frac{1}{(1+t)^{\frac32}}(\left\|g_0\right\|_{L^{\infty}_{\xi,\beta}(L_x^1)}+\left\|g_0\right\|_{L^{\infty}_{\xi,\beta}(L_x^{\infty})}),
\end{aligned}
\end{equation*}
for $g_0\in L^{\infty}_{\xi,\beta}(L_x^1\cap L^{\infty}_x)$. And 
\begin{equation*}
\begin{aligned}
\left\|g(t)\right\|_{L^{\infty}_{\xi,\beta}(L_x^2)}\lesssim\left\|g_0\right\|_{L^{\infty}_{\xi,\beta}(L_x^2)},
\end{aligned}
\end{equation*}
\begin{equation*}
\begin{aligned}
\left\|g(t)\right\|_{L^{\infty}_{\xi,\beta}(L_x^{\infty})}\lesssim\frac{1}{(1+t)^{\frac34}}(\left\|g_0\right\|_{L^{\infty}_{\xi,\beta}(L_x^2)}+\left\|g_0\right\|_{L^{\infty}_{\xi,\beta}(L_x^{\infty})}),
\end{aligned}
\end{equation*}
for $g_0\in L^{\infty}_{\xi,\beta}(L_x^2\cap L^{\infty}_x)$. Moreover, if $\mathsf{P}_0g_0=0$, then we will get extra $(1+t)^{-1/2}$ decay rate for each estimate above.
\end{theorem} 

In the estimates of higher-order derivatives and the resolution of nonlinear problems, to ensure that the solution's $\xi$ weight matches that of the initial data, we also need the following lemma. This lemma can be obtained base on the detailed structure of the Green's function, and we omit the proof here.
\begin{lemma}\label{GTS}
Suppose $S=S(x,t,\xi)$.
The operators $\mathbb{G}_F^t+\mathbb{G}_R^t$ and $\mathbb{G}_K^t$ satisfy that for $1\le p\le r$, $\beta>3/2$,  
\begin{equation*}
\begin{aligned}
\left\| \int_0^t(\mathbb{G}_F^{t-\tau}+\mathbb{G}_R^{t-\tau})S(\tau)d\tau \right\|_{L^{\infty}_{\xi,\beta+1}(L_x^r)}\lesssim\int_0^t\frac{1}{(1+t-\tau)^{\frac32(\frac{1}{p}-\frac{1}{r})}}\left\|  S(\tau) \right\|_{L^\infty_{\xi,\beta}(L_x^p)}d\tau,
\end{aligned}
\end{equation*}
and
\begin{equation*}
\begin{aligned}
\left\| \int_0^t\mathbb{G}_K^{t-\tau}S(\tau)d\tau \right\|_{L^{\infty}_{\xi,\beta+1}(L_x^r)}\lesssim\int_0^te^{-\nu_0(t-\tau)}\left\|  S(\tau) \right\|_{L^\infty_{\xi,\beta}(L_x^r)}d\tau.
\end{aligned}
\end{equation*}
The solution operator $\mathbb{G}^t$ satisfies that for $\beta>3/2$, 
\begin{equation*}
\begin{aligned}
\left\|  \int_0^t\mathbb{G}^{t-\tau}S(\tau)d\tau \right\|_{L^{\infty}_{\xi,\beta+1}(L_x^2)}\lesssim&\int_0^t \frac{1}{(1+t-\tau)^{\frac34}}\left\|  S(\tau) \right\|_{L^{\infty}_{\xi,\beta}(L_x^1)}\\
&+e^{-\nu_0(t-\tau)} (\left\|  S(\tau) \right\|_{L^{\infty}_{\xi,\beta}(L_x^2)}+\left\|  S(\tau) \right\|_{L^{2}_{\xi}(L_x^2)})d\tau,
\end{aligned}
\end{equation*}
and
\begin{equation*}
\begin{aligned}
\left\|  \int_0^t\mathbb{G}^{t-\tau}S(\tau)d\tau \right\|_{L^{\infty}_{\xi,\beta+1}(L_x^{\infty})}\lesssim&\int_0^t \frac{1}{(1+t-\tau)^{\frac32}}\left\|  S(\tau) \right\|_{L^{\infty}_{\xi,\beta}(L_x^1)}\\
&+e^{-\nu_0(t-\tau)} (\left\|  S(\tau) \right\|_{L^{\infty}_{\xi,\beta}(L_x^{\infty})}+\left\|  S(\tau) \right\|_{L^{2}_{\xi}(L_x^2)}+\left\|  S(\tau) \right\|_{L^{2}_{\xi}(L_x^{\infty})})d\tau.
\end{aligned}
\end{equation*}
 Moreover, if $\mathsf{P}_0S(x,t,\xi)=0$, then we will get extra $(1+t-\tau)^{-1/2}$ decay rate in each time integral above.
\end{lemma}

%\section{Linearized problem}
\section{The linearized problem}\label{Linearized problem}
In this section, we study the linearized Boltzmann equation with random uncertainties, originating from both the initial data and the collision kernels. Our focus is on analyzing the higher-order derivatives of the solution. Notably, while solutions in the whole space exhibit polynomial decay over time, those on the torus decay exponentially. The polynomial time decay of the solution presents significant challenges in analyzing higher-order derivatives.
%\subsection{Uncertainty of initial value}

\subsection{The solution operator for the linearized problem}
We first consider the solution operator $\mathbb{G}^t(z)$ for  problem (\ref{LUC}), i.e., the linearized Boltzmann equations with uncertainties.
 Under the influence of uncertainty, the semigroup $e^{\left( -i\xi\cdot \eta +\mathsf{L}\right) t}$ will transform into a new semigroup $e^{\left( -i\xi
\cdot \eta +\mathsf{L}^z\right) t}$. For any fixed uncertainty variable $z$, the Green's function method can still be applied to solve equation (\ref{LUC}). The impact of uncertainty on the large-time behavior is limited to a constant, determined by $C_{b_1}$ and $C_{b_2}$ in condition (\ref{b1}). Hence, when considering the large-time behavior of the solution, estimates can be obtained uniformly with respect to $z$. By Theorem \ref{UIf}, we have the following lemma:
\begin{lemma}\label{UCf}
$(\mathrm{I})$ Let $\mathbb{G}^t(z)$ be the solution operator for equation (\ref{LUC}). For given $C_{b_1}$ and $C_{b_2}$ in assumption (\ref{b1}), there exist positive constants $C_G$ and $C$ such that 
\begin{equation*}
\begin{aligned}
\left\|\mathbb{G}^t(z)g_0\right\|_{X^{\xi,2}_{x,r}}\le C_G\bigg[\frac{1}{(1+t)^{\frac32(\frac{1}{p}-\frac{1}{r})}}\left\|g_0\right\|_{X^{\xi,2}_{x,p}}
+e^{-t/C}(\left\|g_0\right\|_{X^{x,2}_{\xi,2}}+\left\|g_0\right\|_{X^{x,2}_{\xi,2}}^{2/r}\left\|g_0\right\|_{X^{x,\infty}_{\xi,2}}^{1-2/r})\bigg],
\end{aligned}
\end{equation*}
\begin{equation*}
\begin{aligned}
\left\|\mathbb{G}^t(z)g_0\right\|_{X^{x,r}_{\xi;\beta}}&\le C_G\bigg[\frac{1}{(1+t)^{\frac32(\frac{1}{p}-\frac{1}{r})}}\left\|g_0\right\|_{X^{\xi,2}_{x,p}}+e^{-t/C}(\left\|g_0\right\|_{X^{x,r}_{\xi;\beta}}+\left\|g_0\right\|_{X^{x,2}_{\xi,2}}+\left\|g_0\right\|_{X^{x,2}_{\xi,2}}^{2/r}\left\|g_0\right\|_{X^{x,\infty}_{\xi,2}}^{1-2/r})\bigg],
\end{aligned}
\end{equation*}
for $\beta\ge 0,2\le r\le \infty$ and $1\le p\le r$. Moreover, if $\mathsf{P}_0g_0=0$, then we will get extra $(1+t)^{-1/2}$ decay rate for each estimate above.

$(\mathrm{II})$ $\mathbb{G}^t(z)$ can be decomposed into three parts by
\begin{equation}\label{GC}
\begin{aligned}
\mathbb{G}^t(z)&=\mathbb{G}_F^t(z)+\mathbb{G}_K^t(z)+\mathbb{G}_R^t(z).
\end{aligned}
\end{equation}
Let
$$
\mathbb{G}_0^t(z):=\mathbb{G}_F^t(z)+\mathbb{G}_R^t(z).
$$
For any $1\le p\le r$, it holds that
\begin{equation*}
\begin{aligned}
\left\|\mathbb{G}_0^t(z)g_0\right\|_{X_{\xi;\beta}^{x,r}}\le\frac{C_G}{(1+t)^{\frac32(\frac{1}{p}-\frac{1}{r})}}\left\|g_0\right\|_{X_{\xi;\beta}^{x,p}},
\end{aligned}
\end{equation*}
and
\begin{equation*}
\begin{aligned}
\left\|\mathsf{P}_1\mathbb{G}_0^t(z)g_0\right\|_{X_{\xi;\beta}^{x,r}},\left\|\mathbb{G}_0^t(z)\mathsf{P}_1g_0\right\|_{X_{\xi;\beta}^{x,r}}\le\frac{C_G}{(1+t)^{\frac32(\frac{1}{p}-\frac{1}{r})+\frac12}}\left\|g_0\right\|_{X_{\xi;\beta}^{x,p}}.
\end{aligned}
\end{equation*}
For any $1\le r\le \infty$, it holds that
\begin{equation*}
\begin{aligned}
\left\|\mathbb{G}_K^t(z)g_0\right\|_{X_{\xi;\beta}^{x,r}}\le C_Ge^{-t/C}\left\|g_0\right\|_{X_{\xi;\beta}^{x,r}}.\\
\end{aligned}
\end{equation*}

In this lemma, all above constants are independent of $z$.
\end{lemma}
So by Lemma \ref{UCf}, a precise description of the large-time behavior of solutions to the linearized Boltzmann equation with uncertainties have been obtained.

%\subsection{Time decay of derivatives of the solutions}
\subsection{Time decay of derivatives of the solutions}\label{the solutions}
The estimates of the solution's derivatives and higher-order derivatives with respect to $z$ are key ingredients in this paper, as they are crucial for ensuring the accuracy of numerical methods.

 First, we investigate the first-order derivative of the solution for the linearized problem (\ref{LUC}). Let $\partial_zg_0,g_0\in X_{\xi;\beta}^{x,1}\cap X_{\xi;\beta}^{x,\infty},\beta>3/2$. By differentiating the equation (\ref{LUC}) with respect to $z$, we obtain equation \eqref{UCD}.
By Duhamel's principle, the derivative $\partial_zg$ satisfies 
\begin{equation*}
\partial_zg(x,t,\xi,z)=\mathbb{G}^{t}(z)\partial_zg_0+\int_0^t\mathbb{G}^{t-\tau}(z)(\partial_z\mathsf{L}^z)g(\tau)d\tau.
\end{equation*}
From \eqref{Lz} in Lemma \ref{bz}, we can see that the weight of $\xi$ in the estimate of $(\partial_z\mathsf{L}^z)g$ mismatch the weight of $\xi$ in the estimate of $g$.
 To address this issue, Lemma \ref{GTS} is needed. $\mathbb{G}_F^t(z), \mathbb{G}_K^t(z), \mathbb{G}_R^t(z)$ and $\mathbb{G}^t(z)$ also satisfy the properties in Lemma \ref{GTS}.

Following the idea in Section \ref{Main result}, we provide the following lemma to obtain $L_x^2$ and $L_x^\infty$ estimates for $\partial_zg$. The decomposition estimate for the Green's function, along with the extra time decay by the microscopic projection operator $\mathsf{P}_1$, are key ingredient of the proof.
\begin{lemma}\label{TC1}
If $g_0,\partial_z g_0 \in X_{\xi;\beta}^{x,1}\cap X_{\xi;\beta}^{x,\infty}$, $\beta>\frac32$, then for given $C_{b_*}$ in assumption (\ref{b2}), there exists a positive constant $C_{G,1}$ independent of $z$ such that the solution $\partial_zg$ of equation (\ref{UCD}) satisfies the following estimates:
 %$L_z^{\infty}L^{\infty}_{\xi,\beta}(L_x^{1}\cap L_x^{\infty})$
\begin{equation*}
\begin{aligned}
\left\|\partial_zg(t)\right\|_{X_{\xi;\beta}^{x,2}}&\le \frac{C_{G,1}}{(1+t)^{\frac34}}(\left\|g_0\right\|_{X_{\xi;\beta}^{x,1}}+\left\|g_0\right\|_{X_{\xi;\beta}^{x,2}}+\left\|\partial_zg_0\right\|_{X_{\xi;\beta}^{x,1}}+\left\|\partial_zg_0\right\|_{X_{\xi;\beta}^{x,2}}),\\
\left\|\partial_zg(t)\right\|_{X_{\xi;\beta}^{x,\infty}}&\le \frac{C_{G,1}}{(1+t)^{\frac32}}(\left\|g_0\right\|_{X_{\xi;\beta}^{x,1}}+\left\|g_0\right\|_{X_{\xi;\beta}^{x,\infty}}+\left\|\partial_zg_0\right\|_{X_{\xi;\beta}^{x,1}}+\left\|\partial_zg_0\right\|_{X_{\xi;\beta}^{x,\infty}}).
\end{aligned}
\end{equation*}
\end{lemma}

\begin{proof}
By Duhamel's principle,
\begin{equation*}
\partial_zg(x,t,\xi,z)=\mathbb{G}^{t}(z)\partial_zg_0+\int_0^t\mathbb{G}^{t-\tau}(z)(\partial_z\mathsf{L}^z)g(\tau)d\tau.
\end{equation*}
In view of the decompostion of solution operator $\mathbb{G}^t(z)$ in \eqref{GC}, the decompostion of $g(\tau)$ and $\mathsf{P}_1\partial_z\mathsf{L}^z\mathsf{P}_1=\partial_z\mathsf{L}^z$, we write
\begin{equation*}\label{FF1}
\int_0^t\mathbb{G}^{t-\tau}(z)(\partial_z\mathsf{L}^z)g(\tau)d\tau=g_{FF}+g_{FK}+g_{KF}+g_{KK},
\end{equation*}
where
%For instance, from the properties of Green's function, $f_{FF}$, $f_{FK}$, $f_{KF}$ and $f_{KK}$ can be expressed as
\begin{equation*}\label{FF2}
g_{FF}=\int_0^t\mathbb{G}_0^{t-\tau}\mathsf{P}_1[(\partial_z\mathsf{L}^z)\mathsf{P}_1\mathbb{G}_0^{\tau}g_0]d\tau,
\end{equation*}
\begin{equation*}
g_{FK}=\int_0^t\mathbb{G}_0^{t-\tau}\mathsf{P}_1[(\partial_z\mathsf{L}^z)\mathsf{P}_1\mathbb{G}_K^{\tau}g_0]d\tau,
\end{equation*}
\begin{equation*}
g_{KF}=\int_0^t\mathbb{G}_K^{t-\tau}\mathsf{P}_1[(\partial_z\mathsf{L}^z)\mathsf{P}_1\mathbb{G}_0^{\tau}g_0]d\tau,
\end{equation*}
\begin{equation*}
g_{KK}=\int_0^t\mathbb{G}_K^{t-\tau}\mathsf{P}_1[(\partial_z\mathsf{L}^z)\mathsf{P}_1\mathbb{G}_K^{\tau}g_0]d\tau.
\end{equation*}

Firstly, we consider $L^2_x$ estimate of $g_{FF}$. By Lemma \ref{GTS} and Lemma \ref{UCf}, we have 
\begin{equation*}
\begin{aligned}
&\left\|g_{FF}(t)\right\|_{X_{\xi;\beta}^{x,2}}\\
&=\left\|\int_0^t\mathbb{G}_0^{t-\tau}\mathsf{P}_1[(\partial_z\mathsf{L}^z)\mathsf{P}_1\mathbb{G}_0^{\tau}g_0]d\tau\right\|_{X_{\xi;\beta}^{x,2}}\\
&\le C_G\int_0^t(1+t-\tau)^{-\frac32(\frac{1}{r}-\frac{1}{2})-\frac12}\left\|(\partial_z\mathsf{L}^z)\mathsf{P}_1\mathbb{G}_0^{\tau}g_0 \right\|_{X_{\xi;\beta-1}^{x,r}}d\tau\\
&\le C_G^2C_\beta^*\int_0^t(1+t-\tau)^{-\frac32(\frac{1}{r}-\frac{1}{2})-\frac12}(1+\tau)^{-\frac32(1-\frac{1}{r})-\frac12} \left\| g_0\right\|_{X_{\xi;\beta}^{x,1}} d\tau.\\
\end{aligned}
\end{equation*}
Let $r=1$, and it follows
\begin{equation}\label{FF1}
\begin{aligned}
\int_0^{\frac{t}{2}}(1+t-\tau)^{-\frac32(\frac{1}{r}-\frac{1}{2})-\frac12}(1+\tau)^{-\frac32(1-\frac{1}{r})-\frac12} d\tau&=\int_0^{\frac{t}{2}}(1+t-\tau)^{-\frac{5}{4}}(1+\tau)^{-\frac12} d\tau \lesssim(1+t)^{-\frac34}.
\end{aligned}
\end{equation}
Let $r=2$, and it follows
\begin{equation}\label{FF2}
\begin{aligned}
\int^t_{\frac{t}{2}}(1+t-\tau)^{-\frac32(\frac{1}{r}-\frac{1}{2})-\frac12}(1+\tau)^{-\frac32(1-\frac{1}{r})-\frac12} d\tau&=\int^t_{\frac{t}{2}}(1+t-\tau)^{-\frac{1}{2}}(1+\tau)^{-\frac54} d\tau \lesssim(1+t)^{-\frac34}.
\end{aligned}
\end{equation}
Combining \eqref{FF1} and \eqref{FF2}, we have
\begin{equation}\label{TC11}
\left\|g_{FF}(t)\right\|_{X_{\xi;\beta}^{x,2}}\le CC_G^2C_\beta^*\frac{1}{(1+t)^{\frac34}}\left\| g_0\right\|_{X_{\xi;\beta}^{x,1}}.
\end{equation}

Then we consider $L^2_x$ estimate of $g_{FK}$. Note that the time decay of particle wave operator $\mathbb{G}_K^t$ is exponential. 
By Lemma \ref{GTS} and Lemma \ref{UCf}, it follows that
\begin{equation}\label{TC12}
\begin{aligned}
&\left\|g_{FK}(t)\right\|_{X_{\xi;\beta}^{x,2}}\\
&=\left\|\int_0^t\mathbb{G}_0^{t-\tau}\mathsf{P}_1[(\partial_z\mathsf{L}^z)\mathsf{P}_1\mathbb{G}_K^{\tau}g_0]d\tau\right\|_{X_{\xi;\beta}^{x,2}}\\
&\lesssim\int_0^t(1+t-\tau)^{-\frac32(\frac{1}{1}-\frac{1}{2})-\frac12}\left\|(\partial_z\mathsf{L}^z)\mathsf{P}_1\mathbb{G}_K^{\tau}g_0 \right\|_{X_{\xi;\beta-1}^{x,1}}d\tau\\
&\lesssim\int_0^t(1+t-\tau)^{-\frac32(1-\frac{1}{2})-\frac12}e^{-\tau/C} \left\|g_0\right\|_{X_{\xi;\beta}^{x,1}} d\tau\\
&\le C C_G^2C_\beta^*\frac{1}{(1+t)^{\frac54}}\left\| g_0\right\|_{X_{\xi;\beta}^{x,1}}.
\end{aligned}
\end{equation}

 For $L^2_x$ estimates of $g_{KK}$ and $g_{KF}$, it follows
\begin{equation}\label{TC13}
\begin{aligned}
&\left\|g_{KK}(t)\right\|_{X_{\xi;\beta}^{x,2}}\\
&=\left\|\int_0^t\mathbb{G}_K^{t-\tau}\mathsf{P}_1[(\partial_z\mathsf{L}^z)\mathsf{P}_1\mathbb{G}_K^{\tau}g_0]d\tau\right\|_{X_{\xi;\beta}^{x,2}}\\
&\lesssim\int_0^te^{-\frac{t-\tau}{C}}\left\|(\partial_z\mathsf{L}^z)\mathsf{P}_1\mathbb{G}_K^{\tau}g_0 \right\|_{X_{\xi;\beta-1}^{x,2}}d\tau\\
&\lesssim\int_0^t e^{-\frac{t-\tau}{C}}e^{-\tau/C}\left\|g_0\right\|_{X_{\xi;\beta}^{x,2}} d\tau\\
&\lesssim e^{-t/C}\left\|g_0\right\|_{X_{\xi;\beta}^{x,2}},
\end{aligned}
\end{equation}
and
\begin{equation}\label{TC14}
\begin{aligned}
&\left\|g_{KF}(t)\right\|_{X_{\xi;\beta}^{x,2}}\\
&=\left\|\int_0^t\mathbb{G}_K^{t-\tau}\mathsf{P}_1[(\partial_z\mathsf{L}^z)\mathsf{P}_1\mathbb{G}_0^{\tau}g_0]d\tau\right\|_{X_{\xi;\beta}^{x,2}}\\
&\lesssim\int_0^te^{-\frac{t-\tau}{C}}\left\|(\partial_z\mathsf{L}^z)\mathsf{P}_1\mathbb{G}_0^{\tau}g_0 \right\|_{X_{\xi;\beta-1}^{x,2}}d\tau\\
&\lesssim\int_0^t e^{-\frac{t-\tau}{C}}(1+\tau)^{-\frac34-\frac12}\left\|g_0\right\|_{X_{\xi;\beta}^{x,1}} d\tau\\
&\le C C_G^2C_\beta^*\frac{1}{(1+t)^{\frac54}}\left\| g_0\right\|_{X_{\xi;\beta}^{x,1}}.
\end{aligned}
\end{equation}
Combining \eqref{TC11}-\eqref{TC14}, we have
\begin{equation}\label{p1d}
\begin{aligned}
\left\|\int_0^t\mathbb{G}^{t-\tau}(z)(\partial_z\mathsf{L}^z)g(\tau)d\tau\right\|_{X_{\xi;\beta}^{x,2}}&\le \frac{C_{G,1}}{(1+t)^{\frac34}}(\left\| g_0\right\|_{X_{\xi;\beta}^{x,1}}+\left\| g_0\right\|_{X_{\xi;\beta}^{x,2}}).\\
%&\le \frac{C_{G,1}}{(1+t)^{\frac34}}(\left\| g_0\right\|_{X_{\xi;\beta}^{x,1}}+\left\| g_0\right\|_{X_{\xi;\beta}^{x,\infty}}).
\end{aligned}
\end{equation}
Combining \eqref{p1d} and the estimate of $\mathbb{G}^{t}(z)\partial_zg_0$, we have 
\begin{equation}\label{2TC1}
\left\|\partial_zg(t)\right\|_{X_{\xi;\beta}^{x,2}}\le \frac{C_{G,1}}{(1+t)^{\frac34}}(\left\|g_0\right\|_{X_{\xi;\beta}^{x,1}}+\left\|g_0\right\|_{X_{\xi;\beta}^{x,2}}+\left\|\partial_zg_0\right\|_{X_{\xi;\beta}^{x,1}}+\left\|\partial_zg_0\right\|_{X_{\xi;\beta}^{x,2}}).
\end{equation}

 As for the $X_{\xi;\beta}^{x,\infty}$ estimate, one needs to change the index in the above procedure of $g_{FF}$, $g_{KF}$ and $g_{FK}$ from $-\frac32(\frac{1}{r}-\frac{1}{2})-\frac12$ to $-\frac32(\frac{1}{r}-\frac{1}{\infty})-\frac12$, so we omit it. And for the $L_x^\infty$ estimates of $g_{KK}$, it follows 
\begin{equation*}\label{TC1KK}
\begin{aligned}
\left\|g_{KK}(t)\right\|_{X_{\xi;\beta}^{x,\infty}}\lesssim e^{-t/C}\left\|g_0\right\|_{X_{\xi;\beta}^{x,\infty}}.
\end{aligned}
\end{equation*}
Therefore, we have 
\begin{equation}\label{ITC1}
\left\|\partial_zg(t)\right\|_{X_{\xi;\beta}^{x,\infty}}\le \frac{C_{G,1}}{(1+t)^{\frac32}}(\left\|g_0\right\|_{X_{\xi;\beta}^{x,1}}+\left\|g_0\right\|_{X_{\xi;\beta}^{x,\infty}}+\left\|\partial_zg_0\right\|_{X_{\xi;\beta}^{x,1}}+\left\|\partial_zg_0\right\|_{X_{\xi;\beta}^{x,\infty}}).
\end{equation}

Combining \eqref{2TC1} and \eqref{ITC1}, we complete the proof.
\end{proof}
\begin{remark}
If we consider the initial value problem on $\mathbb{T}^3$, from \cite{ukai1974existence}, the solution of the initial value problem decays exponentially with respect to $t$. By following the above procedure, one can straightforwardly obtain the exponential decay of $\partial_zg$ with respect to $t$. Similarly, we can obtain the estimates of higher-order derivatives on the $\mathbb{T}^3$, and reach the same conclusion as in \cite{liu2018hypocoercivity}.
\end{remark}

%\subsection{Time decay of higher derivatives}
\subsection{Time decay of higher derivatives}\label{higher derivatives}
Next, we consider higher-order derivatives. 
With the results in Lemma \ref{UCf} and Lemma \ref{TC1}, we can give the following lemma for higher-order derivatives:
\begin{lemma}\label{TC2}
If $\partial^k_zg_0 \in X^{x,1}_{\xi;\beta} \cap X^{x,\infty}_{\xi;\beta}$, $\beta>\frac32$, %and $|\partial^k_zb(\eta,z)| \le C_{b_*}$, 
 for $0\le k \le \alpha$, then for given $\alpha\in \mathbb{N}_+$, there exists positive constants $C_{G,k}, 2\le k \le \alpha$ independent of $z$ such that the solution of equation (\ref{LUC}) satisfies the following estimates:
\begin{equation*}
\begin{aligned}
\left\|\partial^k_zg(t)\right\|_{X_{\xi;\beta}^{x,2}}&\le C_{G,k}\bigg[\frac{1}{(1+t)^{\frac34}}(\left\|\partial^k_zg_0\right\|_{X_{\xi;\beta}^{x,1}}+\left\|\partial^k_zg_0\right\|_{X_{\xi;\beta}^{x,2}})\\
&\ \ +\sum_{s=0}^{k-1}\frac{(\ln(1+t))^{k-1-s}}{(1+t)^{\frac34}}(\left\|\partial^s_zg_0\right\|_{X_{\xi;\beta}^{x,1}}+\left\|\partial^s_zg_0\right\|_{X_{\xi;\beta}^{x,2}})\bigg],\\
\left\|\partial^k_zg(t)\right\|_{X_{\xi;\beta}^{x,\infty}}&\le C_{G,k}\bigg[\frac{1}{(1+t)^{\frac32}}(\left\|\partial^k_zg_0\right\|_{X_{\xi;\beta}^{x,1}}+\left\|\partial^k_zg_0\right\|_{X_{\xi;\beta}^{x,\infty}})\\
&\ \ +\sum_{s=0}^{k-1}\frac{(\ln(1+t))^{k-1-s}}{(1+t)^{\frac32}}(\left\|\partial^s_zg_0\right\|_{X_{\xi;\beta}^{x,1}}+\left\|\partial^s_zg_0\right\|_{X_{\xi;\beta}^{x,\infty}})\bigg],\\
\end{aligned}
\end{equation*}
for $2\le k \le \alpha$.
\end{lemma}
\begin{proof}
Below we will take $\partial_z^2g$ as an example to illustrate our method, and the estimates for other higher-order derivatives are similar. By differentiating the equation (\ref{UCD}), $\partial_z^2g$ satisfies that 
\begin{equation*}
\partial_t(\partial^2_zg)+\xi\cdot \nabla_x(\partial^2_zg)=2(\partial_z\mathsf{L}^z)(\partial_zg)+(\partial^2_z\mathsf{L}^z)g+\mathsf{L}^z(\partial^2_zg).
\end{equation*}
By Duhamel's principle, it follows
\begin{equation*}
\partial^2_zg(x,t,\xi,z)=\mathbb{G}^{t}(z)\partial^2_zg_0+\int_0^t\mathbb{G}^{t-\tau}(z)[2(\partial_z\mathsf{L}^z)(\partial_zg)+(\partial^2_z\mathsf{L}^z)g](\tau)d\tau.
\end{equation*}

By Lemma \ref{UCf}, the estimates for $\mathbb{G}^{t}(z)\partial^2_zg_0$ can be obtained.
Following the same procedure as the estimate for $\partial_zg$ in Lemma \ref{TC1}, we can obtain the estimate with $(\partial^2_z\mathsf{L}^z)g$ as the source. From Lemma \ref{TC1}, $\partial_zg$ satisfies
\begin{equation*}
\partial_zg=\mathbb{G}^{t}(z)\partial_zg_0+\int_0^t\mathbb{G}^{t-\tau}(z)(\partial_z\mathsf{L}^z)g(\tau)d\tau.
\end{equation*}
The estimates for the part
$$
\int_0^t\mathbb{G}^{t-\tau}(z)(\partial_z\mathsf{L}^z)\mathbb{G}^{\tau}(z)\partial_zg_0d\tau
$$
can be obtained following the exactly same procedure as in Lemma \ref{TC1}. Let
$$
g_R:=\mathbb{G}^{t}(z)\partial^2_zg_0+\int_0^t\mathbb{G}^{t-\tau}(z)[(\partial^2_z\mathsf{L}^z)g](\tau)d\tau+\int_0^t\mathbb{G}^{t-\tau}(z)(\partial_z\mathsf{L}^z)\mathbb{G}^{\tau}(z)\partial_zg_0d\tau,
$$
and we have 
\begin{equation}\label{p2r2}
\left\|g_R(t)\right\|_{X_{\xi;\beta}^{x,2}}\lesssim \frac{1}{(1+t)^{\frac34}}(\sum_{s=0}^{2}\left\|\partial^s_zg_0\right\|_{X_{\xi;\beta}^{x,1}}+\left\|\partial^s_zg_0\right\|_{X_{\xi;\beta}^{x,2}}),
\end{equation}
and 
\begin{equation*}\label{p2ri}
\left\|g_R(t)\right\|_{X_{\xi;\beta}^{x,\infty}}\lesssim \frac{1}{(1+t)^{\frac32}}(\sum_{s=0}^{2}\left\|\partial^s_zg_0\right\|_{X_{\xi;\beta}^{x,1}}+\left\|\partial^s_zg_0\right\|_{X_{\xi;\beta}^{x,\infty}}).
\end{equation*}

Therefore, we only need to consider the remaining part.
Define
\begin{equation*}\label{FFF1}
\int_0^t\int_{0}^{\tau_1}\mathbb{G}^{t-\tau_1}(z)(\partial_z\mathsf{L}^z)\mathbb{G}^{\tau_1-\tau}(z)(\partial_z\mathsf{L}^z)g(\tau)d\tau d\tau_1=:\sum_{i,j,m \in\{F,K\}}g_{ijm}.
\end{equation*}
%Next we use $1, 2$ and $3$ to represent the operator of $\mathbb{G}_{L;0}^{t}(z)$, $\mathbb{G}_{L;\bot}^{t}(z)$ and $\mathbb{G}_{S}^{t}(z)$, respectively.
For example, from the decompostion of solution operator $\mathbb{G}^t(z)$ in \eqref{GC}, $g_{FFF}$, $g_{FKF}$ and $g_{FFK}$ can be expressed as
\begin{equation*}\label{FFF2}
g_{FFF}=\int_0^t\int_{0}^{\tau_1}\mathbb{G}_{0}^{t-\tau_1}\mathsf{P}_1(\partial_z\mathsf{L}^z\mathsf{P}_1\mathbb{G}_{0}^{\tau_1-\tau}\mathsf{P}_1)(\partial_z\mathsf{L}^z\mathsf{P}_1\mathbb{G}_{0}^{\tau}g_0)(\tau)d\tau d\tau_1,
\end{equation*}
\begin{equation*}
g_{FKF}=\int_0^t\int_0^{\tau_1}\mathbb{G}_{0}^{t-\tau_1}\mathsf{P}_1(\partial_z\mathsf{L}^z\mathsf{P}_1\mathbb{G}_{K}^{\tau_1-\tau}\mathsf{P}_1)(\partial_z\mathsf{L}^z\mathsf{P}_1\mathbb{G}_0^{\tau}g_0)(\tau)d\tau d\tau_1,
\end{equation*}
and 
\begin{equation*}
g_{KKK}=\int_0^t\int_0^{\tau_1}\mathbb{G}_{K}^{t-\tau_1}\mathsf{P}_1(\partial_z\mathsf{L}^z\mathsf{P}_1\mathbb{G}_{K}^{\tau_1-\tau}\mathsf{P}_1)(\partial_z\mathsf{L}^z\mathsf{P}_1\mathbb{G}_K^{\tau}g_0)(\tau)d\tau d\tau_1.
\end{equation*}
The remaining items can be defined similarly.

Next, we consider $L^2_x$ estimate of $g_{FFF}$. By Lemma \ref{GTS} and  Lemma \ref{UCf}, for $1\le p\le r\le 2$, we have 
\begin{equation*}
\begin{aligned}
&\left\|g_{FFF}(t)\right\|_{X_{\xi;\beta}^{x,2}}\\
&=\left\|\int_0^t\int_{0}^{\tau_1}\mathbb{G}_{0}^{t-\tau_1}\mathsf{P}_1(\partial_z\mathsf{L}^z\mathsf{P}_1\mathbb{G}_{0}^{\tau_1-\tau}\mathsf{P}_1)(\partial_z\mathsf{L}^z\mathsf{P}_1\mathbb{G}_{0}^{\tau}g_0)d\tau d\tau_1\right\|_{X_{\xi;\beta}^{x,2}}\\
&\le C_G\int_0^t\int_0^{\tau_1}(1+t-\tau_1)^{-\frac32(\frac{1}{r}-\frac{1}{2})-\frac12}\left\|(\partial_z\mathsf{L}^z\mathsf{P}_1\mathbb{G}_{0}^{\tau_1-\tau}\mathsf{P}_1)(\partial_z\mathsf{L}^z\mathsf{P}_1\mathbb{G}_{0}^{\tau}g_0) \right\|_{X_{\xi;\beta-1}^{x,r}}d\tau d\tau_1\\
&\le C_G^2C_\beta^*\int_0^t\int_0^{\tau_1}(1+t-\tau_1)^{-\frac32(\frac{1}{r}-\frac{1}{2})-\frac12}(1+\tau_1-\tau)^{-\frac32(\frac{1}{p}-\frac{1}{r})-1}\left\|(\partial_z\mathsf{L}^z\mathsf{P}_1\mathbb{G}_{0}^{\tau}g_0)\right\|_{X_{\xi;\beta-1}^{x,p}}d\tau d\tau_1\\
&\le C_G^3(C_\beta^*)^2\int_0^t\int_0^{\tau_1}(1+t-\tau_1)^{-\frac32(\frac{1}{r}-\frac{1}{2})-\frac12}(1+\tau_1-\tau)^{-\frac32(\frac{1}{p}-\frac{1}{r})-1}(1+\tau)^{-\frac32(1-\frac{1}{p})-\frac12} \left\|g_0\right\|_{X_{\xi;\beta}^{x,1}}d\tau d\tau_1\\
&=C_G^3(C_\beta^*)^2\left\|g_0\right\|_{X_{\xi;\beta}^{x,1}}\bigg(\int^{\frac{t}{2}}_0\int^{\frac{\tau_1}{2}}_0+\int^{\frac{t}{2}}_0\int_{\frac{\tau_1}{2}}^{\tau_1}+\int_{\frac{t}{2}}^t\int^{\frac{\tau_1}{2}}_0+\int_{\frac{t}{2}}^t\int_{\frac{\tau_1}{2}}^{\tau_1}\bigg)\\
&=:C_G^3(C_\beta^*)^2\left\|g_0\right\|_{X_{\xi;\beta}^{x,1}}(I_1+I_2+I_3+I_4).\\
%&\le C_G^3(C_b^*)^2\frac{\ln^2(1+t)}{(1+t)^{\frac34}}\left\|\nu^2f_0(z)\right\|_{ L_x^1(L^2_{\xi})}.
\end{aligned}
\end{equation*}
Let $r=p=1$, then $I_1$ and $I_2$ satisfy that
\begin{equation*}
\begin{aligned}
I_1=\int^{\frac{t}{2}}_0\int^{\frac{\tau_1}{2}}_0(1+t-\tau_1)^{-\frac32(\frac{1}{1}-\frac{1}{2})-\frac12}(1+\tau_1-\tau)^{-\frac32(\frac{1}{1}-\frac{1}{1})-1}(1+\tau)^{-\frac32(1-\frac{1}{1})-\frac12}d\tau d\tau_1 \lesssim \frac{1}{(1+t)^{\frac34}},
\end{aligned}
\end{equation*}
and
\begin{equation*}
\begin{aligned}
I_2=\int^{\frac{t}{2}}_0\int_{\frac{\tau_1}{2}}^{\tau_1}(1+t-\tau_1)^{-\frac32(\frac{1}{1}-\frac{1}{2})-\frac12}(1+\tau_1-\tau)^{-\frac32(\frac{1}{1}-\frac{1}{1})-1}(1+\tau)^{-\frac32(1-\frac{1}{1})-\frac12}d\tau d\tau_1 \lesssim \frac{\ln(1+t)}{(1+t)^{\frac34}}.
\end{aligned}
\end{equation*}
Let $r=p=2$, then $I_3$ and $I_4$ satisfy that
\begin{equation*}
\begin{aligned}
I_3=\int_{\frac{t}{2}}^t\int^{\frac{\tau_1}{2}}_0(1+t-\tau_1)^{-\frac32(\frac{1}{2}-\frac{1}{2})-\frac12}(1+\tau_1-\tau)^{-\frac32(\frac{1}{2}-\frac{1}{2})-1}(1+\tau)^{-\frac32(1-\frac{1}{2})-\frac12}d\tau d\tau_1 \lesssim \frac{1}{(1+t)^{\frac34}},
\end{aligned}
\end{equation*}
and
\begin{equation*}
\begin{aligned}
I_4=\int_{\frac{t}{2}}^t\int_{\frac{\tau_1}{2}}^{\tau_1}(1+t-\tau_1)^{-\frac32(\frac{1}{2}-\frac{1}{2})-\frac12}(1+\tau_1-\tau)^{-\frac32(\frac{1}{2}-\frac{1}{2})-1}(1+\tau)^{-\frac32(1-\frac{1}{2})-\frac12}d\tau d\tau_1 \lesssim \frac{\ln(1+t)}{(1+t)^{\frac34}}.
\end{aligned}
\end{equation*}
Therefore, we have
\begin{equation}\label{p2a}
\left\|g_{FFF}(t)\right\|_{X_{\xi;\beta}^{x,2}}\lesssim \frac{\ln(1+t)}{(1+t)^{\frac34}}\left\|g_0\right\|_{X_{\xi;\beta}^{x,1}}.
\end{equation}

Next, we present the $L^2_x$ estimates for $g_{FKF}$ and $g_{KKK}$; the remaining estimates follow in the same way. By Lemma \ref{GTS} and  Lemma \ref{UCf}, for $g_{FKF}$, it follows 
\begin{equation*}
\begin{aligned}
&\left\|g_{FKF}(t)\right\|_{X_{\xi;\beta}^{x,2}}\\
&=\left\|\int_0^t\int_{0}^{\tau_1}\mathbb{G}_{0}^{t-\tau_1}\mathsf{P}_1(\partial_z\mathsf{L}^z\mathsf{P}_1\mathbb{G}_{K}^{\tau_1-\tau}\mathsf{P}_1)(\partial_z\mathsf{L}^z\mathsf{P}_1\mathbb{G}_{0}^{\tau}g_0)d\tau d\tau_1\right\|_{X_{\xi;\beta}^{x,2}}\\
&\le C_G\int_0^t\int_0^{\tau_1}(1+t-\tau_1)^{-\frac32(\frac{1}{r}-\frac{1}{2})-\frac12}\left\|(\partial_z\mathsf{L}^z\mathsf{P}_1\mathbb{G}_{K}^{\tau_1-\tau}\mathsf{P}_1)(\partial_z\mathsf{L}^z\mathsf{P}_1\mathbb{G}_{0}^{\tau}g_0) \right\|_{X_{\xi;\beta-1}^{x,r}}d\tau d\tau_1\\
&\le C_G^2C_\beta^*\int_0^t\int_0^{\tau_1}(1+t-\tau_1)^{-\frac32(\frac{1}{r}-\frac{1}{2})-\frac12}e^{-\frac{\tau_1-\tau}{C}}\left\|(\partial_z\mathsf{L}^z\mathsf{P}_1\mathbb{G}_{0}^{\tau}g_0)\right\|_{X_{\xi;\beta-1}^{x,r}}d\tau d\tau_1\\
&\le C_G^3(C_\beta^*)^2\int_0^t\int_0^{\tau_1}(1+t-\tau_1)^{-\frac32(\frac{1}{r}-\frac{1}{2})-\frac12}e^{-\frac{\tau_1-\tau}{C}}(1+\tau)^{-\frac32(1-\frac{1}{r})-\frac12} \left\|g_0\right\|_{X_{\xi;\beta}^{x,1}}d\tau d\tau_1\\
&=C_G^3(C_\beta^*)^2\left\|g_0\right\|_{X_{\xi;\beta}^{x,1}}\bigg(\int^{\frac{t}{2}}_0\int_{0}^{\tau_1}+\int_{\frac{t}{2}}^t\int_{0}^{\tau_1}\bigg):=C_G^3(C_\beta^*)^2\left\|g_0\right\|_{X_{\xi;\beta}^{x,1}}(I_5+I_6).\\
%&\le C_G^3(C_b^*)^2\frac{\ln^2(1+t)}{(1+t)^{\frac34}}\left\|\nu^2f_0(z)\right\|_{ L_x^1(L^2_{\xi})}.
\end{aligned}
\end{equation*}
Let $r=1$, then $I_5$ satisfies that
$$
\int_{\frac{t}{2}}^t\int_0^{\tau_1}(1+t-\tau_1)^{-\frac32(\frac{1}{1}-\frac{1}{2})-\frac12}e^{-\frac{\tau_1-\tau}{C}}(1+\tau)^{-\frac32(1-\frac{1}{1})-\frac12}d\tau d\tau_1\lesssim \frac{1}{(1+t)^{\frac34}}.
$$
Let $r=2$, then $I_6$ satisfies that
$$
\int_{\frac{t}{2}}^t\int_0^{\tau_1}(1+t-\tau_1)^{-\frac32(\frac{1}{2}-\frac{1}{2})-\frac12}e^{-\frac{\tau_1-\tau}{C}}(1+\tau)^{-\frac32(1-\frac{1}{2})-\frac12}d\tau d\tau_1\lesssim \frac{1}{(1+t)^{\frac34}}.
$$
Therefore, we have
\begin{equation}\label{p2b}
\left\|g_{FKF}(t)\right\|_{X_{\xi;\beta}^{x,2}}\lesssim \frac{1}{(1+t)^{\frac34}}\left\|g_0\right\|_{X_{\xi;\beta}^{x,1}}.
\end{equation}
For the $L^2_x$ estimate of $g_{KKK}$, it follows
\begin{equation}\label{p2c}
\begin{aligned}
&\left\|g_{KKK}(t)\right\|_{X_{\xi;\beta}^{x,2}}\\
&=\left\|\int_0^t\int_{0}^{\tau_1}\mathbb{G}_{K}^{t-\tau_1}\mathsf{P}_1(\partial_z\mathsf{L}^z\mathsf{P}_1\mathbb{G}_{K}^{\tau_1-\tau}\mathsf{P}_1)(\partial_z\mathsf{L}^z\mathsf{P}_1\mathbb{G}_{K}^{\tau}g_0)d\tau d\tau_1\right\|_{X_{\xi;\beta}^{x,2}}\\
&\le C_G\int_0^t\int_0^{\tau_1}e^{-\frac{t-\tau_1}{C}}\left\|(\partial_z\mathsf{L}^z\mathsf{P}_1\mathbb{G}_{K}^{\tau_1-\tau}\mathsf{P}_1)(\partial_z\mathsf{L}^z\mathsf{P}_1\mathbb{G}_{K}^{\tau}g_0) \right\|_{X_{\xi;\beta-1}^{x,2}}d\tau d\tau_1\\
&\le C_G^2C_\beta^*\int_0^t\int_0^{\tau_1}e^{-\frac{t-\tau_1}{C}}e^{-\frac{\tau_1-\tau}{C}}\left\|(\partial_z\mathsf{L}^z\mathsf{P}_1\mathbb{G}_{K}^{\tau}g_0)\right\|_{X_{\xi;\beta-1}^{x,2}}d\tau d\tau_1\\
&\le C_G^3(C_\beta^*)^2\int_0^t\int_0^{\tau_1}e^{-\frac{t-\tau_1}{C}}e^{-\frac{\tau_1-\tau}{C}}e^{-\tau/C} \left\|g_0\right\|_{X_{\xi;\beta}^{x,2}}d\tau d\tau_1\\
&\lesssim e^{-t/C}\left\|g_0\right\|_{X_{\xi;\beta}^{x,2}}.%\lesssim e^{-t/C}(\left\|g_0\right\|_{X_{\xi;\beta}^{x,1}}+\left\|g_0\right\|_{X_{\xi;\beta}^{x,\infty}}).
\end{aligned}
\end{equation}
Combining \eqref{p2r2}, \eqref{p2a}, \eqref{p2b} and \eqref{p2c}, together with the remaining estimates, we have
\begin{equation*}
\begin{aligned}
\left\|\partial^2_zg(t)\right\|_{X_{\xi;\beta}^{x,2}}&\le C_{G,2}\bigg[\frac{1}{(1+t)^{\frac34}}(\left\|\partial^2_zg_0\right\|_{X_{\xi;\beta}^{x,1}}+\left\|\partial^k_zg_0\right\|_{X_{\xi;\beta}^{x,2}})\\
&\ \ +\sum_{s=0}^{1}\frac{(\ln(1+t))^{1-s}}{(1+t)^{\frac34}}(\left\|\partial^s_zg_0\right\|_{X_{\xi;\beta}^{x,1}}+\left\|\partial^s_zg_0\right\|_{X_{\xi;\beta}^{x,2}})\bigg].\\
\end{aligned}
\end{equation*}

 As for the $X_{\xi;\beta}^{x,\infty}$ estimate, similarly to Lemma \ref{TC1}, we still change the index in the above procedure from $-\frac32(\frac{1}{r}-\frac{1}{2})-\frac12$ to $-\frac32(\frac{1}{r}-\frac{1}{\infty})-\frac12$, so we omit it.
 This completes the proof.
\end{proof}

%\section{onlinear problem}
\section{The nonlinear problem}\label{Nonlinear problem}
In this section, we examine the initial value problem for the nonlinear Boltzmann equation and provide the proof of Theorem \ref{T1}. We will propose an appropriate ansatz for the solution and close the nonlinear problem (\ref{UC}).

\subsection{Solutions to the nonlinear problems}\label{nonlinear problems}
 Consider the nonlinear Boltzmann equation (\ref{UC}), and let $\partial^k_zf_0\in X^{x,1}_{\xi;\beta}\cap X^{x,\infty}_{\xi;\beta}$, for $0\le k\le\alpha, \beta>3/2$. By Duhamel's principle, the solution of (\ref{UC}) satisfies the integal equation
\begin{equation*}
f(x,t,\xi,z)=\mathbb{G}^t(z)f_0+\int_0^t\mathbb{G}^{t-\tau}(z)\Gamma^z(f,f)(\tau)d\tau.
\end{equation*}

Following the approach in Section 10 of \cite{liu2004green}, we apply the fixed point method to show the existence of the solution. The procedure consists of
two steps: the first step is to show the nonlinear map is an into, and the second one is to prove
the map is contractive. Since they are analogous, here we mainly show the first one.
From the linearized problem and Lemma \ref{UCf}, we have
\begin{equation*}
\begin{aligned}
\left\|\mathbb{G}^t(z)f_0\right\|_{X_{\xi;\beta}^{x,2}}&\le \frac{C_G}{(1+t)^{\frac34}}(\left\|f_0\right\|_{X_{\xi;\beta}^{x,1}}+\left\|f_0\right\|_{X_{\xi;\beta}^{x,2}}),\\
\left\|\mathbb{G}^t(z)f_0\right\|_{X_{\xi;\beta}^{x,\infty}}&\le \frac{C_G}{(1+t)^{\frac34}}(\left\|f_0\right\|_{X_{\xi;\beta}^{x,1}}+\left\|f_0\right\|_{X_{\xi;\beta}^{x,\infty}}).
\end{aligned}
\end{equation*}
This motivates us to make the following ansatz:
\begin{equation}\label{ansatzINL}
 \begin{cases}
\left\|f(t)\right\|_{X_{\xi;\beta}^{x,2}}\le\frac{10C_G}{(1+t)^{\frac34}}(\left\|f_0\right\|_{X_{\xi;\beta}^{x,1}}+\left\|f_0\right\|_{X_{\xi;\beta}^{x,\infty}}),\\
\left\|f(t)\right\|_{X_{\xi;\beta}^{x,\infty}}\le\frac{10C_G}{(1+t)^{\frac32}}(\left\|f_0\right\|_{X_{\xi;\beta}^{x,1}}+\left\|f_0\right\|_{X_{\xi;\beta}^{x,\infty}}).
\end{cases}
\end{equation}

Under ansatz \eqref{ansatzINL}, we are mainly concerned with the nonlinear part. %We must firstly consider the $L^1_x$ estimate and $L^{\infty}_x$ estimate. 
From \eqref{gamma} in Lemma \ref{bz}, we can see that the weight of $\xi$ in the estimate of nonlinear operator $\Gamma^z$ mismatch the weight of $\xi$ in the estimate of $h$ and $u$.
 To address this issue, we need  Lemma \ref{GTS}, and $\mathbb{G}^t(z)$ has the same properties as $\mathbb{G}^t$.

Now we close the ansatz (\ref{ansatzINL}). For the $L^2_x$ estimate, it follows
\begin{equation*}
\begin{aligned}
\left\|f(t)\right\|_{X_{\xi;\beta}^{x,2}}\le \left\|\mathbb{G}^t(z)f_0\right\|_{X_{\xi;\beta}^{x,2}}+\left\|\int_0^t\mathbb{G}^{t-\tau}(z)\Gamma^z(f,f)(\tau)d\tau\right\|_{X_{\xi;\beta}^{x,2}}.\\
\end{aligned}
\end{equation*}
And by  Lemma \ref{GTS}, Lemma \ref{UCf} and (\ref{gamma}), the nonlinear term satisfies 
\begin{equation}\label{4.11}
\begin{aligned}
\left\|  \int_0^t\mathbb{G}^{t-\tau}(z)\Gamma^z(f,f)(\tau)d\tau \right\|_{X_{\xi;\beta}^{x,2}}&\lesssim\int_0^t \frac{1}{(1+t-\tau)^{\frac34+\frac12}}\left\|\Gamma^z(f,f)(\tau)\right\|_{X_{\xi;\beta-1}^{x,1}}+e^{-\frac{t-\tau}{C}} \left\| \Gamma^z(f,f)(\tau) \right\|_{X_{\xi;\beta-1}^{x,2}}d\tau\\
&\lesssim\int_0^t \frac{1}{(1+t-\tau)^{\frac34+\frac12}}\left\|f \right\|_{X_{\xi;\beta}^{x,2}}\left\|f \right\|_{X_{\xi;\beta}^{x,2}}+e^{-\frac{t-\tau}{C}}\left\|f \right\|_{X_{\xi;\beta}^{x,2}}\left\|f \right\|_{X_{\xi;\beta}^{x,\infty}}d\tau\\
&\lesssim\int_0^t \frac{1}{(1+t-\tau)^{\frac34+\frac12}}\frac{1}{(1+\tau)^{\frac32}}(\left\|f_0\right\|_{X_{\xi;\beta}^{x,1}}+\left\|f_0\right\|_{X_{\xi;\beta}^{x,\infty}})^2d\tau\\
&\lesssim\frac{1}{(1+t)^{\frac54}}(\left\|f_0\right\|_{X_{\xi;\beta}^{x,1}}+\left\|f_0\right\|_{X_{\xi;\beta}^{x,\infty}})^2.
\end{aligned}
\end{equation}
For the $L^{\infty}_x$ estimate, we have 
\begin{equation}\label{4.12}
\begin{aligned}
\left\|  \int_0^t\mathbb{G}^{t-\tau}(z)\Gamma^z(f,f)(\tau)d\tau \right\|_{X_{\xi;\beta}^{x,\infty}}&\lesssim\int_0^t \frac{1}{(1+t-\tau)^{\frac32+\frac12}}\left\|\Gamma^z(f,f)(\tau)\right\|_{X_{\xi;\beta-1}^{x,1}}+e^{-\frac{t-\tau}{C}} \left\| \Gamma^z(f,f)(\tau) \right\|_{X_{\xi;\beta-1}^{x,\infty}}d\tau\\
&\lesssim\int_0^t \frac{1}{(1+t-\tau)^{\frac32+\frac12}}\left\|f \right\|_{X_{\xi;\beta}^{x,2}}\left\|f \right\|_{X_{\xi;\beta}^{x,2}}+e^{-\frac{t-\tau}{C}}\left\|f \right\|_{X_{\xi;\beta}^{x,\infty}}\left\|f \right\|_{X_{\xi;\beta}^{x,\infty}}d\tau\\
&\lesssim\int_0^t \frac{1}{(1+t-\tau)^{\frac32+\frac12}}\frac{1}{(1+\tau)^{\frac32}}(\left\|f_0\right\|_{X_{\xi;\beta}^{x,1}}+\left\|f_0\right\|_{X_{\xi;\beta}^{x,\infty}})^2d\tau\\
&\lesssim\frac{1}{(1+t)^{\frac32}}(\left\|f_0\right\|_{X_{\xi;\beta}^{x,1}}+\left\|f_0\right\|_{X_{\xi;\beta}^{x,\infty}})^2.
\end{aligned}
\end{equation}
If we assume that the initial value is sufficiently small, combining Lemma \ref{UCf}, \eqref{4.11} and \eqref{4.12}, we obtain the $L^{2}_x$ estimate and the $L^{\infty}_x$ estimate
for $f$, and these just satisfy the ansatz (\ref{ansatzINL}), which implies that the nonlinear map is an into.
On the other hand, we can prove the map is contractive by considering
\begin{equation*}
\int_0^t\mathbb{G}^{t-\tau}(z)\Gamma^z(g,g)(\tau)d\tau-\int_0^t\mathbb{G}^{t-\tau}(z)\Gamma^z(h,h)(\tau)d\tau=\int_0^t\mathbb{G}^{t-\tau}(z)(\Gamma^z(g,g-h)+\Gamma^z(g-h,h))d\tau.
\end{equation*}

By the fixed point theorem, we have the following lemma:
\begin{lemma}\label{INLf}
There exist positive constants $\delta$ and $C_0$ such that if
 $$
\left\|f_0\right\|_{X_{\xi;\beta}^{x,1}}+\left\|f_0\right\|_{X_{\xi;\beta}^{x,\infty}}<\delta,
$$
then there exists a unique global solution $f=f(x,t,\xi,z)$ for equation (\ref{UC}), satisfying
$$
\left\|f(t)\right\|_{X_{\xi;\beta}^{x,2}}\le C_0(1+t)^{-\frac34}(\left\|f_0\right\|_{X_{\xi;\beta}^{x,1}}+\left\|f_0\right\|_{X_{\xi;\beta}^{x,\infty}}),
$$
and
$$
\left\|f(t)\right\|_{X_{\xi;\beta}^{x,\infty}}\le C_0(1+t)^{-\frac32}(\left\|f_0\right\|_{X_{\xi;\beta}^{x,1}}+\left\|f_0\right\|_{X_{\xi;\beta}^{x,\infty}}).
$$

In this lemma, all constants are independent of $z$.
\end{lemma}

\subsection{The proof of the  main theorem}Next, we study the derivatives of the solution for equation (\ref{UC}) to complete the proof of the main theorem. We first consider the first-order derivative $\partial_zf$.
By differentiating the equation (\ref{UC}) with respect to $z$, using
$$
\partial_z(\mathsf{L}^zf)=(\partial_z\mathsf{L}^z)f+\mathsf{L}^z(\partial_zf), \partial_z(\Gamma^z(f,f))=(\partial_z\Gamma^z)(f,f)+2\Gamma^z(\partial_zf,f),
$$
 we have
\begin{equation}\label{NLFD}
 \begin{cases}
(\partial_zf)_t+\xi\cdot \nabla_x(\partial_zf)=(\partial_z\mathsf{L}^z)f+\mathsf{L}^z(\partial_zf)+(\partial_z\Gamma^z)(f,f)+2\Gamma^z(\partial_zf,f),\\
\partial_zf(x,0,\xi,z)=\partial_zf_0(x,\xi,z).
\end{cases}
\end{equation}

By Lemma \ref{TC1}, we make the following ansatz for the solution of (\ref{NLFD}):
\begin{equation}\label{ansatzINL2}
\begin{aligned}
\left\|\partial_zf(t)\right\|_{X_{\xi;\beta}^{x,2}}\le\frac{10C_{G,1}}{(1+t)^{\frac34}}(\left\|\partial_zf_0\right\|_{X_{\xi;\beta}^{x,1}}+\left\|\partial_zf_0\right\|_{X_{\xi;\beta}^{x,\infty}}+\left\|f_0\right\|_{X_{\xi;\beta}^{x,1}}+\left\|f_0\right\|_{X_{\xi;\beta}^{x,\infty}}),\\
\left\|\partial_zf(t)\right\|_{X_{\xi;\beta}^{x,\infty}}\le\frac{10C_{G,1}}{(1+t)^{\frac32}}(\left\|\partial_zf_0\right\|_{X_{\xi;\beta}^{x,1}}+\left\|\partial_zf_0\right\|_{X_{\xi;\beta}^{x,\infty}}+\left\|f_0\right\|_{X_{\xi;\beta}^{x,1}}+\left\|f_0\right\|_{X_{\xi;\beta}^{x,\infty}}).
\end{aligned}
\end{equation}
By Duhamel's principle, $\partial_zf$ satisfies the integal equation
\begin{equation*}
\begin{aligned}
\partial_zf(x,t,\xi,z)&=\mathbb{G}^t(z)\partial_zf_0+\int_0^t\mathbb{G}^{t-\tau}(z)\big[(\partial_z\mathsf{L}^z)f+(\partial_z\Gamma^z)(f,f)+2\Gamma^z(\partial_zf,f)\big](\tau)d\tau.
\end{aligned}
\end{equation*}
We estimate the integral part term by term.
Firstly, note that
\begin{equation}\label{xianxing}
\int_0^t\mathbb{G}^{t-\tau}(z)(\partial_z\mathsf{L}^z)f(\tau)d\tau=\int_0^t\mathbb{G}^{t-\tau}(\partial_z\mathsf{L}^z)\mathbb{G}^{\tau}f_0d\tau+\int_0^t\int_0^{\tau_1}\mathbb{G}^{t-\tau_1}(\partial_z\mathsf{L}^z)\mathbb{G}^{\tau_1-\tau}\Gamma^z(f,f)(\tau)d\tau d\tau_1.
\end{equation}
We observe that
 $
\mathbb{G}^t(z)\partial_zf_0+\int_0^t\mathbb{G}^{t-\tau}(\partial_z\mathsf{L}^z)\mathbb{G}^{\tau}f_0d\tau
$
is the solution of equation (\ref{UCD}) with initial data $f_0$, and using the results of Lemma \ref{TC1}, we obtain the following estimate:
\begin{equation*}
\begin{aligned}
\bigg\|\mathbb{G}^t(z)\partial_zf_0+&\int_0^t\mathbb{G}^{t-\tau}(\partial_z\mathsf{L}^z)\mathbb{G}^{\tau}f_0d\tau\bigg\|_{X_{\xi;\beta}^{x,2}}\\
&\le \frac{C_{G,1}}{(1+t)^{\frac34}}(\left\|f_0\right\|_{X_{\xi;\beta}^{x,1}}+\left\|f_0\right\|_{X_{\xi;\beta}^{x,\infty}}+\left\|\partial_zf_0\right\|_{X_{\xi;\beta}^{x,1}}+\left\|\partial_zf_0\right\|_{X_{\xi;\beta}^{x,\infty}}),\\
\end{aligned}
\end{equation*}
and
\begin{equation*}
\begin{aligned}
\bigg\|\mathbb{G}^t(z)\partial_zf_0+&\int_0^t\mathbb{G}^{t-\tau}(\partial_z\mathsf{L}^z)\mathbb{G}^{\tau}f_0d\tau\bigg\|_{X_{\xi;\beta}^{x,\infty}}\\
&\le \frac{C_{G,1}}{(1+t)^{\frac32}}(\left\|f_0\right\|_{X_{\xi;\beta}^{x,1}}+\left\|f_0\right\|_{X_{\xi;\beta}^{x,\infty}}+\left\|\partial_zf_0\right\|_{X_{\xi;\beta}^{x,1}}+\left\|\partial_zf_0\right\|_{X_{\xi;\beta}^{x,\infty}}).
\end{aligned}
\end{equation*}

Next, we estimate the remaining part of (\ref{xianxing}), by Lemma \ref{GTS}, Lemma \ref{UCf}, Lemma \ref{INLf} and (\ref{gamma}), we have
\begin{equation*}
\begin{aligned}
&\left\| \int_0^t\int_0^{\tau_1}\mathbb{G}^{t-\tau_1}(\partial_z\mathsf{L}^z)\mathbb{G}^{\tau_1-\tau}\Gamma^z(f,f)(\tau)d\tau d\tau_1\right\|_{X^{x,2}_{\xi;\beta}}\\
&\lesssim\int_0^t\int_0^{\tau_1}(1+t-\tau_1)^{-\frac32(\frac{2}{3}-\frac{1}{2})-\frac12}\left\|(\partial_z\mathsf{L}^z)\mathsf{P}_1\mathbb{G}^{\tau_1-\tau}\Gamma^z(f,f)(\tau)\right\|_{X^{x,\frac32}_{\xi;\beta-1}}d\tau d\tau_1\\
&\lesssim\int_0^t\int_0^{\tau_1}(1+t-\tau_1)^{-\frac34}(1+\tau_1-\tau)^{-\frac32(1-\frac{2}{3})-1}\left\| \mathsf{P}_1\Gamma^z(f,f)(\tau)\right\|_{X^{x,1}_{\xi;\beta-1}}d\tau d\tau_1\\
&\lesssim\int_0^t\int_0^{\tau_1}(1+t-\tau_1)^{-\frac34}(1+\tau_1-\tau)^{-\frac32}\left\|f(\tau)\right\|^2_{X^{x,2}_{\xi;\beta}}d\tau d\tau_1\\
&\lesssim\int_0^t\int_0^{\tau_1}(1+t-\tau_1)^{-\frac34}(1+\tau_1-\tau)^{-\frac32}(1+\tau)^{-\frac32} (\left\|f_0\right\|_{X_{\xi;\beta}^{x,1}}+\left\|f_0\right\|_{X_{\xi;\beta}^{x,\infty}})^2d\tau d\tau_1\\
&\lesssim\frac{1}{(1+t)^{\frac34}}(\left\|f_0\right\|_{X_{\xi;\beta}^{x,1}}+\left\|f_0\right\|_{X_{\xi;\beta}^{x,\infty}})^2,
\end{aligned}
\end{equation*}
and
\begin{equation*}
\begin{aligned}
&\left\| \int_0^t\int_0^{\tau_1}\mathbb{G}^{t-\tau_1}(\partial_z\mathsf{L}^z)\mathbb{G}^{\tau_1-\tau}\Gamma^z(f,f)(\tau)d\tau d\tau_1\right\|_{X^{x,\infty}_{\xi;\beta}}\\
&\lesssim\int_0^t\int_0^{\tau_1}(1+t-\tau_1)^{-\frac32(\frac{2}{3})-\frac12}\left\|(\partial_z\mathsf{L}^z)\mathsf{P}_1\mathbb{G}^{\tau_1-\tau}\Gamma^z(f,f)(\tau)\right\|_{X^{x,\frac32}_{\xi;\beta-1}}d\tau d\tau_1\\
&\lesssim\int_0^t\int_0^{\tau_1}(1+t-\tau_1)^{-\frac32}(1+\tau_1-\tau)^{-\frac32(1-\frac{2}{3})-1}\left\| \mathsf{P}_1\Gamma^z(f,f)(\tau)\right\|_{X^{x,1}_{\xi;\beta-1}}d\tau d\tau_1\\
&\lesssim\int_0^t\int_0^{\tau_1}(1+t-\tau_1)^{-\frac32}(1+\tau_1-\tau)^{-\frac32}\left\|f(\tau)\right\|^2_{X^{x,2}_{\xi;\beta}}d\tau d\tau_1\\
&\lesssim\int_0^t\int_0^{\tau_1}(1+t-\tau_1)^{-\frac32}(1+\tau_1-\tau)^{-\frac32}(1+\tau)^{-\frac32} (\left\|f_0\right\|_{X_{\xi;\beta}^{x,1}}+\left\|f_0\right\|_{X_{\xi;\beta}^{x,\infty}})^2d\tau d\tau_1\\
&\lesssim\frac{1}{(1+t)^{\frac32}}(\left\|f_0\right\|_{X_{\xi;\beta}^{x,1}}+\left\|f_0\right\|_{X_{\xi;\beta}^{x,\infty}})^2.
\end{aligned}
\end{equation*}

 The estimate satisfied by $\partial_z\Gamma^z$ is similar to that satisfied by $\Gamma^z$, as given in \eqref{gamma}, so
the estimate procedure for $\int_0^t\mathbb{G}^{t-\tau}(z)(\partial_z\Gamma^z)(f,f)(\tau)d\tau$ is exactly the same as \eqref{4.11} and \eqref{4.12} in Section \ref{nonlinear problems}, and it follows
\begin{equation*}
\begin{aligned}
\left\|\int_0^t\mathbb{G}^{t-\tau}(z)(\partial_z\Gamma^z)(f,f)(\tau)d\tau\right\|_{X_{\xi;\beta}^{x,2}}\lesssim\frac{1}{(1+t)^{\frac34}}(\left\|f_0\right\|_{X_{\xi;\beta}^{x,1}}+\left\|f_0\right\|_{X_{\xi;\beta}^{x,\infty}})^2,\\
\left\|\int_0^t\mathbb{G}^{t-\tau}(z)(\partial_z\Gamma^z)(f,f)(\tau)d\tau\right\|_{X_{\xi;\beta}^{x,\infty}}\lesssim\frac{1}{(1+t)^{\frac32}}(\left\|f_0\right\|_{X_{\xi;\beta}^{x,1}}+\left\|f_0\right\|_{X_{\xi;\beta}^{x,\infty}})^2.
\end{aligned}
\end{equation*}

For the last part, by Lemma \ref{INLf} and the ansatz (\ref{ansatzINL}), we have
\begin{equation*}
\begin{aligned}
&\left\|\int_0^t\mathbb{G}^{t-\tau}(z)\Gamma^z(\partial_zf,f)(\tau)d\tau\right\|_{X^{x,2}_{\xi;\beta}}\\
&\lesssim\int_0^t\frac{1}{(1+t-\tau)^{\frac34}}\frac{1}{(1+\tau)^{\frac34}}\frac{1}{(1+\tau)^{\frac34}}\bigg[\sum_{s=0}^{1}(\left\|\partial^s_zf_0\right\|_{X_{\xi;\beta}^{x,1}}+\left\|\partial^s_zf_0\right\|_{X_{\xi;\beta}^{x,\infty}})\bigg]^2d\tau\\
&\lesssim\frac{1}{(1+t)^{\frac34}}\bigg[\sum_{s=0}^{1}(\left\|\partial^s_zf_0\right\|_{X_{\xi;\beta}^{x,1}}+\left\|\partial^s_zf_0\right\|_{X_{\xi;\beta}^{x,\infty}})\bigg]^2,
\end{aligned}
\end{equation*}
and
\begin{equation*}
\begin{aligned}
&\left\|\int_0^t\mathbb{G}^{t-\tau}(z)\Gamma^z(\partial_zf,f)(\tau)d\tau\right\|_{X^{x,\infty}_{\xi;\beta}}\\
&\lesssim\int_0^t\frac{1}{(1+t-\tau)^{\frac32}}\frac{1}{(1+\tau)^{\frac34}}\frac{1}{(1+\tau)^{\frac34}}\bigg[\sum_{s=0}^{1}(\left\|\partial^s_zf_0\right\|_{X_{\xi;\beta}^{x,1}}+\left\|\partial^s_zf_0\right\|_{X_{\xi;\beta}^{x,\infty}})\bigg]^2d\tau\\
&\lesssim\frac{1}{(1+t)^{\frac32}}\bigg[\sum_{s=0}^{1}(\left\|\partial^s_zf_0\right\|_{X_{\xi;\beta}^{x,1}}+\left\|\partial^s_zf_0\right\|_{X_{\xi;\beta}^{x,\infty}})\bigg]^2.
\end{aligned}
\end{equation*}

 Based on the estimates for the three parts, we close the ansatz (\ref{ansatzINL2}). Thus, we have obtained the large-time behavior of the derivatives of the solution to the nonlinear problem.

Next, we study the higher-order derivatives.
We take the second-order derivative $\partial^2_zf$ as an example, and the estimates for other higher-order derivatives are similar. By differentiating equation (\ref{UC}) twice, and using
$$
\partial^2_z(\mathsf{L}^zf)=2(\partial_z\mathsf{L}^z)(\partial_zf)+(\partial^2_z\mathsf{L}^z)f+\mathsf{L}^z(\partial^2_zf),
$$
$$
\partial^2_z(\Gamma^z(f,f))=(\partial^2_z\Gamma^z)(f,f)+2\Gamma^z(\partial^2_zf,f)+2\Gamma^z(\partial_zf,\partial_zf)+4(\partial_z\Gamma^z)(\partial_zf,f),
$$
we have 
\begin{equation*}\label{NLSD}
 \begin{cases}
(\partial^2_zf)_t+\xi\cdot \nabla_x(\partial^2_zf)&=2(\partial_z\mathsf{L}^z)(\partial_zf)+(\partial^2_z\mathsf{L}^z)f+\mathsf{L}^z(\partial^2_zf)\\
& \ \  +(\partial^2_z\Gamma^z)(f,f)+2\Gamma^z(\partial^2_zf,f)+2\Gamma^z(\partial_zf,\partial_zf)+4(\partial_z\Gamma^z)(\partial_zf,f)\\
\partial^2_zf(x,0,\xi,z)&=\partial^2_zf_0(x,\xi,z).\\
\end{cases}
\end{equation*}

By Lemma \ref{TC2}, we make the following ansatz for $\partial^2_zf$:
\begin{equation}\label{ansatzINL3}
\begin{aligned}
\left\|\partial^2_zf(t)\right\|_{X_{\xi;\beta}^{x,2}}&\le 10C_{G,2}\bigg[\frac{1}{(1+t)^{\frac34}}(\left\|\partial^2_zf_0\right\|_{X_{\xi;\beta}^{x,1}}+\left\|\partial^2_zf_0\right\|_{X_{\xi;\beta}^{x,\infty}})\\
&\ \ +\sum_{s=0}^{1}\frac{(\ln(1+t))^{1-s}}{(1+t)^{\frac34}}(\left\|\partial^s_zf_0\right\|_{X_{\xi;\beta}^{x,1}}+\left\|\partial^s_zf_0\right\|_{X_{\xi;\beta}^{x,\infty}})\bigg],\\
\left\|\partial^2_zf(t)\right\|_{X_{\xi;\beta}^{x,\infty}}&\le 10C_{G,2}\bigg[\frac{1}{(1+t)^{\frac32}}(\left\|\partial^2_zf_0\right\|_{X_{\xi;\beta}^{x,1}}+\left\|\partial^2_zf_0\right\|_{X_{\xi;\beta}^{x,\infty}})\\
&\ \ +\sum_{s=0}^{1}\frac{(\ln(1+t))^{1-s}}{(1+t)^{\frac32}}(\left\|\partial^s_zf_0\right\|_{X_{\xi;\beta}^{x,1}}+\left\|\partial^s_zf_0\right\|_{X_{\xi;\beta}^{x,\infty}})\bigg].\\
\end{aligned}
\end{equation}

By Duhamel's principle, $\partial^2_zf$ satisfies the integal equation
\begin{equation*}
\begin{aligned}
\partial^2_zf(x,t,\xi,z)&=\mathbb{G}^t(z)\partial^2_zf_0+\int_0^t\mathbb{G}^{t-\tau}(z)\big[2(\partial_z\mathsf{L}^z)(\partial_zf)+(\partial^2_z\mathsf{L}^z)f\\
&\ \ +\partial^2_z\Gamma^z(f,f)+2\Gamma^z(\partial^2_zf,f)+2\Gamma^z(\partial_zf,\partial_zf)+4\partial_z\Gamma^z(\partial_zf,f)\big](\tau)d\tau.
\end{aligned}
\end{equation*}
Clearly, for the part
$$\mathbb{G}^t(z)\partial^2_zf_0+\int_0^t\mathbb{G}^{t-\tau}(z)\big[2(\partial_z\mathsf{L}^z)(\partial_zf)+(\partial^2_z\mathsf{L}^z)f\big](\tau)d\tau,$$
 the proof procedure is the same as the Lemma \ref{TC2} and the case of the first-order derivative, so we omit the details of calculation and state the results as follows:
\begin{equation}\label{bianhao1}
\begin{aligned}
\bigg\|\mathbb{G}^t(z)\partial^2_zf_0&+\int_0^t\mathbb{G}^{t-\tau}(z)\big[2(\partial_z\mathsf{L}^z)(\partial_zf)+(\partial^2_z\mathsf{L}^z)f\big](\tau)d\tau\bigg\|_{X_{\xi;\beta}^{x,2}}\\
&\lesssim \frac{1}{(1+t)^{\frac34}}\sum_{s=0}^{2}(\left\|\partial^s_zf_0\right\|_{X_{\xi;\beta}^{x,1}}+\left\|\partial^s_zf_0\right\|_{X_{\xi;\beta}^{x,\infty}}),\\
\bigg\|\mathbb{G}^t(z)\partial^2_zf_0&+\int_0^t\mathbb{G}^{t-\tau}(z)\big[2(\partial_z\mathsf{L}^z)(\partial_zf)+(\partial^2_z\mathsf{L}^z)f\big](\tau)d\tau\bigg\|_{X_{\xi;\beta}^{x,\infty}}\\
&\lesssim \frac{1}{(1+t)^{\frac32}}\sum_{s=0}^{2}(\left\|\partial^s_zf_0\right\|_{X_{\xi;\beta}^{x,1}}+\left\|\partial^s_zf_0\right\|_{X_{\xi;\beta}^{x,\infty}}).
\end{aligned}
\end{equation}
As for 
$$\int_0^t\mathbb{G}^{t-\tau}(z)\big[\partial^2_z\Gamma^z(f,f)+2\Gamma^z(\partial^2_zf,f)+2\Gamma^z(\partial_zf,\partial_zf)+4\partial_z\Gamma^z(\partial_zf,f)\big](\tau)d\tau,$$
considering its $L_x^2$ estimates, from \eqref{gamma} and ansatz \eqref{ansatzINL3}, it follows that 
\begin{equation}\label{bianhao2}
\begin{aligned}
&\left\|\int_0^t\mathbb{G}^{t-\tau}(z)\big[\partial^2_z\Gamma^z(f,f)+2\Gamma^z(\partial^2_zf,f)+2\Gamma^z(\partial_zf,\partial_zf)+4\partial_z\Gamma^z(\partial_zf,f)\big](\tau)d\tau\right\|_{X^{x,2}_{\xi;\beta}}\\
&\lesssim\int_0^t\frac{1}{(1+t-\tau)^{\frac34}}\bigg(\frac{1}{(1+\tau)^{\frac34}}\frac{1}{(1+\tau)^{\frac34}}+ \frac{\ln(1+\tau)}{(1+\tau)^{\frac34}}\frac{1}{(1+\tau)^{\frac34}}+\frac{1}{(1+\tau)^{\frac34}}\frac{1}{(1+\tau)^{\frac34}}\\
&\ \ +\frac{1}{(1+\tau)^{\frac34}}\frac{1}{(1+\tau)^{\frac34}}\bigg)\bigg[\sum_{s=0}^{2}(\left\|\partial^s_zf_0\right\|_{X_{\xi;\beta}^{x,1}}+\left\|\partial^s_zf_0\right\|_{X_{\xi;\beta}^{x,\infty}})\bigg]^2d\tau\\
&\lesssim\frac{\ln(1+t)}{(1+t)^{\frac34}}\bigg[\sum_{s=0}^{2}(\left\|\partial^s_zf_0\right\|_{X_{\xi;\beta}^{x,1}}+\left\|\partial^s_zf_0\right\|_{X_{\xi;\beta}^{x,\infty}})\bigg]^2.
\end{aligned}
\end{equation}
And for the $L_x^\infty$ estimates, it follows that 
\begin{equation}\label{bianhao3}
\begin{aligned}
&\left\|\int_0^t\mathbb{G}^{t-\tau}(z)\big[\partial^2_z\Gamma^z(f,f)+2\Gamma^z(\partial^2_zf,f)+2\Gamma^z(\partial_zf,\partial_zf)+4\partial_z\Gamma^z(\partial_zf,f)\big](\tau)d\tau\right\|_{X^{x,\infty}_{\xi;\beta}}\\
&\lesssim\int_0^t\frac{1}{(1+t-\tau)^{\frac32}}\bigg(\frac{1}{(1+\tau)^{\frac34}}\frac{1}{(1+\tau)^{\frac34}}+ \frac{\ln(1+\tau)}{(1+\tau)^{\frac34}}\frac{1}{(1+\tau)^{\frac34}}+\frac{1}{(1+\tau)^{\frac34}}\frac{1}{(1+\tau)^{\frac34}}\\
&\ \ +\frac{1}{(1+\tau)^{\frac34}}\frac{1}{(1+\tau)^{\frac34}}\bigg)\bigg[\sum_{s=0}^{2}(\left\|\partial^s_zf_0\right\|_{X_{\xi;\beta}^{x,1}}+\left\|\partial^s_zf_0\right\|_{X_{\xi;\beta}^{x,\infty}})\bigg]^2d\tau\\
&\lesssim\frac{\ln(1+t)}{(1+t)^{\frac32}}\bigg[\sum_{s=0}^{2}(\left\|\partial^s_zf_0\right\|_{X_{\xi;\beta}^{x,1}}+\left\|\partial^s_zf_0\right\|_{X_{\xi;\beta}^{x,\infty}})\bigg]^2.
\end{aligned}
\end{equation}
Combining \eqref{bianhao1}, \eqref{bianhao2} and \eqref{bianhao3}, we have closed the ansatz \eqref{ansatzINL2}. From Lemma \ref{TC2} and the above proof, for given $\alpha \in \mathbb{N}$, the constants $C_\alpha$ in Theorem \ref{T1} are determinded by $C_{G,k}$, $C^*_{b}$ and $\alpha$. Combining \eqref{ansatzINL}, \eqref{ansatzINL2}, \eqref{ansatzINL3} and the estimates for other higher-order derivatives, we have completed the proof of Theorem \ref{T1}.

%\section{Spectral accuracy of the gPC-SG method}\label{gPC-SG method}
\section{Spectral accuracy of the gPC-SG method}\label{gPC-SG method}
In this section, we will introduce the gPC-SG system and provide estimates for its solution, as well as the gPC error.

\subsection{The gPC based stochastic Galerkin method}\label{Galerkin method}
In this subsection, we review the gPC-SG method for solving kinetic equations with uncertainties. Take the
Boltzmann equation as an example (see \cite{g3} for more details). Let $F$ be the solution of equation (\ref{UB}) and $f$ be the solution of equation (\ref{UC}). One seeks a solution in the following form:
\begin{equation*}
\begin{aligned}
F(x,t,\xi,z)\approx \sum_{k=1}^KF_{k}(x,t,\xi)\psi_{k}(z)=:F^K(x,t,\xi,z),
\end{aligned}
\end{equation*}
\begin{equation*}\label{SG1}
\begin{aligned}
f(x,t,\xi,z)\approx \sum_{k=1}^Kf_{k}(x,t,\xi)\psi_{k}(z)=:f^K(x,t,\xi,z).
\end{aligned}
\end{equation*}
By considering the perturbation framework, we write
\begin{equation*}
F_{k}=\mathsf{M}+\sqrt{\mathsf{M}}f_{k}.
\end{equation*}
Then, by the  standard Galerkin projection $P_K$, one obtains the gPC-SG system for $f_{k}$:
\begin{equation}\label{PCSG}
 \begin{cases}
\partial_tf_{k}+\xi\cdot \nabla_xf_{k}=\mathsf{L}_{k}(f^K)+\Gamma_{k}(f^K,f^K),\ x\in \mathbb{R}^3,\\
f_{k}(x,0,\xi)=f_{k}^0(x,\xi), \\
\end{cases}
\end{equation}
for each $1\le k\le K$, with the initial data given by 
\begin{equation*}
f^0_{k}(x,\xi):=\int_{I_z}f_0(x,\xi,z)\psi_{k}(z)\pi(z)dz.
\end{equation*} 
Here, $P_K$, $\pi(z)$ and $\psi_{k}(z), 1\le k\le K$, have already been introduced in Section \ref{Main result}.
The collision parts are given by 
\begin{equation}\label{a1}
\begin{aligned}
\mathsf{L}_{k}(f^K)=-\nu_{k}(f^K)-\mathsf{K}^1_{k}(f^K)+\mathsf{K}^2_{k}(f^K),\ \nu_{k}(f^K)=\sum_{i=1}^K\nu_{ki}f_{i},
\end{aligned}
\end{equation}
\begin{equation}\label{a2}
\begin{aligned}
\mathsf{K}^1_{k}(f^K)=\sqrt{\mathsf{M}(\xi)}\sum_{i=1}^K\int_{\mathbb{R}^3\times S^2,(\xi-\xi_*)\cdot \Omega \ge0}S_{ki}|\xi-\xi_*|f_{i}(\xi_*)\sqrt{\mathsf{M}(\xi_*)}d\xi_*d\Omega,
\end{aligned}
\end{equation}
\begin{equation}\label{a3}
\begin{aligned}
\mathsf{K}^2_{k}(f^K)=\sum_{i=1}^K\int_{\mathbb{R}^3\times S^2,(\xi-\xi_*)\cdot \Omega \ge0}S_{ki}|\xi-\xi_*|(f_{i}(\xi')\sqrt{\mathsf{M}(\xi'_*)}+f_{i}(\xi'_*)\sqrt{\mathsf{M}(\xi')})\sqrt{\mathsf{M}(\xi_*)}(\xi_*)d\xi_*d\Omega,
\end{aligned}
\end{equation}
and
\begin{equation}\label{a4}
\begin{aligned}
\Gamma_{k}(f^K,f^K)=&\frac12\sum_{i,j=1}^K\int_{\mathbb{R}^3\times S^2,(\xi-\xi_*)\cdot \Omega \ge0}S^{\prime}_{kij}|\xi-\xi_*|\sqrt{\mathsf{M}(\xi_*)}\\
&\times (f_{i}(\xi')f_{j}(\xi'_*)-f_{i}(\xi)f_{j}(\xi_*))d\xi_*d\Omega,\\
\end{aligned}
\end{equation}
with
\begin{equation}\label{a5}
\begin{aligned}
S_{ki}:=\int_{I_z}b(\cos{\theta},z)\psi_{k}(z)\psi_{i}(z)\pi(z)dz, \ \ \nu_{ki}:=\int_{\mathbb{R}^3\times S^2,(\xi-\xi_*)\cdot \Omega \ge0}S_{ki}|\xi-\xi_*|\mathsf{M}(\xi_*)d\xi_*d\Omega,
\end{aligned}
\end{equation}
and
\begin{equation}\label{a6}
\begin{aligned}
S^{\prime}_{kij}:=\int_{I_z}b(\cos{\theta},z)\psi_{k}(z)\psi_{i}(z)\psi_{j}(z)\pi(z)dz.
\end{aligned}
\end{equation}

\subsection{The linear system}
Recall that the uncertain variable $z$ is one-dimensional, and $I_z$ has finite support $|z| \le C_z$.
 For the convenience of explanation, we assume the collision kernel is linear in $z$, with the form of
\begin{equation}\label{b3}
b(\cos{\theta},z)=b_0\cos{\theta}+zb_1\cos{\theta},\ |b_1|\le\gamma|b_0|,\ 0<\gamma<\frac{1}{2^m+1},
\end{equation}
here, $m$ is introduced in definition \eqref{weight}.

First of all, we consider coupled linear systems for $g_k,1\le k\le K$:
\begin{equation}\label{PCSGL}
 \begin{cases}
\partial_tg_{k}+\xi\cdot \nabla_xg_{k}=\mathsf{L}_{k}(g^K),\\
g_{k}(x,0,\xi)=g_{k}^0(x,\xi), \\
\end{cases}
\end{equation} 
here, 
$$
g^K=\sum_{k=1}^Kg_{k}(x,t,\xi)\psi_{k}(z).
$$
Through $g^K$, define a vector-valued function $\vec{g}_K(x,t,\xi)$ by
$$
\vec{g}_K=(g_1,\cdots,g_k,\cdots,g_K)^T.
$$
Rewrite the system of equation (\ref{PCSGL}) into a form satisfied by $\vec{g}_K(x,t,\xi)$, and we have
\begin{equation}\label{PCSGL1}
 \begin{cases}
\partial_t\vec{g}_{K}+\xi\cdot \nabla_x\vec{g}_{K}=\mathbf{L}(\vec{g}_{K}),\\
\vec{g}_K(x,0,\xi)=\vec{g}_{\mathrm{in},K}(x,\xi):=(g_{0}^0(x,\xi),\cdots,g_{k}^0(x,\xi),\cdots,g_{K}^0(x,\xi))^T.\\
\end{cases}
\end{equation}
Here, $\mathbf{L}=\{ \mathsf{L}_{ij}\}_{K\times K}$ is a matrix of operators, where 
\begin{equation*}
\begin{aligned}
\mathsf{L}_{ij}g_j= \int_{\mathbb{R}^3\times S^2,(\xi-\xi_*)\cdot \Omega \ge0}&S_{ij}|\xi-\xi_*|\bigg[g_{j}(\xi')\sqrt{\mathsf{M}(\xi'_*)}+g_{j}(\xi'_*)\sqrt{\mathsf{M}(\xi')}\\
&-\sqrt{\mathsf{M}(\xi_*)}g_j(\xi)-\sqrt{\mathsf{M}(\xi)}g_j(\xi_*)\bigg]\sqrt{\mathsf{M}(\xi_*)}d\xi_*d\Omega.
\end{aligned}
\end{equation*}
%and
%\begin{equation*}
%\begin{aligned}
%\mathsf{L}_{ji}g_j= \int_{\mathbb{R}^3\times S^2,(\xi-\xi_*)\cdot \Omega \ge0}&S_{ji}|\xi-\xi_*|\bigg[g_{j}(\xi')\sqrt{\mathsf{M}(\xi'_*)}+g_{j}(\xi'_*)\sqrt{\mathsf{M}(\xi')}\\
%&-\sqrt{\mathsf{M}(\xi_*)}g_j(\xi)-\sqrt{\mathsf{M}(\xi)}g_j(\xi_*)\bigg]\sqrt{\mathsf{M}(\xi_*)}d\xi_*d\Omega.
%\end{aligned}
%\end{equation*}
Note that from \eqref{a5}, $S_{ki}$ satisfies that
\begin{equation*}
\begin{aligned}
S_{ki}=&\int_{I_z}b(\cos{\theta},z)\psi_{k}(z)\psi_{i}(z)\pi(z)dz\\
=&\int_{I_z}(b_0\cos{\theta}+zb_1\cos{\theta})\psi_{k}(z)\psi_{i}(z)\pi(z)dz\\
=&b_0\cos{\theta}\int_{I_z}\psi_{k}(z)\psi_{i}(z)\pi(z)dz+b_1\cos{\theta}\int_{I_z}z\psi_{k}(z)\psi_{i}(z)\pi(z)dz.
\end{aligned}
\end{equation*}
By the properties of orthogonal polynomials, we know that $S_{ki}\ne 0$ only when $|k-i|\le 1$. Clearly, it follows that
\begin{equation*}
\begin{aligned}
&S_{kk}=b_0\cos{\theta},1\le k\le K,\\
&S_{k,k+1}= b_1\cos{\theta}, 1\le k\le K-1, \\
&S_{k,k-1}= b_1\cos{\theta}, 2\le k\le K.
\end{aligned}
\end{equation*} 
Therefore, $\mathbf{L}$ is a symmetric tridiagonal matrix and diagonally dominant.

In fact, compared to $\mathsf{L}$, $\mathsf{L}_{ij}$ satisfies
$$
\mathsf{L}_{ij}=b_{ij}\mathsf{L},%\ \text{for}\ \text{some}, b_ij
$$
here, $b_{kk}=b_0,1\le k\le K$ and $|b_{ij}|\le b_1, i\neq j$. 
 Then define coefficient matrix $\mathbf{B}$ by $\mathbf{B}=\{ b_{ij}\}_{K\times K}$, and without loss of generality, let $b_{kk}=b_0=1,1\le k\le K$. Through the above discussion, the coefficient matrix $\mathbf{B}$ is a symmetric positive definite tridiagonal matrix and 
$$
\mathbf{L}=\mathbf{B}\mathsf{L}.
$$

Next, the semigroup method is employed to solve the system (\ref{PCSGL1}).
The study of linearized Boltzmann collision operator for multiple species has been done in \cite{EJin,daus2016hypocoercivity}. They found that the essential spectrum of the linearized multi-species collision operator is very similar to the mono-species operator. And compared to the their systems, our systems are simpler. 

We will solve  system (\ref{PCSGL1}) by the semigroup method. We first provide the proof of explicit spectral-gap estimate for linear collision operator $\mathbf{L}$. The following lemma shows that there is a spectrum gap between the zero spectrum and the other spectrum for $\mathbf{L}$.
\begin{lemma}[Explicit spectral-gap estimate]\label{ESGE}
Let the collision parts satisfy assumptions (\ref{a1})-(\ref{a6}) and assume the collision kernel is linear in z. Then there exists a constant $C>0$ independent of $K$ such that
$$
(\mathbf{L}\vec{f}_K,\vec{f}_K)_\xi\le-C(\mathsf{P}_1\vec{f}_K,\mathsf{P}_1\vec{f}_K)_\xi,
$$
here, $\mathsf{P}_1\vec{f}_K=(\mathsf{P}_1f_1,\cdots,\mathsf{P}_1f_K)^T$ and $(\vec{h}_K,\vec{f}_K)_\xi=\sum_{k=1}^K(h_k,f_k)_\xi$.
\end{lemma}
\begin{proof}
Decompose $\mathbf{L}=\mathbf{L}^m+\mathbf{L}^b$, where $\mathbf{L}^m=b_0\mathbf{I}$.
%(\mathsf{L}^m_1,...,\mathsf{L}^m_{k},...,\mathsf{L}^m_K)$, $\mathbf{L}^b=(\mathsf{L}^b_1,...,\mathsf{L}^b_{k},...,\mathsf{L}^b_K)$, and 
%\begin{equation*}
%\mathsf{L}^m_{k}=\mathsf{L}_{kk}, \quad \mathsf{L}^b_{k}=\sum_{j\ne k}\mathsf{L}_{kj}.
%\end{equation*}
By (\ref{SGE}) and $b_0=1$, we have 
\begin{equation*}
(f_k,\mathsf{L}_{kk}f_k)_\xi\le-\nu_1(\mathsf{P}_1f_k,\mathsf{P}_1f_k)_\xi,\quad  \text{for}\quad 1\le k \le K,
\end{equation*}
and sum them over $k=1,...,K$, so it follows
\begin{equation}\label{GPCM}
(\mathbf{L}^m\vec{f}_K,\vec{f}_K)_\xi\le-\nu_1(\mathsf{P}_1\vec{f}_K,\mathsf{P}_1\vec{f}_K)_\xi.
\end{equation}
Note that
\begin{equation*}
\begin{aligned}
(\mathsf{L}_{ij}f_i,f_j)_\xi=-\frac{1}{4} \int_{\mathbb{R}^6\times S^2,(\xi-\xi_*)\cdot \Omega \ge0}&\mathsf{M}\mathsf{M}_*(-\frac{f_i}{\sqrt{\mathsf{M}}}-\frac{f_{i,*}}{\sqrt{\mathsf{M}_*}}+\frac{f_i'}{\sqrt{\mathsf{M}'}}+\frac{f_{i,*}^{\prime}}{\sqrt{\mathsf{M}^{\prime}_*}})\\
&\times(-\frac{f_j}{\sqrt{\mathsf{M}}}-\frac{f_{j,*}}{\sqrt{\mathsf{M}_*}}+\frac{f_j'}{\sqrt{\mathsf{M}'}}+\frac{f_{j,*}^{\prime}}{\sqrt{\mathsf{M}^{\prime}_*}}) S_{ij}|\xi-\xi_*|d\xi_*d\Omega d\xi,\\
\end{aligned}
\end{equation*}
with the abbreviations $\phi_*=\phi(\xi_*),\phi^\prime=\phi(\xi^\prime),\phi^{\prime}_*=\phi(\xi^\prime_*)$.
For $1\le k\le K-1$, it follows 
\begin{equation*}
\begin{aligned}
(\mathsf{L}_{kk}f_k,f_k)_\xi=-\frac{1}{4} \int_{\mathbb{R}^6\times S^2,(\xi-\xi_*)\cdot \Omega \ge0}&\mathsf{M}\mathsf{M}_*(-\frac{f_k}{\sqrt{\mathsf{M}}}-\frac{f_{k,*}}{\sqrt{\mathsf{M}_*}}+\frac{f_k'}{\sqrt{\mathsf{M}'}}+\frac{f_{k,*}^{\prime}}{\sqrt{\mathsf{M}^{\prime}_*}})\\
&\times(-\frac{f_k}{\sqrt{\mathsf{M}}}-\frac{f_{k,*}}{\sqrt{\mathsf{M}_*}}+\frac{f_k'}{\sqrt{\mathsf{M}'}}+\frac{f_{k,*}^{\prime}}{\sqrt{\mathsf{M}^{\prime}_*}}) b_0|\xi-\xi_*|d\xi_*d\Omega d\xi,\\
\end{aligned}
\end{equation*}
and 
\begin{equation*}
\begin{aligned}
|(\mathsf{L}_{k,k+1}f_k,f_{k+1})_\xi|\le\frac{1}{4} \bigg|\int_{\mathbb{R}^6\times S^2,(\xi-\xi_*)\cdot \Omega \ge0}&\mathsf{M}\mathsf{M}_*(-\frac{f_k}{\sqrt{\mathsf{M}}}-\frac{f_{k,*}}{\sqrt{\mathsf{M}_*}}+\frac{f_k'}{\sqrt{\mathsf{M}'}}+\frac{f_{k,*}^{\prime}}{\sqrt{\mathsf{M}^{\prime}_*}})\\
&\times(-\frac{f_{k+1}}{\sqrt{\mathsf{M}}}-\frac{f_{k+1,*}}{\sqrt{\mathsf{M}_*}}+\frac{f_{k+1}'}{\sqrt{\mathsf{M}'}}+\frac{f_{k+1,*}^{\prime}}{\sqrt{\mathsf{M}^{\prime}_*}}) b_1|\xi-\xi_*|d\xi_*d\Omega d\xi\bigg|.\\
\end{aligned}
\end{equation*}
By the arithmetic mean inequality and (\ref{b3}), it follows 
\begin{equation}\label{fangda}
\begin{aligned}
(\mathsf{L}_{kk}f_k,f_k)_\xi+(\mathsf{L}_{k+1,k+1}f_{k+1},f_{k+1})_\xi+\frac{1}{\gamma}|(\mathsf{L}_{k,k+1}f_k,f_{k+1})_\xi|+\frac{1}{\gamma}|(\mathsf{L}_{k+1,k}f_{k+1},f_{k})_\xi|\le0.
\end{aligned}
\end{equation}
%Similarly, for $2\le k\le K$, it follows
%\begin{equation}\label{fangda2}
%\begin{aligned}
%(\mathsf{L}_{kk}f_k,f_k)_\xi+(\mathsf{L}_{k-1,k-1}f_{k-1},f_{k-1})_\xi+\frac{1}{\alpha}|(\mathsf{L}_{k,k-1}f_k,f_{k-1})_\xi|+\frac{1}{\alpha}|(\mathsf{L}_{k-1,k}f_{k-1},f_{k})_\xi|\le0.
%\end{aligned}
%\end{equation}
Since $S_{ij}=0$ when $|i-j|>1$, by (\ref{GPCM}) and (\ref{fangda}), we have 
\begin{equation*}
\begin{aligned}
(\mathbf{L}\vec{f}_K,\vec{f}_K)_\xi&=(\mathbf{L}^m\vec{f}_K,\vec{f}_K)_\xi+(\mathbf{L}^b\vec{f}_K,\vec{f}_K)_\xi\\
&=(1-2\gamma)(\mathbf{L}^m\vec{f}_K,\vec{f}_K)_\xi+2\gamma(\mathbf{L}^m\vec{f}_K,\vec{f}_K)_\xi+(\mathbf{L}^b\vec{f}_K,\vec{f}_K)_\xi\\
&\le(1-2\gamma)(\mathbf{L}^m\vec{f}_K,\vec{f}_K)_\xi+\sum_{k=1}^{K-1}\gamma(\mathsf{L}_{kk}f_k,f_k)_\xi+\gamma(\mathsf{L}_{k+1,k+1}f_{k+1},f_{k+1})_\xi+2|(\mathsf{L}_{k,k+1}f_k,f_{k+1})_\xi|\\
&\le-(1-2\gamma)\nu_1(\mathsf{P}_1\vec{f}_K,\mathsf{P}_1\vec{f}_K)_\xi.\\
\end{aligned}
\end{equation*}

Let $C=(1-2\gamma)\nu_1$. Since $\gamma$ in (\ref{b3}) and $\nu_1$ in (\ref{SGE}) are independent of $K$, $C$ is also independent of $K$.
We have finished the proof.
\end{proof}
Next we study the semigroup $e^{\left( -i\xi\cdot \eta \mathbf{I}+\mathbf{L}\right) t}$. Let $(\sigma\left( \eta \right),\vec{e}\left( \eta \right))$ be the eigenvalue-eigenfunctions for $-i\xi\cdot \eta\mathbf{I} +\mathbf{L}$:
$$
(-i\xi\cdot \eta\mathbf{I} +\mathbf{L})\vec{e}(\eta)=\sigma (\eta )\vec{e}(\eta). 
$$
By Lemma \ref{ESGE}, the spectral gap for $\mathbf{L}$ still exists, so following the Lemma 5.4 in \cite{liu2011solving}, we can computate the eigenvalues and eigenfunctions of $-i\xi\cdot \eta\mathbf{I} +\mathbf{L}$. The following lemma concerns the spectrum information and the detail proof is oimtted.
\begin{lemma}\label{SPS}
Consider the spectrum $\mathrm{Spec}(\eta)$ of the operator $-i\xi\cdot \eta\mathbf{I} +\mathbf{L}$, $\eta \in \mathbb{R}^3$. The following statement $(\mathrm{I})-(\mathrm{II})$ are true.

$(\mathrm{I})$ For any $0<\delta\ll 1$, there corresponds $\tau=\tau(\delta)>0$ such that

\ \ \ $(\mathrm{i})$ For $|\eta|>\delta$,
$$
\mathrm{Spec}(\eta)\subset\{z\in\mathbb{C}:\mathrm{Re}(z)<-\tau\}.
$$

 \,  \ $(\mathrm{ii})$ For $|\eta|\le\delta$, the spectrum within the region $\{z\in\mathbb{C}:\mathrm{Re}(z)\ge-\tau\}$ consisting of exactly five eigenvalues $\sigma _{j}\left( \eta \right)\ (0\leq j\leq 4)$:
$$
\mathrm{Spec}(\eta)\cap\{z\in\mathbb{C}:\mathrm{Re}(z)\ge-\tau\}=\{\sigma_1(\eta),\sigma_2(\eta),\sigma_3(\eta),\sigma_4(\eta),\sigma_5(\eta)\}.
$$

$(\mathrm{II})$ For $\left \vert\eta \right \vert \ll 1$, the five
eigenvalues $\sigma _{j}\left( \eta \right) \ (0\leq j\leq 4)$ associated
with the corresponding eigenfunctions $\vec{e}_{j}\left( \eta \right)$ satisfy 
\begin{equation*}
\displaystyle \sigma _{j}(\eta )=-i\,a_{j}|\eta |-A_{j}|\eta |^{2}+O(|\eta
|^{3}), \  0\leq j\leq 4.
\end{equation*}%
Let $\mu_k>0(1\le k\le K)$ be the eigenvalues of the coefficient matrix $\mathbf{B}$ and $\vec{v}_k$ are the corresponding unit eigenvectors, then the eigenfunctions  $\vec{e}_{j}\left( \eta \right)$ satisfy that
$$
\vec{e}_{j}\left( \eta \right)=\alpha_1\vec{v}_1e^1_j+\cdots+\alpha_k\vec{v}_ke^k_j+\cdots+\alpha_K\vec{v}_Ke^K_j,\ 1\le k \le K.
$$
Here, for $1\le k \le K$, $\alpha_k\in \mathbb{R}$, $e^k_j(\eta)$ can be solved by 
$$
(-i\xi\cdot \eta +\mu_k\mathsf{L})e^k_j(\eta)=\sigma _{j}(\eta )e^k_j(\eta), \ (e^k_j(\eta),e^k_j(\eta))_\xi=1,
$$
and
\begin{equation*}
\displaystyle e^k_{j}(\eta )=E_{j}+O(|\eta |)\text{.}%
\end{equation*}
Here $A_{j}>0$, $0\leq j\leq 4$ and
\begin{equation*}
\left \{
\begin{array}{l}
a_{0}=\sqrt{\frac{5}{3}}\text{,}\quad a_{1}=-\sqrt{\frac{5}{3}}\text{,}\,
\quad a_{2}=a_{3}=a_{4}=0\text{,} \\[2mm]
E_{0}=\sqrt{\frac{3}{10}}\chi _{0}+\sqrt{\frac{1}{2}}\om \cdot \overline{%
\chi }+\sqrt{\frac{1}{5}}\chi _{4}\text{,} \\[2mm]
E_{1}=\sqrt{\frac{3}{10}}\chi _{0}-\sqrt{\frac{1}{2}}\om \cdot \overline{%
\chi }+\sqrt{\frac{1}{5}}\chi _{4}\text{,} \\[2mm]
E_{2}=-\sqrt{\frac{2}{5}}\chi _{0}+\sqrt{\frac{3}{5}}\chi _{4}\, \text{,} \\%
[2mm]
E_{3}=\om_{1}\cdot \overline{\chi }\text{,} \\[2mm]
E_{4}=\om_{2}\cdot \overline{\chi }\text{,}%
\end{array}%
\right.
\end{equation*}%
where $\overline{\chi }=(\chi _{1},\chi _{2},\chi _{3})$, $\eta =\left \vert
\eta \right \vert \om$ ($\om \in S^{2}$) and $\{ \om_{1},\om_{2},\om \}$ is
an orthonormal basis of ${\mathbb{R}}^{3}$.
\end{lemma}

Following the computations in Section 7 of \cite{liu2011solving}, by Lemma \ref{SPS} we can still provide pointwise estimates for the solution operator of system \eqref{PCSGL1}. Here, we omit the calculation procedure.
Therefore, we solve the system (\ref{PCSGL1}) of the linearized Boltzmann equation using the semigroup method and obtain the following lemma:
\begin{lemma}\label{PCSGTL} 
Let $\mathbf{G}^t$ be the solution operator of system \eqref{PCSGL1}. For given $b_0$ and $b_1$ in \eqref{b3}, there exists a positive constant $C_{G,S}$ independent of $K$, such that if $\beta> 3/2$, and $1\le p\le 2\le r \le \infty$, then $\mathbb{G}^t$ satisfies that for some $C>0$, 
\begin{equation*}
\begin{aligned}
\left\|\mathbf{G}^t\vec{g}_{\mathrm{in},K}\right\|_{L_x^r(L^{2}_{\xi})}&\le\frac{C_{G,S}}{(1+t)^{\frac32(\frac{1}{p}-\frac{1}{r})}}\left\|\vec{g}_{\mathrm{in},K}\right\|_{L_x^p(L^2_{\xi})}\\
&\ \ +C_{G,S}e^{-t/C}(\left\|\vec{g}_{\mathrm{in},K}\right\|_{L^{2}_{\xi}(L_x^2)}+\left\|\vec{g}_{\mathrm{in},K}\right\|_{L^{2}_{\xi}(L_x^2)}^{2/r}\left\|\vec{g}_{\mathrm{in},K}\right\|_{L^{2}_{\xi}(L_x^{\infty})}^{1-2/r}),
\end{aligned}
\end{equation*}
\begin{equation*}
\begin{aligned}
\left\|\mathbf{G}^t\vec{g}_{\mathrm{in}.K}\right\|_{L^{\infty}_{\xi,\beta}(L_x^{r})}&\le\frac{C_{G,S}}{(1+t)^{\frac32(\frac{1}{p}-\frac{1}{r})}}\left\|\vec{g}_{\mathrm{in},K}\right\|_{L_x^p(L^2_{\xi})}\\
&\ \ +C_{G,S}e^{-t/C}(\left\|\vec{g}_{\mathrm{in},K}\right\|_{L^{\infty}_{\xi,\beta}(L_x^{r})}+\left\|\vec{g}_{\mathrm{in},K}\right\|_{L^{2}_{\xi}(L_x^2)}+\left\|\vec{g}_{\mathrm{in},K}\right\|_{L^{2}_{\xi}(L_x^2)}^{2/r}\left\|\vec{g}_{\mathrm{in},K}\right\|_{L^{2}_{\xi}(L_x^{\infty})}^{1-2/r}).
\end{aligned}
\end{equation*}
 Moreover, if $\mathsf{P}_0\vec{g}_{\mathrm{in},K}=\vec{0}$, then we will get extra $(1+t)^{-1/2}$ decay rate for each estimate above.
\end{lemma} 

In order to close the the nonlinear problem, weighted estimates are needed. Recall that the weight matrix $\mathbf{W}$ is given in \eqref{weight} with 
$$
m>n+1.
$$ 
Next, we state that the solution operator $\mathbf{G}^t$ also satisfies the weighted estimates. Note that $\mathbf{W}\vec{g}_K$ is the solution of the following system:
\begin{equation}\label{PCSGL3}
 \begin{cases}
\partial_t\mathbf{W}\vec{g}_K+\xi\cdot \nabla_x\mathbf{W}\vec{g}_K=\mathbf{W}\mathbf{L}\mathbf{W}^{-1}(\mathbf{W}\vec{g}_K),\\
\mathbf{W}\vec{g}_K(x,0,\xi)=\mathbf{W}\vec{g}_{\mathrm{in},K}(x,\xi)=(g_{0}^0(x,\xi),\cdots,k^mg_{k}^0(x,\xi),\cdots,K^mg_{K}^0(x,\xi))^T,\\
\end{cases}
\end{equation}
here, $\mathbf{W}\mathbf{L}\mathbf{W}^{-1}=\mathbf{W}\mathbf{B}\mathbf{W}^{-1}\mathsf{L}=:\mathbf{B}'\mathsf{L}$,
and $\mathbf{B}'=\{ b'_{ij}\}_{K\times K}$ with
$$
b'_{ij}=\frac{i^m}{j^m}b_{ij}.
$$
The difference between $\mathbf{W}\mathbf{L}\mathbf{W}^{-1}$ and $\mathbf{L}$ lies in the fact that $\mathbf{B}'$ is not a symmetric matrix and $\mathbf{W}\mathbf{L}\mathbf{W}^{-1}$ is not a self-adjoint operator. Let $\rho(\mathbf{L})$ be the resolvent set of $\mathbf{L}$. Considering the resolvent operator of $\mathbf{L}$, we obtain that for any $\lambda \in \rho(\mathbf{L})$, $(\mathbf{L}-\lambda \mathbf{I})^{-1}$ is bounded and $(\mathbf{W}\mathbf{L}\mathbf{W}^{-1}-\lambda \mathbf{I})^{-1}=\mathbf{W}(\mathbf{L}-\lambda \mathbf{I})^{-1}\mathbf{W}^{-1}$ is also bounded. Thus $\mathbf{W}\mathbf{L}\mathbf{W}^{-1}$ and $\mathbf{L}$ share the same spectrum.

To ensure that the boundedness of $(\mathbf{W}\mathbf{L}\mathbf{W}^{-1}-\lambda \mathbf{I})^{-1}$ is independent of $K$, we still need to guarantee that the explicit spectral-gap estimate for $\mathbf{W}\mathbf{L}\mathbf{W}^{-1}$ is also independent of $K$.
Note that inequality \eqref{fangda} is key to proving the explicit spectral-gap estimate for $\mathsf{L}$, and this relies on the property of diagonal dominance of $\mathbf{B}$. For $1\le k \le K-1$, we have 
$$
|b'_{k,k+1}|+|b'_{k+1,k}|\le(\frac{k}{k+1})^m|b_{k,k+1}|+(\frac{k+1}{k})^m|b_{k+1,k}|\le(2^m+1)|b_1|,
$$
and \eqref{fangda} still holds for the operator $\mathbf{W}\mathbf{L}\mathbf{W}^{-1}$. Therefore the explicit spectral-gap estimate for the operator $\mathbf{W}\mathbf{L}\mathbf{W}^{-1}$ still exists and it is independent of $K$.
The coefficient matrix $\mathbf{B}'$  is similar to $\mathbf{B}$, so they share the same eigenvalues. Therefore, the solution operator of the system  \eqref{PCSGL3} satisfies the same estimate as $\mathbf{G}^t$.
By applying the result of Lemma \ref{PCSGTL} 
to solve the system \eqref{PCSGL3}, we obtain the following corollary for the weighted estimate:
\begin{corollary}\label{PCSGTL1}
For $\beta> 3/2$, $\mathbf{G}^t$ satisfies that
\begin{equation*}
\begin{aligned}
\left\|\mathbf{W}\mathbf{G}^t\vec{g}_{\mathrm{in},K}\right\|_{L^{\infty}_{\xi,\beta}(L_x^2)}\le C_{G,S}\left\|\mathbf{W}\vec{g}_{\mathrm{in},K}\right\|_{L^{\infty}_{\xi,\beta}(L_x^{2})},
\end{aligned}
\end{equation*}
\begin{equation*}
\begin{aligned}
\left\|\mathbf{W}\mathbf{G}^t\vec{g}_{\mathrm{in},K}\right\|_{L^{\infty}_{\xi,\beta}(L_x^{\infty})}\le C_{G,S}(1+t)^{-\frac34}(\left\|\mathbf{W}\vec{g}_{\mathrm{in},K}\right\|_{L^{\infty}_{\xi,\beta}(L_x^\infty)}+\left\|\mathbf{W}\vec{g}_{\mathrm{in},K}\right\|_{L^{\infty}_{\xi,\beta}(L_x^{2})}).
\end{aligned}
\end{equation*}
 Moreover, if $\mathsf{P}_0\vec{g}_{\mathrm{in},K}=\vec{0}$, then we will get extra $(1+t)^{-1/2}$ decay rate for each estimate above.
\end{corollary}

%\subsection{Estimate of the gPC solution}
\subsection{Estimate of the gPC solution}
For the nonlinear system (\ref{PCSG}), following the idea in  Section 5.2 of \cite{liu2018hypocoercivity}, we perform weighted estimates.
First, we focus on the nonlinear term. Recall that the collision kernel is linear in $z$ and let 
\begin{equation}\label{1Skij}
\begin{aligned}
S_{kij}:=\int_{I_z}(b_0+zb_1)\psi_{k}(z)\psi_{i}(z)\psi_{j}(z)\pi(z)dz.
\end{aligned}
\end{equation}
By $b_0=1$, $|b_1|\le\gamma|b_0|$, $|z|\le C_z$ and \eqref{psik}, we have
\begin{equation}\label{Skij1}
\begin{aligned}
|S_{kij}|&\le(1+\gamma C_z)\|\psi_i\|_{L^\infty_z}\bigg| \int_{I_z}|\psi_k||\psi_j|\pi(z)dz\bigg|\\
&\le(1+\gamma C_z)\|\psi_i\|_{L^\infty_z}\sqrt{\bigg| \int_{I_z}|\psi_k|^2\pi(z)dz\bigg|\bigg| \int_{I_z}|\psi_j|^2\pi(z)dz\bigg|}\le\overline{C}_1i^n,
\end{aligned}
\end{equation}
where $\overline{C}_1=C(1+\gamma C_z)$ with $C$ given in \eqref{psik}. Now we assume that $\psi_k$ is a $(k-1)$th degree polynomial,
orthogonal to all lower-order polynomials, and 
\begin{equation}\label{Skij2}
S_{kij}=0,\ \text{if}\ (i-1)+(j-1)+1<k-1.
\end{equation}
 So $S_{kij}$ may be nonzero only when the inequality
\begin{equation*}
\begin{aligned}
i+j\ge k
\end{aligned}
\end{equation*}
holds. Note that \eqref{Skij1} and \eqref{Skij2} also hold if $k,i,j$ are permuted, that is, when the
inequalities
\begin{equation}\label{Skij7}
i+j\ge k,\ k+i\ge j \ \ \text{and}\ \ k+j\ge i,
\end{equation}
are satisfied, $S_{kij}$ may be nonzero. Let $\chi_{kij}$ be the indicator function of the set of indexes $(k,i,j)$ for which $S_{kij}\ne 0$, namely,
\begin{equation}\label{Skij8}
\chi_{kij}=
\begin{cases}
1, \ \  \  \ \ S_{kij}\ne 0,\\
0, \ \  \  \ \ S_{kij}= 0.
\end{cases}
\end{equation}

Next, we consider the case where $j\ge i$. By \eqref{Skij1} and \eqref{Skij2}, $(\frac{k}{2})^m|S_{kij}|i^{m-n}\le\overline{C}_1i^mj^m$, and we have
\begin{equation}\label{Skij3}
\begin{aligned}
\frac{k^{m}}{i^mj^m}|S_{kij}|\le\overline{C}i^{n-m},
\end{aligned}
\end{equation}
here, $\overline{C}=2^m\overline{C}_1$ independent of $K$. Similarly, for the case where $j\le i$, we have 
\begin{equation}\label{Skij6}
\begin{aligned}
\frac{k^{m}}{i^mj^m}|S_{kij}|\le\overline{C}j^{n-m}.
\end{aligned}
\end{equation}

Define operator $\vec{\Gamma}_K$ by
$$
\vec{\Gamma}_K(h^K,u^K)=(\Gamma_{1}(h^K,u^K),\cdots,\Gamma_{k}(h^K,u^K),\cdots,\Gamma_{K}(h^K,u^K))^T,
 $$
where, $h^K$ and $u^K$ are given functions, and satisfy that
$$
h^K(x,t,\xi,z)=\sum_{k=1}^Kh_{k}(x,t,\xi)\psi_{k}(z), \ \text{and} \ u^K(x,t,\xi,z)=\sum_{k=1}^Ku_{k}(x,t,\xi)\psi_{k}(z). 
$$
Here $\Gamma_{k}(h^K,u^K)$ is the quadratic form associated with \eqref{a4} for $1\le k\le K$.
Similar to Lemma \ref{Gamma}, we have the following lemma for the weighted estimates of $\vec{\Gamma}_K$:
\begin{lemma}\label{Gamma1}
Let $\vec{h}_K$ and $\vec{u}_K$ be the vector-valued function corresponding to $h^K$ and $u^K$, then $\vec{\Gamma}_K(h^K,u^K)$ satisfies that
\begin{equation*}
\begin{aligned}
\left\|\mathbf{W}\vec{\Gamma}_K(h^K,u^K)\right\|^2_{L^{\infty}_{\xi,\beta}(L_x^2)}&\lesssim\left\|\mathbf{W}\vec{h}_K(t)\right\|^2_{L^{\infty}_{\xi,\beta+1}(L_x^2)}\left\|\mathbf{W}\vec{u}_K(t)\right\|^2_{L^{\infty}_{\xi,\beta+1}(L^\infty_x)},\\
\left\|\mathbf{W}\vec{\Gamma}_K(h^K,u^K)\right\|^2_{L^{\infty}_{\xi,\beta}(L_x^1)}&\lesssim\left\|\mathbf{W}\vec{h}_K(t)\right\|^2_{L^{\infty}_{\xi,\beta+1}(L_x^2)}\left\|\mathbf{W}\vec{u}_K(t)\right\|^2_{L^{\infty}_{\xi,\beta+1}(L_x^2)},\\
\left\|\mathbf{W}\vec{\Gamma}_K(h^K,u^K)\right\|^2_{L^{\infty}_{\xi,\beta}(L_x^\infty)}&\lesssim\left\|\mathbf{W}\vec{h}_K(t)\right\|^2_{L^{\infty}_{\xi,\beta+1}(L^\infty_x)}\left\|\mathbf{W}\vec{u}_K(t)\right\|^2_{L^{\infty}_{\xi,\beta+1}(L_x^\infty)}.
\end{aligned}
\end{equation*}
\end{lemma} 
\begin{proof}
Consider the weighted ${L^{\infty}_{\xi,\beta}(L_x^2)}$ estimates for $\vec{\Gamma}_K(h^K,u^K)$, and by \eqref{a4}, \eqref{a6}, \eqref{1Skij} and Minkowski's  inequality, it follows
\begin{equation}\label{Skij4}
\begin{aligned}
\left\|\mathbf{W}\vec{\Gamma}_K(h^K,u^K)\right\|^2_{L^{\infty}_{\xi,\beta}(L_x^2)}&=\sum_{k=1}^K\left\|k^m\Gamma_{k}(h^K,u^K) \right\|^2_{L^{\infty}_{\xi,\beta}(L_x^2)}=\sum_{k=1}^K\left\|k^m\sum_{i,j=1}^KS_{kij}\chi_{kij}\Gamma(h_i,u_j) \right\|^2_{L^{\infty}_{\xi,\beta}(L_x^2)}\\
&\le\sum_{k,i,j=1}^Kk^{2m}|S_{kij}|^2\chi_{kij}\left\|\Gamma(h_i,u_j) \right\|^2_{L^{\infty}_{\xi,\beta}(L_x^2)}.
\end{aligned}
\end{equation}
By Lemma \ref{Gamma}, \eqref{Skij3}, \eqref{Skij6} and \eqref{Skij4}, it follows 
\begin{equation}\label{Skij5}
\begin{aligned}
&\ \ \sum_{k,i,j=1}^Kk^{2m}|S_{kij}|^2\chi_{kij}\left\|\Gamma(h_i,u_j) \right\|^2_{L^{\infty}_{\xi,\beta}(L_x^2)}\\
&\lesssim\sum_{k,i,j=1}^Kk^{2m}|S_{kij}|^2\chi_{kij}\left\|h_i \right\|^2_{L^{\infty}_{\xi,\beta+1}(L_x^2)}\left\|u_j \right\|^2_{L^{\infty}_{\xi,\beta+1}(L_x^\infty)}\\
&\lesssim\sum_{k,i,j=1}^K\frac{k^{2m}}{i^{2m}j^{2m}}|S_{kij}|^2\chi_{kij}\left\|i^mh_i \right\|^2_{L^{\infty}_{\xi,\beta+1}(L_x^2)}\left\|j^mu_j \right\|^2_{L^{\infty}_{\xi,\beta+1}(L_x^\infty)}\\
&\lesssim\sum_{k,i,j=1}^Ki^{n-m}\chi_{kij}\left\|i^mh_i \right\|^2_{L^{\infty}_{\xi,\beta+1}(L_x^2)}\left\|j^mu_j \right\|^2_{L^{\infty}_{\xi,\beta+1}(L_x^\infty)}\\
&\ \ \ +\sum_{k,i,j=1}^Kj^{n-m}\chi_{kij}\left\|i^mh_i \right\|^2_{L^{\infty}_{\xi,\beta+1}(L_x^2)}\left\|j^mu_j \right\|^2_{L^{\infty}_{\xi,\beta+1}(L_x^\infty)}.\\
\end{aligned}
\end{equation}
Next one expects that 
\begin{equation}\label{Skij9}
\begin{aligned}
\sum_{k,i,j=1}^Ki^{n-m}\chi_{kij}\left\|i^mh_i \right\|^2_{L^{\infty}_{\xi,\beta+1}(L_x^2)}\left\|j^mu_j \right\|^2_{L^{\infty}_{\xi,\beta+1}(L_x^\infty)}\le 2(\sum_{j=1}^K\left\|j^mh_j \right\|^2_{L^{\infty}_{\xi,\beta+1}(L_x^\infty)})(\sum_{i=1}^K\left\|i^mu_i \right\|^2_{L^{\infty}_{\xi,\beta+1}(L_x^2)}).
\end{aligned}
\end{equation}
To get \eqref {Skij9}, one writes 
$$
\sum_{k,i,j=1}^Ki^{n-m}\chi_{kij}\left\|i^mh_i \right\|^2_{L^{\infty}_{\xi,\beta+1}(L_x^2)}\left\|j^mu_j \right\|^2_{L^{\infty}_{\xi,\beta+1}(L_x^\infty)}=\sum_{j=1}^K\left\|j^mu_j \right\|^2_{L^{\infty}_{\xi,\beta+1}(L_x^\infty)}I_j,
$$
and
$$
I_j=\sum_{k,i=1}^Ki^{n-m}\chi_{kij}\left\|i^mh_i \right\|^2_{L^{\infty}_{\xi,\beta+1}(L_x^2)}.
$$
By \eqref {Skij7} and \eqref {Skij8}, $\chi_{kij}=1$ indicates that $j-i\le k \le j+i$, so $k$ has at most $2i$ choices for a fixed $i$. That is, 
$$
I_j\le 2\sum_{i=1}^K i^{n-m+1}\left\|i^mh_i \right\|^2_{L^{\infty}_{\xi,\beta+1}(L_x^2)}\le 2\sum_{i=1}^K\left\|i^mh_i \right\|^2_{L^{\infty}_{\xi,\beta+1}(L_x^2)}.
$$
Similarly, we have 
\begin{equation}\label{Skij10}
\begin{aligned}
\sum_{k,i,j=1}^Kj^{n-m}\chi_{kij}\left\|i^mh_i \right\|^2_{L^{\infty}_{\xi,\beta+1}(L_x^2)}\left\|j^mu_j \right\|^2_{L^{\infty}_{\xi,\beta+1}(L_x^\infty)}\le 2(\sum_{j=1}^K\left\|j^mu_j \right\|^2_{L^{\infty}_{\xi,\beta+1}(L_x^\infty)})(\sum_{i=1}^K\left\|i^mh_i \right\|^2_{L^{\infty}_{\xi,\beta+1}(L_x^2)}).
\end{aligned}
\end{equation}
Combining \eqref{Skij4}-\eqref{Skij10}, we have 
\begin{equation*}\label{Skij}
\begin{aligned}
\left\|\mathbf{W}\vec{\Gamma}_K(h^K,u^K)\right\|^2_{L^{\infty}_{\xi,\beta}(L_x^2)}&\lesssim(\sum_{j=1}^K\left\|j^mu_j \right\|^2_{L^{\infty}_{\xi,\beta+1}(L_x^\infty)})(\sum_{i=1}^K\left\|i^mh_i \right\|^2_{L^{\infty}_{\xi,\beta+1}(L_x^2)})\\
&\lesssim\left\|\mathbf{W}\vec{u}_K(t)\right\|^2_{L^{\infty}_{\xi,\beta+1}(L^\infty_x)}\left\|\mathbf{W}\vec{h}_K(t)\right\|^2_{L^{\infty}_{\xi,\beta+1}(L_x^2)}.
\end{aligned}
\end{equation*}
Similarly, we can obtain the weighted ${L^{\infty}_{\xi,\beta}(L_x^1)}$ and ${L^{\infty}_{\xi,\beta}(L_x^\infty)}$ estimates for $\vec{\Gamma}_K(h^K,u^K)$, and we omit it here. We have completed the proof.
\end{proof}

We state that $\mathbf{G}^t$ satisfies similar properties as those in Lemma \ref{GTS}.
Through Lemma \ref{PCSGTL} and a proof similar to Lemma \ref{INLf}, we can solve the system (\ref{PCSG}) of the Boltzmann equation, and obtain the following lemma:
\begin{lemma}\label{PCSGT} 
Let $f^K$ be the solution of (\ref{PCSG}) with the initial data $f_0$. $\vec{f}_K$ is the vector-valued function corresponding to $f^K$ and $\vec{f}_{\mathrm{in},K}$ is the  initial data corresponding to $f_0$. There exist constants $\delta,C_S>0$ independent of $K$ such that if 
$$
\left\|\mathbf{W}\vec{f}_{\mathrm{in},K}\right\|_{L^{\infty}_{\xi,\beta}(L_x^\infty)}+\left\|\mathbf{W}\vec{f}_{\mathrm{in},K}\right\|_{L^{\infty}_{\xi,\beta}(L_x^{1})}<\delta, \ \beta > 3/2,
$$
 then 
\begin{equation*}
\begin{aligned}
\left\|\mathbf{W}\vec{f}_K(t)\right\|_{L^{\infty}_{\xi,\beta}(L_x^2)}\le C_S(1+t)^{-\frac34}(\left\|\mathbf{W}\vec{f}_{\mathrm{in},K}\right\|_{L^{\infty}_{\xi,\beta}(L_x^\infty)}+\left\|\mathbf{W}\vec{f}_{\mathrm{in},K}\right\|_{L^{\infty}_{\xi,\beta}(L_x^{1})}),
\end{aligned}
\end{equation*}
\begin{equation*}
\begin{aligned}
\left\|\mathbf{W}\vec{f}_K(t)\right\|_{L^{\infty}_{\xi,\beta}(L_x^{\infty})}\le C_S(1+t)^{-\frac32}(\left\|\mathbf{W}\vec{f}_{\mathrm{in},K}\right\|_{L^{\infty}_{\xi,\beta}(L_x^\infty)}+\left\|\mathbf{W}\vec{f}_{\mathrm{in},K}\right\|_{L^{\infty}_{\xi,\beta}(L_x^{1})}).
\end{aligned}
\end{equation*}

 Moreover, if $\mathsf{P}_0\vec{f}_{\mathrm{in}}=\vec{0}$, then we will get extra $(1+t)^{-1/2}$ decay rate for each estimate above.
\end{lemma} 

%\subsection{Estimate of the gPC error}
\subsection{Estimate of the gPC error}
In this subsection, we discuss the estimate of the gPC error to prove Theorem \ref{ESG}. The total gPC error is defined by 
\begin{equation*}
\begin{aligned}
f^e:=f-f^K=f-P_Kf+P_Kf-f^K,
\end{aligned}
\end{equation*}
and define the projection error $R^K=f-P_Kf$, the numerical error $e^K=P_Kf-f^K$.

By Theorem \ref{T1} and the standard estimate on the projection error in \cite{g7,g8}, for a fixed $\alpha \in \mathbb{N}_+$, $\delta$ and $C_\alpha$ are those in Theorem \ref{T1}, if $f_0(x,t,\xi,z)$ satisfies that
$$
\sum_{s=0}^\alpha\sum_{k=1}^\infty\left\|k^m\int_{I_z}\psi_k(z)\pi(z)\partial^s_zf_0dz\right\|_{L^{\infty}_{\xi,\beta}(L_x^{1}\cap L_x^{\infty})}<\delta,
$$
then, $R^K$ satisfies that
\begin{equation}\label{pc1}
\begin{aligned}
\sum_{k=1}^\infty\left\|k^m\int_{I_z}\psi_k(z)\pi(z)R^Kdz\right\|_{L^{\infty}_{\xi,\beta}(L_x^{2})}&\\
\le \frac{ C_\alpha(\ln(1+t))^{\alpha-1}}{K^\alpha(1+t)^{\frac34}}&\sum_{s=0}^\alpha\sum_{k=1}^\infty\left\|k^m\int_{I_z}\psi_k(z)\pi(z)\partial^s_zf_0dz\right\|_{L^{\infty}_{\xi,\beta}(L_x^{1}\cap L_x^{\infty})},\\
\sum_{k=1}^\infty\left\|k^m\int_{I_z}\psi_k(z)\pi(z)R^Kdz\right\|_{L^{\infty}_{\xi,\beta}(L_x^{\infty})}& \\
\le \frac{ C_\alpha(\ln(1+t))^{\alpha-1}}{K^\alpha(1+t)^{\frac32}}&\sum_{s=0}^\alpha\sum_{k=1}^\infty\left\|k^m\int_{I_z}\psi_k(z)\pi(z)\partial^s_zf_0dz\right\|_{L^{\infty}_{\xi,\beta}(L_x^{1}\cap L_x^{\infty})},
\end{aligned}
\end{equation}
which means that increasing the number $K$ of terms in the expansion can lead to a decrease in the projection error $R^K$. In order to simplify the notation, we define $\varepsilon_\alpha(f_0)$ independent of $K$ by
$$
\varepsilon_\alpha(f_0):=\sum_{s=0}^\alpha\sum_{k=1}^\infty\left\|k^m\int_{I_z}\psi_k(z)\pi(z)\partial^s_zf_0dz\right\|_{L^{\infty}_{\xi,\beta}(L_x^{1}\cap L_x^{\infty})}.
$$

From (\ref{PK}), $e^K$ satisfies that
\begin{equation*}
\begin{aligned}
e^K=P_Kf-f^K= \sum_{k=1}^{K}(\widehat{f}_{k}(x,t,\xi)-f_k(x,t,\xi))\psi_{k}(z)=: \sum_{k=1}^{K}e_k(x,t,\xi)\psi_{k}(z),
\end{aligned}
\end{equation*}
where one defines the coefficients of $e^K$ by 
\begin{equation*}
\begin{aligned}
e_k=\widehat{f}_{k}-f_k, \ 1\le k\le K.
\end{aligned}
\end{equation*}
Through $P_Kf$, we also define a vector-valued function $\vec{P}_K(x,t,\xi)$ by
$$
\vec{P}_K(x,t,\xi)=(\widehat{f}_1,\cdots,\widehat{f}_k,\cdots,\widehat{f}_K)^T,
$$
and it satisfies that
\begin{equation}\label{PKf}
 \begin{cases}
\partial_t\vec{P}_{K}(x,t,\xi)+\xi\cdot \nabla_x\vec{P}_{K}(x,t,\xi)=\mathbf{L}(\vec{P}_{K}(x,t,\xi))+\vec{R}_K+\vec{\Gamma}_K(f,f),\\\
\vec{P}_{K}(x,t,\xi)=\vec{f}_{\mathrm{in},K}(x,\xi).\\
\end{cases}
\end{equation}
Here, 
$\vec{R}_K$ satisfies that
$$
\vec{R}_K=(r_1,\cdots,r_k,\cdots,r_K)^T,
$$
where $r_k=0, 1\le k\le K-1$, $r_K=b_{K,K+1}\mathsf{L}\widehat{f}_{K+1}$, and $|b_{K,K+1}|\le b_1$. Moreover, $\widehat{f}_{K+1}$ satisfies the estimate \eqref{pc1} for $R^K$, and it follows that
\begin{equation}\label{pc3}
\begin{aligned}
\left\|(K+1)^m\widehat{f}_{K+1}(t)\right\|_{L^\infty_{\xi,\beta}(L_x^2)}&\le \frac{ C_\alpha(\ln(1+t))^{\alpha-1}}{K^\alpha(1+t)^{\frac34}}\varepsilon_\alpha(f_0),\\
\left\|(K+1)^m\widehat{f}_{K+1}(t)\right\|_{L^\infty_{\xi,\beta}(L_x^\infty)}& \le \frac{ C_\alpha(\ln(1+t))^{\alpha-1}}{K^\alpha(1+t)^{\frac32}}\varepsilon_\alpha(f_0).
\end{aligned}
\end{equation}
%It is worth noting that by definition \eqref{a4}, for $1\le k\le K$, we have
%$$
%\Gamma_k(f,f)=\sum_{i,j=1}^{\infty}S_{kij}\Gamma(\widehat{f}_{i},\widehat{f}_{j}).
%$$

Let $\vec{E}_K$ be the vector-valued function corresponding to $e^K$. Recall that $e^K=P_Kf-f^K$, then by comparing the differences between system \eqref{PCSG} and \eqref{PKf}, $\vec{E}_{K}$ satisfies
\begin{equation}\label{Ek}
 \begin{cases}
\partial_t\vec{E}_{K}+\xi\cdot \nabla_x\vec{E}_{K}=\mathbf{L}\vec{E}_{K}+\vec{R}_K+\vec{\Gamma}_K(f,f)-\vec{\Gamma}_K(f^K,f^K),\\
\vec{E}_{K}(x,0,\xi)=\vec{0}. \\
\end{cases}
\end{equation}
Treating $\vec{R}_K$ and $\vec{\Gamma}_K(f,f)-\vec{\Gamma}_K(f^K,f^K)$ as source terms, we apply Duhamel's principle to system \eqref{Ek}, and it follows that
$$
\vec{E}_K=\int_0^t\mathbf{G}^{t-\tau}\vec{R}_K(\tau)d\tau+\int_0^t\mathbf{G}^{t-\tau}(\vec{\Gamma}_K(f,f)-\vec{\Gamma}_K(f^K,f^K))(\tau)d\tau.
$$
 By Corollary \ref{PCSGTL1}, (\ref{pc3}) and $\mathsf{P}_0\vec{R}_K=0$, the first term satisfies
\begin{equation}\label{GRK}
\begin{aligned}
\left\|\mathbf{W}\int_0^t\mathbf{G}^{t-\tau}\vec{R}_K(\tau)d\tau\right\|_{L^\infty_{\xi,\beta}(L_x^2)}&\le \frac{ C_{G,S}C_\alpha(\ln(1+t))^{\alpha-1}}{K^\alpha(1+t)^{\frac12}}\varepsilon_\alpha(f_0),\\
\left\|\mathbf{W}\int_0^t\mathbf{G}^{t-\tau}\vec{R}_K(\tau)d\tau\right\|_{L^\infty_{\xi,\beta}(L_x^\infty)}& \le \frac{ C_{G,S}C_\alpha(\ln(1+t))^{\alpha-1}}{K^\alpha(1+t)^{\frac54}}\varepsilon_\alpha(f_0).
\end{aligned}
\end{equation}
Here, compared to the results of Lemma \ref{PCSGTL}, the slower time decay is due to the availability of only the $L^2_x$ estimate for $\vec{R}_K$.
So motivated by \eqref{GRK}, we make the following assumption for $\vec{E}_{K}$:
\begin{equation}\label{pc2}
\begin{aligned}
\left\|\mathbf{W}\vec{E}_{K}(t)\right\|_{L^\infty_{\xi,\beta}(L_x^2)}&\le \frac{ 10C_{G,S}C_\alpha(\ln(1+t))^{\alpha-1}}{K^\alpha(1+t)^{\frac12}}\varepsilon_\alpha(f_0),\\
\left\|\mathbf{W}\vec{E}_{K}(t)\right\|_{L^\infty_{\xi,\beta}(L_x^\infty)}& \le \frac{ 10C_{G,S}C_\alpha(\ln(1+t))^{\alpha-1}}{K^\alpha(1+t)^{\frac54}}\varepsilon_\alpha(f_0).
\end{aligned}
\end{equation}

Note that
\begin{equation*}
\begin{aligned}
\vec{\Gamma}_{K}(f,f)-\vec{\Gamma}_{K}(f^K,f^K)&=\vec{\Gamma}_{K}(f,f-f^K)+\vec{\Gamma}_{K}(f-f^K,f^K),\\
\end{aligned}
\end{equation*}
and by Lemma \ref{Gamma1}, its weighted $L^\infty_{\xi,\beta-1}(L_x^{2})$ estimate satisfies
\begin{equation}\label{pc8}
\begin{aligned}
&\ \ \ \left\|\mathbf{W}\vec{\Gamma}_{K}(f,f)-\mathbf{W}\vec{\Gamma}_{K}(f^K,f^K)\right\|^2_{L^\infty_{\xi,\beta-1}(L_x^{2})}\\
&\lesssim\left\|\mathbf{W}\vec{\Gamma}_{K}(f,f-f^K)\right\|^2_{L^\infty_{\xi,\beta-1}(L_x^{2})}+\left\|\mathbf{W}\vec{\Gamma}_{K}(f-f^K,f^K)\right\|^2_{L^\infty_{\xi,\beta-1}(L_x^{2})}\\
&\lesssim(\left\|\mathbf{W}\vec{E}_K\right\|^2_{L^\infty_{\xi,\beta}(L_x^{\infty})}+\sum_{k=1}^\infty\left\|k^m\int_{I_z}\psi_k(z)\pi(z)R^Kdz\right\|^2_{L^{\infty}_{\xi,\beta}(L_x^{\infty})})\\
&\ \ \ \times \bigg(\sum_{k=1}^\infty\left\|k^m\int_{I_z}\psi_k(z)\pi(z)fdz\right\|^2_{L^{\infty}_{\xi,\beta}(L_x^{2})}+\left\|\mathbf{W}\vec{f}_K\right\|^2_{L^\infty_{\xi,\beta}(L_x^{2})}\bigg).\\
\end{aligned}
\end{equation}
Clearly, the estimates of both $f$ and $\vec{f}_K$ have been established in Theorem \ref{T1} and Lemma \ref{PCSGT}. So by Theorem \ref{T1}, Lemma \ref{PCSGT}, \eqref{pc1} and assumption \eqref{pc2}, it follows from \eqref{pc8} that
\begin{equation}\label{pc4}
\begin{aligned}
&\ \ \ \left\|\mathbf{W}\vec{\Gamma}_{K}(f,f)-\mathbf{W}\vec{\Gamma}_{K}(f^K,f^K)\right\|_{L^\infty_{\xi,\beta-1}(L_x^{2})}\\
&\lesssim\bigg(\frac{(\ln(1+t))^{\alpha-1}}{K^\alpha(1+t)^{\frac54}}\varepsilon_\alpha(f_0)+(\frac{K}{K+1})^m\frac{(\ln(1+t))^{\alpha-1}}{K^\alpha(1+t)^{\frac54}}\varepsilon_\alpha(f_0)\bigg)\bigg(\frac{1}{(1+t)^\frac34}\varepsilon_\alpha(f_0)+\frac{1}{(1+t)^\frac34}\varepsilon_\alpha(f_0)\bigg)\\
&\lesssim\frac{(\ln(1+t))^{\alpha-1}}{K^\alpha(1+t)^{2}}(\varepsilon_\alpha(f_0))^2.
\end{aligned}
\end{equation}
Similarly, the weighted $L^\infty_{\xi,\beta}(L_x^{1})$ and $L^\infty_{\xi,\beta}(L_x^{\infty})$ estimates satisfy that
\begin{equation}\label{pc5}
\begin{aligned}
&\ \ \ \left\|\mathbf{W}\vec{\Gamma}_{K}(f,f)-\mathbf{W}\vec{\Gamma}_{K}(f^K,f^K)\right\|_{L^\infty_{\xi,\beta-1}(L_x^{1})}\lesssim\frac{(\ln(1+t))^{\alpha-1}}{K^\alpha(1+t)^{\frac54}}(\varepsilon_\alpha(f_0))^2,\\
\end{aligned}
\end{equation}
and
\begin{equation*}
\begin{aligned}
&\ \ \ \left\|\mathbf{W}\vec{\Gamma}_{K}(f,f)-\mathbf{W}\vec{\Gamma}_{K}(f^K,f^K)\right\|_{L^\infty_{\xi,\beta-1}(L_x^{\infty})}\lesssim\frac{(\ln(1+t))^{\alpha-1}}{K^\alpha(1+t)^{\frac{11}{4}}}(\varepsilon_\alpha(f_0))^2.
\end{aligned}
\end{equation*}

For the estimates of the remaining part, it is similar to \eqref{4.11}, and by Theorem \ref{T1}, Lemma \ref{PCSGT} and \eqref{pc2}, it follows
\begin{equation*}
\begin{aligned}
&\ \ \ \left\| \mathbf{W}\int_0^t\mathbf{G}^{t-\tau}(\vec{\Gamma}_K(f,f)-\vec{\Gamma}_K(f^K,f^K))(\tau)d\tau \right\|_{L^\infty_{\xi,\beta}(L_x^{2})}\\
&\lesssim\int_0^t \frac{1}{(1+t-\tau)^{\frac34+\frac12}}\left\|\mathbf{W}\vec{\Gamma}_{K}(f,f)(\tau)-\mathbf{W}\vec{\Gamma}_{K}(f^K,f^K)(\tau)\right\|_{L^\infty_{\xi,\beta-1}(L_x^{1})}d\tau\\
&\ \ \ + \int_0^t \frac{1}{(1+t-\tau)^{\frac34+\frac12}}\left\| \mathbf{W}\vec{\Gamma}_{K}(f,f)(\tau)-\mathbf{W}\vec{\Gamma}_{K}(f^K,f^K)(\tau)\right\|_{L^\infty_{\xi,\beta-1}(L_x^{2})}d\tau.\\
\end{aligned}
\end{equation*}
On the one hand, by \eqref{pc5}, we have 
\begin{equation}\label{pc6}
\begin{aligned}
 &\ \ \ \int_0^t \frac{1}{(1+t-\tau)^{\frac34+\frac12}}\left\|\mathbf{W}\vec{\Gamma}_{K}(f,f)(\tau)-\mathbf{W}\vec{\Gamma}_{K}(f^K,f^K)(\tau)\right\|_{L^\infty_{\xi,\beta-1}(L_x^{1})}d\tau\\
&\lesssim\int_0^t \frac{1}{(1+t-\tau)^{\frac34+\frac12}}\frac{ (\ln(1+\tau))^{\alpha-1}}{K^\alpha(1+\tau)^{\frac54}}d\tau(\varepsilon_\alpha(f_0))^2\\
&\lesssim\frac{ (\ln(1+t))^{\alpha-1}}{K^{\alpha}(1+t)^{\frac54}}(\varepsilon_\alpha(f_0))^2.
\end{aligned}
\end{equation}
On the other hand, by \eqref{pc4}, we have 
\begin{equation}\label{pc7}
\begin{aligned}
 &\ \ \ \int_0^t \frac{1}{(1+t-\tau)^{\frac34+\frac12}}\left\|\mathbf{W}\vec{\Gamma}_{K}(f,f)(\tau)-\mathbf{W}\vec{\Gamma}_{K}(f^K,f^K)(\tau)\right\|_{L^\infty_{\xi,\beta-1}(L_x^{2})}d\tau\\
&\lesssim\int_0^t \frac{1}{(1+t-\tau)^{\frac34+\frac12}}\frac{ (\ln(1+\tau))^{\alpha-1}}{K^\alpha(1+\tau)^{2}}d\tau(\varepsilon_\alpha(f_0))^2\\
&\lesssim\frac{ (\ln(1+t))^{\alpha-1}}{K^{\alpha}(1+t)^{\frac54}}(\varepsilon_\alpha(f_0))^2.
\end{aligned}
\end{equation}
Combining \eqref{GRK}, \eqref{pc6} and \eqref{pc7}, the $L^\infty_{\xi,\beta}(L_x^{2})$ estimate for $\vec{E}_K$ satisfies ansatz \eqref{pc2}. Similarly, for the $L^\infty_{\xi,\beta}(L_x^{\infty})$ estimate, we still can prove it satisfies the assumption \eqref{pc2}  and the details are omitted here. Therefore, we close ansatz \eqref{pc2}.

Finally, we give the estimate for $f^e$. Since $m-n>1$, for the $X_{\xi;\beta}^{x,2}$ estimate of $f^e$, we have 
\begin{equation*}
\begin{aligned}
\left\|f^e\right\|_{X_{\xi;\beta}^{x,2}}\le\left\|R^K\right\|_{X_{\xi;\beta}^{x,2}}+\left\|e^K\right\|_{X_{\xi;\beta}^{x,2}}.
\end{aligned}
\end{equation*}
By \eqref{pc1}, it follows
\begin{equation}\label{fepk}
\begin{aligned}
\left\|R^K\right\|_{X_{\xi;\beta}^{x,2}}&=\left\|\sum_{k=K+1}^\infty \widehat{f}_{k}\psi_k\right\|_{X_{\xi;\beta}^{x,2}}\le  \sum_{k=K+1}^\infty\| k^n\widehat{f}_{k}\|_{L^\infty_{\xi,\beta}(L_x^{2})}\lesssim (\sum_{k=K+1}^\infty\|  k^m\widehat{f}_{k}\|_{L^\infty_{\xi,\beta}(L_x^{2})})(\sum_{k=K+1}^\infty k^{n-m})\\
&\lesssim\sum_{k=1}^\infty\left\|k^m\int_{I_z}\psi_k(z)\pi(z)R^Kdz\right\|_{L^{\infty}_{\xi,\beta}(L_x^{2})}\lesssim \frac{ (\ln(1+t))^{\alpha-1}}{K^\alpha(1+t)^{\frac12}}\varepsilon_\alpha(f_0).
\end{aligned}
\end{equation}
By \eqref{pc2}, it follows
\begin{equation}\label{feek}
\begin{aligned}
\left\|e^K\right\|_{X_{\xi;\beta}^{x,2}}&=\left\|\sum_{k=1}^K e_k\psi_k\right\|_{X_{\xi;\beta}^{x,2}}\le  \sum_{k=1}^K\| k^ne_k\|_{L^\infty_{\xi,\beta}(L_x^{2})}\lesssim (\sum_{k=1}^K\|  k^me_k\|_{L^\infty_{\xi,\beta}(L_x^{2})})(\sum_{k=1}^Kk^{n-m})\\
&\lesssim\left\|\mathbf{W}\vec{E}_{K}(t)\right\|_{L^\infty_{\xi,\beta}(L_x^2)}\lesssim \frac{ (\ln(1+t))^{\alpha-1}}{K^\alpha(1+t)^{\frac12}}\varepsilon_\alpha(f_0).
\end{aligned}
\end{equation}
Combining Lemma \ref{PCSGT}, \eqref{fepk}, \eqref{feek} and the similar estimates for $\left\|f^e\right\|_{X_{\xi;\beta}^{x,\infty}}$, we have completed the proof of Theorem \ref{ESG}.

%
%
%
%
%
%
%
%
%
%
%
%
%
%
%
%
%
%
%
%
%
%
%%%%%%%%%%%%%%%%%%%%%%%%%%%%%%%%%%%%%%%%%%%%%%%%%%%%%%%%%%%%%%%%%%%%%%%%%%%%%%%%%%%%%%%%%%%%%%%%%%%%%%%%%%%%%%%%%%%%%%%%%%%
%\appendix

%%%%%%%%%%%%%%%%%%%%%%%%%%%%%%%%%%%%%%%%%%%%%%%%%%%%%%%%%%%%%%%%%%%%%%%%%%%%%%%%%%%%%%%%%%%%%%%%%%%%%%%%%%%%%%%%%%%%%%%%%%

\section*{Acknowledgment} This work is partially supported by the National Key R\&D Program of
China under grants 2022YFA1007300, 2024YFA1013302 and 21Z010300242. H.-T. Wang is supported by NSFC
under Grant No. 12371220 and 12031013. S. Jin was supported by NSFC grants No. 12031013 and 12350710181.

\end{document}